\documentclass{amsart}
\usepackage[utf8]{inputenc}
\usepackage[T1]{fontenc}
\usepackage[latin9]{luainputenc}
\usepackage{geometry}

\usepackage{bbm}
\usepackage{babel}
\usepackage{amsmath}
\usepackage{amsthm}
\usepackage{amssymb}
\usepackage{comment}
\usepackage[unicode=true,pdfusetitle,
 bookmarks=true,bookmarksnumbered=false,bookmarksopen=false,
 breaklinks=false,pdfborder={0 0 1},backref=false,colorlinks=false]
 {hyperref}
\usepackage{dsfont,graphicx}
\usepackage[dvipsnames]{xcolor}

\usepackage{esint}

\makeatletter
\numberwithin{equation}{section}
\numberwithin{figure}{section}
\theoremstyle{plain}
\newtheorem{thm}{\protect\theoremname}[section]
\theoremstyle{plain}
\newtheorem{prop}[thm]{\protect\propositionname}
\theoremstyle{plain}
\newtheorem{cor}[thm]{\protect\corname}
\theoremstyle{plain}
\newtheorem{lem}[thm]{\protect\lemmaname}
\theoremstyle{definition}
\newtheorem{defn}[thm]{\protect\definitionname}
\theoremstyle{remark}
\newtheorem{rem}[thm]{\protect\remarkname}
\theoremstyle{plain}
\newtheorem{exa}[thm]{\protect\examplename}

\@ifundefined{date}{}{\date{}}

\def\vep{\varepsilon}

\def\R{\mathbb{R}}
\def\cH{\mathbf{H}}
\def\intr{\int_{\R^d}}
\def\grad{\nabla}

\DeclareMathOperator*{\argmin}{arg\,min}

\makeatother

\providecommand{\corname}{Corollary}
\providecommand{\lemmaname}{Lemma}
\providecommand{\propositionname}{Proposition}
\providecommand{\remarkname}{Remark}
\providecommand{\theoremname}{Theorem}
\providecommand{\definitionname}{Definition}
\providecommand{\examplename}{Example}
\title[Regularity of the free boundary for the supercooled Stefan problem]{Regularity of the free boundary for the supercooled Stefan problem in arbitrary dimensions}
\author{Max Engelstein, Inwon Kim, Sebastian Munoz}
\date{}
 \address{ME: School of Mathematics, University of Minnesota, Minneapolis, MN, 55455, USA}
 \address{IK, SM: Department of Mathematics, University of California, Los Angeles, CA,
90095, USA}
 
	\email{mengelst@umn.edu}
	\email{ikim@math.ucla.edu}
	 \email{sebastian@math.ucla.edu}

\begin{document}
\maketitle
\begin{abstract}
    We study the free boundary in the supercooled Stefan problem, a classical model for the solidification of water below its freezing temperature. In contrast with the melting problem, physical experiments and heuristics indicate that the water--ice interface in the supercooled problem may exhibit fractal freezing sets, infinite-speed propagation of the frozen front, and nucleation (the spontaneous appearance of ice). Despite this, we show that the free boundary has a robust structure. 
    
    We decompose the free boundary into three parts: (1) a regular part that advances with finite speed in time; (2) a singular part consisting of points where the front attains infinite speed or nucleates, but with controlled space-time (i.e., $\leq d-1$ parabolic) dimension; and (3) a jump component, which can have large dimension in a time slice, but which is contained in a space-time smooth graph and occurs only at a zero-dimensional set of times. Examples show that each of these parts can be nonempty. 

    Furthermore, we prove that the free boundary is the graph $t=s(x)$ of a continuously differentiable freezing time $s$, and the singular set coincides with the critical set of $s$, proving that singularities in supercooled freezing always occur with infinite speed. 

    These results provide the first free boundary regularity theory for the supercooled Stefan problem in arbitrary dimensions. 
\end{abstract}

\tableofcontents

\section{Introduction}\label{s:intro}

We study the free boundary of solutions to the supercooled Stefan problem
\begin{equation} \label{eq:eta intro}
\eta_t-\Delta \eta =(\chi_{\{\eta>0\}})_t \; \;\text{in } \;\;\Omega  \times (0,\infty),
\end{equation}
where $\Omega \subset \R^d$, $-\eta\leq0$ represents the temperature of the supercooled water, and $\chi_{\{\eta>0\}}$ represents the change in latent heat when water freezes into ice (see \cite{Rub}). We denote the \emph{transition zone} of points that freeze in finite positive time by
\begin{equation*}
    \Gamma:=\{x\in \Omega: \inf\{t:\eta(x,t)=0\}\in (0,\infty)\}.
\end{equation*}
Then the transformation $w(x,t):=\int_{t}^{\infty}\eta(x,s)ds$ converts the problem into the weighted parabolic obstacle problem $w_t -\Delta w=-\chi_{\{w>0\}}\chi_{\Gamma}$. In particular, if $x_0$ is an interior point of $\Gamma$ that freezes at time $t_0$, then, near the free boundary point $(x_0,t_0)$, the problem is locally equivalent to
\begin{equation} \label{eq:obstacle intro}
    \begin{cases}w_t -\Delta w=-\chi_{\{w>0\}}\\
    w\geq 0\\
    w_t\leq 0, \;\{w>0\}=\{w_t<0\}.         
    \end{cases}
\end{equation}

The goal of this paper is to study the regularity of the free boundary $\partial\{w>0\}$ for solutions to \eqref{eq:obstacle intro}. This amounts to studying the free boundary of solutions to \eqref{eq:eta intro} near interior points of the transition zone $\Gamma$. In particular, if $\Gamma$ is an open set, then our results apply to all free boundary points. 

Although \eqref{eq:eta intro} is a classical and well-studied model, there are surprisingly few results on the regularity of the free boundary. This is due to the ill-posedness of \eqref{eq:eta intro}; indeed the existence of weak solutions to \eqref{eq:eta intro} for general initial data was only proven recently \cite{CKK}. Physical heuristics suggest that there exist solutions of the supercooled Stefan problem with highly irregular liquid--ice interfaces. This includes jump propagation in one dimension \cite{She70,HV96,DT19}, and much richer patterns in higher dimensions \cite{DF84,HOL85,HMV00}. On the other hand, for certain classes of weak solutions in one space dimension, the third named author showed that the free boundary's jump times are highly constrained \cite{Mun25}. This suggests that, while arbitrary weak solutions may develop pathological behavior, the singularities in the free boundary of certain special weak solutions may be better controlled. 

Together with its companion paper \cite{CKM25}, this work provides further evidence for this perspective. In particular, we develop the first free boundary regularity result for a physically meaningful class of weak solutions to \eqref{eq:eta intro} that holds in arbitrary dimensions. More precisely, \cite[Thm. 1.4]{CKM25} shows that the main existence result available in the literature, \cite[Thm. 1.2]{CKK}, produces {\it maximal} solutions for which the transition zone $\Gamma$ is open modulo a low-dimensional set\footnote{The set $\Gamma \setminus \text{int}(\Gamma)$ is $(d-2)$--dimensional in the Hausdorff sense.}. For such solutions, the results  of the present paper  yield a thorough description of the free boundary. 
On the other hand, \cite[Thm. 1.3]{CKM25} shows that for general initial data one can construct non-maximal solutions for which $\Gamma$ exhibits
genuinely pathological behavior, including fractal freezing and interior waiting times.
 Maximal solutions may be interpreted as solutions that, in an averaged sense, delay solidification as much as possible. Our results therefore support the heuristic that this bias against solidification promotes additional free boundary regularity relative to arbitrary weak solutions of \eqref{eq:eta intro}. Finally, we note that the statements proved here also apply to the \emph{global} solutions constructed in \cite[Thm. 9.3]{KK}, for which $\Gamma=\R^d$ (see the discussion after Theorem \ref{thm:sigma}).

We now turn to the model \eqref{eq:obstacle intro}. Classical results on the parabolic obstacle problem (see, e.g., \cite{CaPeSh}) show that, outside of a closed set $\Sigma$ of \emph{singular} points, the free boundary of any solution to \eqref{eq:obstacle intro} is a $C^{\infty}$ space-time surface, and that surface moves with finite positive speed, leading to $(d-1)$--dimensional time slices. The next step in understanding the free boundary is, therefore, to study the nature of the singular
points. Without any assumptions on the sign of $w_t$, a first step in this analysis was completed by \cite{Bla06,monneau}, who stratified the singular points by their infinitesimal symmetries and gave dimension bounds describing how the different strata are situated in space. However, their methods did not discuss the question of how the singular points are distributed in space-time. 

This question remained largely open until the recent breakthrough work of Figalli, Ros-Oton and Serra \cite{figalli}. In the case of the \emph{melting problem}, corresponding to $w_t>0$, \cite{figalli} showed that the \emph{parabolic} Hausdorff dimension of the singular set satisfies the sharp estimate $\dim_{\operatorname{par}}(\Sigma)\leq d-1$. Such a bound is false in the supercooled case. Indeed, the simple example
\begin{equation} \label{eq:-t example intro}
    w(x,t)=\max(-t,0), \quad (x,t)\in \R^d \times (-\infty,\infty),
\end{equation}
which readily solves \eqref{eq:obstacle intro}, shows that $\Sigma$ may, in general, have time slices of spatial dimension $d$ that appear with infinite speed. Our first result provides the correct sharp estimate for the parabolic Hausdorff dimension\footnote{Recall the parabolic Hausdorff dimension is simply the Hausdorff dimension with respect to the parabolic metric, $d_{\text{par}}((x,t), (y,s)):= |x-y| + |t-s|^{1/2}$. To illustrate this, consider a set $E \subset \mathbb R^{d+1}$ of the form $E = E_x\times E_t$ where $E_x$ is a subset of $t= 0$ and $\dim_H E_x = d_1$ and $E_t$ is a subset of $x= 0$ and $\dim_H E_t = d_2$. Then $\dim_{\text{par}} E = d_1 + 2d_2$.} of the singular set for the supercooled Stefan problem. 
\begin{thm}\label{thm: dim par} Let $\Omega \subset \R^d$ be an open set, let $T>0$, let $w$ be a bounded solution to \eqref{eq:obstacle intro} on $\Omega \times (0,T)$, and let $\Sigma \subset \R^d \times (0,T)$ be the set of singular free boundary points of $w$ (i.e. points around which $\partial \{w > 0\}$ cannot be written as a smooth hypersurface advancing smoothly in time). Then
\begin{equation} \label{eq:dim par intro}
    \dim_{\operatorname{par}}(\Sigma)\leq d.
\end{equation}  
Moreover, if $\eta:=-w_t<1$ , then $\dim_{\operatorname{par}}(\Sigma)\leq d-1$.
\end{thm}
Equality in \eqref{eq:dim par intro} is attained for any time translation of the solution in \eqref{eq:-t example intro} (see also Example \ref{ex:tychonoff}). The second statement reflects the previously observed fact that the value $1$ is the critical supercooling  threshold for $\eta$ to exhibit a discontinuous jump (see \cite{Pel,CaCaRo} regarding the physical role of this critical parameter).
In particular, if $d=2$ and $\eta<1$, the singular set is empty outside of a $(1/2)$--dimensional set of times.

Our next result concerns the space-time regularity of the free boundary. In the classical melting problem, it is known that one may locally write $\partial \{w>0\}=\{t=s(x)\}$, where the melting time $s$ is Lipschitz continuous but, in general, not $C^1$ (see \cite{Ca78,FioLatRos})\footnote{Note that this expression of the free boundary as a graph is qualitatively different from the local $C^\infty$--graphicality of $\partial \{w > 0\}\backslash \Sigma$ discussed above. See Remark \ref{r:graphs} for more discussion.}. A prototypical example of this would be a singular point $(x_0,t_0)\in \Sigma$ arising as the collision of two independent fronts that are melting at different (finite) speeds, with melting times $s_1$ and $s_2$. Formally the overall melting time is then $\min(s_1(x),s_2(x))$, so $s$ is at best Lipschitz.

However, recent one-dimensional results in \cite{Mun25} suggest that the stronger claim $s\in C^1$ does hold for the supercooled problem. We prove here that this is the case in arbitrary dimensions.
\begin{thm}\label{thm:reg intro} Let $w$ be a bounded solution to \eqref{eq:obstacle intro} on $B_1 \times (-1,1)$, and assume that $(0,0)\in \partial\{w>0\} $. Then there exist $r_0\in (0,1)$ and $s:B_{r_0}\to (-1,1)$ such that \begin{equation*}
  (B_{r_0}\times(-1,1)) \cap \{w>0\}=\{t<s(x)\}, \quad  s\in C^1(B_{r_0}),\end{equation*}
and the singular set may be characterized as
\begin{equation} \label{eq:sigma char s}
    \Sigma \cap (B_{r_0}\times (-1,1))=\{(x,s(x)): x\in B_{r_0}, \; \grad s(x)=0\}.    
\end{equation}
In particular, $\partial \{w>0\}$ is a space-time $C^1$ hypersurface.

\end{thm}
This result provides a precise description of the singular set in terms of the critical points of the \emph{freezing time} $s$. Since the quantity $|\grad s(x)|^{-1}$ is the speed of propagation of the interface at the point $(x,s(x))$, the improved regularity of $s$ compared with the melting problem, paired with \eqref{eq:sigma char s}, may be understood as a rigorous statement of the fact that singularities in supercooled freezing always occur with infinite speed. 
One may verify (e.g. from Proposition \ref{p:sigma2example}) that the freezing time is, in general, not $C^{2,\alpha}$ regular for any $\alpha>0$. It is, however, an interesting open question whether the result of Theorem \ref{thm:reg intro}  can be upgraded, for instance, to $C^{1,\alpha}$ regularity. Such a result would give a finer description of the free boundary as its speed approaches infinity.

The theory of the parabolic obstacle problem \cite{CaPeSh,Bla06} shows that every singular point $(x_0,t_0)$ has a unique blow-up profile of the form
\begin{equation} \label{eq:blow-up intro}
    \lim_{r\downarrow 0}\frac{w(x_0+rx,t_0+r^2t)}{r^2}=-mt+\frac12Ax\cdot x, \quad m\in [0,1], \;A:=A_{x_0,t_0}\geq 0, \quad \operatorname{tr}(A)=1-m,
\end{equation}
where $(x,t)\in \R^d \times (-\infty,0].$ The singular set $\Sigma$ can then be stratified as
\begin{equation} \label{eq:stratification}
    \Sigma=\bigcup_{k=1}^{d}\Sigma_{k}, \quad \Sigma_k=\{(x_0,t_0)\in \Sigma: \dim(\ker (A_{x_0,t_0}))=k\}.
\end{equation}
As we show in the proof of Theorem \ref{thm: dim par}, one actually has the finer estimate for each stratum,
\begin{equation} \label{eq:dim par strata}
    \dim_{\operatorname{par}}(\Sigma_k)\leq k. 
\end{equation}
The top stratum $\Sigma_d$ therefore corresponds to the case $A=0,\;m=1$, which has as a blow-up profile precisely the example \eqref{eq:-t example intro}. In particular, a key feature that distinguishes the supercooled case from the classical melting problem is the fact that it may have time-dependent blow-up profiles in all strata and a nonempty top stratum $\Sigma_d$. These singularities model the possibility of nucleation and advancing ice ``jumping''.  By \emph{nucleation} we mean the appearance of ice `out of nowhere'; see Definition \ref{def:nucleation} for a precise formulation. We refer to ``jumps'' as the phenomenon where at a given time the liquid--ice interface has nonempty interior, considered as a subset of $\mathbb R^d$ (i.e. $\mathrm{int}(\Sigma\cap \{t= t_0\}) \neq \emptyset$).

It is then natural to ask whether a finer analysis of the singular set is possible by studying the structure of $\Sigma_d$, which is the stratum where the bound provided by \eqref{eq:dim par strata} is largest. Indeed, in light of \eqref{eq:dim par strata} if $\Sigma\cap \{t = t_0\}$ has interior or develops fractal structure with dimension $> d-1$, it must be due a preponderance of points in $\Sigma_d$. This is the content of the next result.

\begin{thm}\label{thm:sigma} Let $w$ be a bounded solution to \eqref{eq:obstacle intro} on $\Omega \times (0,T)$, and let $\Sigma$ be the set of singular points of $w$. There exists a closed set $\Sigma_{d}^{\infty}\subset \Sigma$ such that
\begin{equation} \label{eq:dimpar finite freq}
    \dim_{\operatorname{par}}(\Sigma \backslash \Sigma_{d}^{\infty})\leq d-1,  
\end{equation}
\begin{equation}\label{eq:dim jump times}
    \dim_{\mathcal{H}}(\{t\in(0,T): (x,t)\in \Sigma_{d}^{\infty} \;\;\;\text{for some }\; x\in \Omega\})=0.
\end{equation}
Additionally, the set $\Sigma_{d}^{\infty} \subset \Omega \times (0,T)$ is contained in the graph $\{t=h(x)\}$ of a smooth function $h\in C^{\infty}(\Omega)$. Furthermore, if $\Omega=\R^d$, then one has the dichotomy
\begin{equation}\label{dichotomy}\Sigma_{d}^{\infty}=\emptyset \quad \text{ or } \quad w\equiv (t_0-t)^+ \quad \text{for some}\quad  t_0\in [0,T].
\end{equation}
\end{thm}
This theorem shows that the singular set may be decomposed into a lower-dimensional set in which jump or fractal time-slice phenomena cannot occur, and a set  $\Sigma_{d}^{\infty}$ that is rare in time and contained in a space-time smooth hypersurface (for the corresponding structure theorem for the melting case, see \cite[Thm. 1.3]{figalli}). Note that, in particular, it follows from \eqref{eq:dimpar finite freq}--\eqref{eq:dim jump times} that the free boundary $\partial\{w(t)>0\}$ is $(d-1)$-dimensional outside of a Hausdorff dimension zero set of times $t$.

Moreover, in the case of \emph{global} solutions to \eqref{eq:obstacle intro}, such fractals do not develop, and jumps cannot occur except in the fully degenerate case when the entire liquid freezes instantaneously. We emphasize that this dichotomy only holds for global solutions:  Example \ref{ex:tychonoff} constructs a local solution in $B_1\times (-1,1)$, where $w$ is positive and space-dependent in $B_1\times (-1,0)$, and yet features total freezing at $t=0$, namely $B_1\times \{0\} \subset \Sigma_d^{\infty}.$  Global solutions were constructed in \cite[Thm. 9.3]{KK} (see also \cite{DF84}), under well-prepared initial data. It is nevertheless striking that in this setting one can completely rule out partial jumps and fractal interfaces: despite similarities between \eqref{eq:eta intro} and  diffusion‑limited aggregation (DLA) \cite{NSZ24}, whose dynamics typically generate fractal boundaries, such behavior does not occur for these global solutions. 

From a technical standpoint, our results require several new ideas. In particular, we are not aware of previous work that performs ``dimension reduction'' on singularities where the first blow-up is space-independent. We expect these techniques to be useful in other parabolic free boundary problems where space-independent solutions arise, such as the time-like singularities in the parabolic thin obstacle problem \cite{DanGarPetTo}.

\subsection{Main ideas of the proofs}
 While our analysis builds on previous studies of the singular set for the melting Stefan problem (e.g. \cite{figalli}, \cite{colombo}, \cite{FioLatRos}), the supercooled case contains unique challenges, such as the discontinuity of the temperature $w_t=-\eta$, the potential jumps of the free boundary over time, and, as discussed earlier, the presence of time-dependent blow-up profiles in every stratum of the singular set. We write
\begin{equation*}
    \Sigma=\Sigma_{\operatorname{stat}}\cup \Sigma_{\operatorname{dyn}},
\end{equation*}
where $\Sigma_{\operatorname{stat}}$ and $\Sigma_{\operatorname{dyn}}$ are, respectively, the singular points with stationary and time-dependent blow-up profiles. The latter set $\Sigma_{\operatorname{dyn}}$, corresponding to $m\neq 0$ in \eqref{eq:blow-up intro}, is ruled out immediately by the condition $w_t > 0$ in the melting problem. However, in the supercooled setting it is the set of singular points with potentially the largest dimension and also contains the singular points where the free boundary recedes fastest, including potential jumps.  Its detailed analysis is one of the main novelties of this paper.

We turn to the proof of our first main result, Theorem \ref{thm: dim par}. Work of \cite{monneau} (see also Theorem \ref{thm:monneau} below), says that the spatial projection of the $k$th singular strata has dimension $\leq k$. A simple GMT lemma, inspired by \cite[Lemma 8.9]{figalli} (see also their earlier work in elliptic setting \cite{FRS20}), says that if at each point in $\Sigma_k$ there exists a backwards in time parabola which does not intersect any other point of $\Sigma_k$, then the estimate on the projections of $\Sigma_k$ can be upgraded to $\dim_{\text{par}}(\Sigma_k) \leq k$. The existence of said parabola is almost trivial for points in $\Sigma_{\text{dyn}}$ and we are able to follow work of \cite{colombo} to prove it for points in $\Sigma_{\text{stat}}\backslash \Sigma_{d-1}$ (though executing this roadmap requires new technical estimates in the supercooled case, as some of the estimates in the melting case do not hold, see Section \ref{sec:estimates} and especially Section \ref{subsec:nondegeneracy estimates}). 

Thus, to complete our proof of Theorem \ref{thm: dim par}, we only need to study $\Sigma_{\text{stat}}\cap \Sigma_{d-1}$. Here our approach diverges substantially from \cite{figalli, colombo}; we prove that, in the supercooled setting, $\Sigma_{\text{stat}}\cap \Sigma_{d-1} = \emptyset$. We do this in Section \ref{ss:topstationarystrata}. Roughly, by first establishing an {\it a priori} algebraic condition on the second blowup at such points (if they existed), via a method inspired by \cite{FRS20, FiSe19}. However, we then show that these second blowups must satisfy a  quantitative non-degeneracy condition (see Lemma \ref{l:upperqtbound}) inherited from $w_t < 0$. These two conditions contradict one another, ruling out the existence of the points altogether. 

We emphasize that this is not just an algebraic coincidence; our analysis shows that points in $\Sigma_{d-1}\cap \Sigma_{\text{stat}}$ are the only singularities which might form with finite speed. Ruling them out is thus a physically significant result and also a key point in proving Theorem \ref{thm:reg intro}.

Turning to the proof of Theorem \ref{thm:reg intro}, the global $C^1$ regularity of $\partial \{w > 0\}$, recall that we write the free boundary $\partial \{w > 0\} = (x, s(x))$ and that $s(x) \in C^\infty$ away from $\Sigma$. So to complete the proof we must show that if $(x,s(x)) \in \Sigma$ we have $\nabla s(x) = 0$. Implicit differentiation suggests that this should be trivial in $\Sigma_{\text{dyn}}$. However,  it is a pathological feature of the supercooled problem that $w_t$ may not be continuous. Indeed, we prove that $w_t$ is discontinuous at all points in $\Sigma_{\text{dyn}}$ (see Proposition \ref{prop: eta continuous}). To prove regularity at time-dependent singular points, we first obtain lower bound on the speed of the free boundary at all singular points (Lemma \ref{lem:eta grad w}), and then upgrade it infinite speed at time-dependent singular points. 

For the stationary singular points, having ruled out points in $\Sigma_{d-1}$, we are able to show that $\partial_t w$ quantitatively controls $\nabla w$ near points in $\Sigma_{\text{stat}}$ (see Lemma \ref{lem: eta nondegeneracy space blow-ups}). This is a barrier argument which seems technically different from, but related to, recent interesting work establishing first-order regularity of the free boundary near lower dimensional strata in the melting problem, \cite[Theorem 1.4]{FioLatRos}. 

Finally, we turn to the proof of our last main result, Theorem \ref{thm:sigma}. To do this we must show that we can split $\Sigma_d$ (the top stratum) into two components $\Sigma_d^{\leq \infty}$ which has projections that are $\leq d-1$ dimensional and $\Sigma_d^\infty$ which is empty for all but a zero dimensional set of times. Our key tool is an analysis of the second blowup. While this is highly reminiscent of the statement and proof strategy of \cite[Theorem 1.3]{figalli}, we emphasize that the similarities are mostly superficial. Indeed, our proofs of both the dimension drop of $\Sigma_d^{\leq \infty}$ and the time regularity of $\Sigma_d^\infty$ contain substantially new ideas.

The first new idea is to categorize the points in the top stratum according to their second blow-up. In the melting case, it is known that the second blow-up of the solution is always a function of parabolic degree $2$ or $3$. By contrast, in our case the second blow-up at a point of $\Sigma_d$ can have arbitrarily large degree (in fact, the second blow-up for the solution \eqref{eq:-t example intro} at the singular point $(0,0)$ is $q\equiv 0$, corresponding to ``infinite degree''). That $w - (-t)^+$ can vanish both quadratically and to higher integer order was first observed in a one-dimensional setting by Herrero and Vel\'azquez \cite[Theorem 1]{HV96}. Although the details of our analysis shares little in common with theirs, we were inspired by this observation. 

In this vein, we further decompose the top stratum as
\begin{equation*}
    \Sigma_d = \Sigma_d^{\infty}\cup\bigcup_{k=2}^{\infty}\Sigma_d^k,
\end{equation*}
where $\Sigma_{d}^k$ is the set of singular points in $\Sigma_d$ where the second blow-up has degree $k$ (we refer to Section \ref{sec:top stratum quadratic} for the precise definitions). The set $\Sigma_{d}^{\infty}$ is the points at which $w-(-t)^+$ vanishes to infinite order. We emphasize again that this is a fundamentally different decomposition than the one in \cite{figalli} (where the ``rare" set comes from the set of points at which two different polynomial expansions agree to infinite order). 

In our setting, it is straightforward to show that $\Sigma_d^\infty$ is contained in a  space-time smooth manifold and is rare in time. The main difficulty is therefore to show that the high-dimensional pathologies of supercooled freezing are entirely contained in $\Sigma_{d}^{\infty}$. To this end, we need to deal with  two distinct cases; the set of quadratic second blow-ups versus the higher order second blow-ups.   While quadratic functions are the simplest  profiles, the main difficulty with these points is that the very slow rate of convergence (slower than any power) makes it difficult to extract information about the solution (or to even prove that the second blow-up is unique). Quadratic second blow-ups are known to exist in both the elliptic case and the melting case (\cite[App. A]{FiSe19}, \cite[Lem. 5.8]{figalli}), but they typically belong to a lower-dimensional stratum, so sharp dimension bounds for the singular set can be obtained without analyzing the second blow-up of these points. Indeed, \cite[Corollary 4.14]{colombo} shows that quadratic second blow-ups never appear in the top stratum in the melting problem.  In our case, however, quadratic second blow-ups can occur in the top stratum (we present a family of examples in Proposition \ref{p:sigma2example}), and thus we must treat them in detail for our purposes. As an aside, it is conjectured in \cite{HV96} that $\Sigma_d^2\neq \emptyset$ for an open set of initial data. We do not attempt to pursue this direction here but think it is an interesting question for future study. 

We show that, even if the second blow-up profile is not unique and the convergence is slow, the solution at points of $\Sigma_d^2$ inherits, in a very weak  qualitative sense, some of the ``concave'' behavior of its possible subsequential limits (Lemma \ref{lem: sign rigidity}).  We leverage this mild information, together with a two-sided superquadratic cleaning property that is unique to the supercooled problem in its top stratum (Lemma \ref{lem:two-sided cleaning}), to obtain a suitable porosity property for the set $\Sigma_d^2$ that enables the necessary dimension drop (Proposition \ref{prop:-t dimension bound quadratic}).

We then turn to the set $\Sigma_d^k$ for $k \geq 3$ (points where the second blow-up is a degree $k$ caloric polynomial). The main challenge here relative to the quadratic case is that the second blow-up may be a polynomial of arbitrarily large degree, with much more complex behavior. The advantage, however, is that the convergence is much faster, which leads to significantly more regularity. Indeed, by exploiting the fast convergence and the everywhere positivity of the blow-up profile $p_2(x,t)=-t$ for $t<0$, we approximate the solution with caloric polynomials via Campanato iteration, and prove that $w$ admits a smooth, one-sided Taylor expansion at every point of $\Sigma_d^{\geq 3}$ (see Proposition \ref{prop:taylor}). The construction of this expansion has similarities to \cite[Sec. 13]{figalli}, where two different one-sided Taylor expansion are exhibited at (almost) every point of $\Sigma_{d-1}$. 

However, our use of the Taylor expansion is completely different. Without discussing all the technicalities, we view the $k$th order Taylor series expansion at each point as a polynomial in $t$. We then show that the zero set of the (spatially dependent) leading coefficient\footnote{Unless this coefficient is constant, in which case we analyze the penultimate term.} locally well-approximates the nearby singular points, which yields the desired dimension reduction. This is a highly delicate argument using  Reifenberg-type theories for sets well-approximated by the zero sets of harmonic polynomials, as well as the sharp \L ojasiewicz inequality for harmonic polynomials \cite[Thm. 3.1]{BadEngTor}.

\subsection{Outline of the paper} In Section \ref{sec:preliminaries}, we present preliminary results from the literature on the parabolic obstacle problem and the melting problem. Section \ref{sec:estimates} contains the foundational a priori estimates for the solution and the free boundary. In particular, the estimates of Section \ref{subsec:nondegeneracy estimates} are critical to study the precise behavior of the free boundary near points of $\Sigma_{\operatorname{dyn}}$. 
Section \ref{sec:stationary} studies the set $\Sigma_{\operatorname{stat}}$ of singular points with a stationary blow-up profile, with the main highlights being that the bottom and top stationary strata, $\Sigma_{\operatorname{stat}}\cap \Sigma_0$ and $\Sigma_{\operatorname{stat}}\cap \Sigma_{d-1}$, are both shown to be empty, and suitable dimension bounds for the intermediate strata are obtained. We also discuss the physically relevant notion of nucleation, which is a unique feature of the supercooled problem. We describe the profile of the solution at nucleation points, and we precisely characterize the set where the temperature $w_t$ is discontinuous (see Propositions \ref{prop: eta continuous} and \ref{prop:nucleation}). In Section \ref{sec:stationary} we also introduce the second blow-up in our setting and rule out the top stationary strata as described above.


The remainder of the paper is mostly focused on the set $\Sigma_{\operatorname{dyn}}$. In Section \ref{sec:dynamic}, we obtain the basic parabolic dimension estimates for these points, and we take advantage of the fundamental estimates of Section \ref{sec:estimates} and the previous analysis of $\Sigma_{\operatorname{stat}}$ to prove Theorems \ref{thm: dim par} and \ref{thm:reg intro}.

For the following sections, we further restrict our attention to the top stratum of the singular set, namely the set $\Sigma_d \subset \Sigma_{\operatorname{dyn}}$ of points with blow-up profile $p_2(x,t)=-t$. Recalling from above that $\Sigma_d^\infty$ are those points which, roughly, coincide with the blow-up $(-t)^+$ to infinite order, the bulk of Sections \ref{sec:top stratum quadratic} and \ref{sec:top stratum superquadratic} is dedicated to proving the estimate
\begin{equation*}
    \dim_{\operatorname{par}}(\Sigma_d \setminus\Sigma_d^{\infty})\leq d-1.
\end{equation*}

In Section \ref{sec:top stratum quadratic}, we estimate the set $\Sigma_d^2$ of points with quadratic second blow-up. Next, in Section \ref{sec:top stratum superquadratic}, we study the set $\Sigma_{d}^{\geq 3}:=\Sigma_d^{\infty}\cup\bigcup_{k=3}^{\infty}\Sigma_d^k$ of points with at least cubic second blow-up. 

In Section \ref{sec:top stratum infinite}, we show through a frequency argument that, in the case of global solutions, the set $\Sigma_d^{\infty}$ is always empty unless $w$ is identical to the blow-up profile \eqref{eq:-t example intro}, and we then complete the proof of Theorem \ref{thm:sigma}.  

Finally, in Section \ref{sec:global}, we show by example that in bounded domains one can have $\Sigma_d^\infty \neq \emptyset$ even if $w\neq (t_0-t)^+$. We also show that the global radial solutions constructed in \cite[Thm. 9.3]{KK} have the property that $\Sigma_d^2\neq \emptyset$. Finally, we present a simple example (via a gluing construction) showing that $\Sigma_d^2$ can have accumulation points in space and time.

\section{Preliminary facts for the supercooled Stefan problem}
\label{sec:preliminaries}
\subsection*{Notation} We define, for $r>0$, the parabolic cylinders
\begin{equation*}
    Q_{r}^-:=B_r \times (-r^2,0], \quad Q_r:=B_r \times (-r^2,r^2),
\end{equation*}
and the parabolic boundary
\begin{equation*}
    \partial_p Q^-_r= (\partial B_r\times [-r^2,0] )\cup (B_r \times\{-r^2\}).
\end{equation*}
The heat operator and the parabolic homogeneity operator are denoted by
\begin{equation*}
   \cH:=\Delta- \partial_t, \quad  Z:= x \cdot \grad +2t\partial_t,
\end{equation*}
and the backward heat kernel is given by
\begin{equation*}
    G(x,t):=\frac{1}{(-4\pi t)^{d/2}}\exp\left(\frac{|x|^2}{4t}\right), \quad (x,t)\in \R^d \times (-\infty,0).
\end{equation*}
For $k\geq 1$, $r\in (0,1)$, and functions $f,g:\R^d \times (-1,0) \to \R^k$,  we define
\begin{equation*}
    \langle f,g\rangle_{r}:= \intr (f\cdot g)(x,-r^2)G(x,-r^2)dx,
\end{equation*}
and we write $\langle\cdot,\cdot\rangle:=\langle \cdot, \cdot \rangle_1$. When $k=1$ and $f$ is sufficiently regular, we also write
\begin{equation}\label{e:handd}
   H(r,f):=\langle f, f \rangle_r = \int_{\{t=-r^2\}} f^2 G, \quad \quad  D(r,f):=2r^2\langle \grad f, \grad f\rangle_r=2r^2 \int_{\{t=-r^2\}}|\nabla f|^2G. 
\end{equation}
For $\gamma>0$, the parabolic frequency functions are defined by
\begin{equation}\label{e:frequency}
    \phi(r,f):=\frac{D(r,f)}{H(r,f)}, \quad \phi^{\gamma}(r,f):=\frac{D(r,f)+\gamma r^{\gamma}}{H(r,f)+r^{\gamma}}.
\end{equation}
We will denote by $\zeta$ a fixed spatial cut-off function:
\begin{equation}\label{eq:cutoff defi}
    \zeta\in C^{\infty}_c(B_{1/2}), \quad 0\leq \zeta \leq 1, \quad \zeta \equiv 1 \;\;\text{ in }\; B_{1/4}.
\end{equation}
We also define the parabolic rescalings
\begin{equation*}
 f_r(x,t):=f(rx,r^2t). 
\end{equation*}

\subsection{The parabolic obstacle problem}
We begin by presenting the basic properties of bounded solutions $w:Q_1\to \R$ to \eqref{eq:obstacle intro}, from the classical theory of the parabolic obstacle problem. We also summarize at the end the main differences from what holds for the melting Stefan problem.

By \cite[Sec. 4]{CaPeSh}, we have $w\in C^{1,1}_x \cap C_t^0$ and
\begin{equation} \label{eq:C11xCt bounds w}
    \|D^2w\|_{L^{\infty}(Q_{1/2})}+\|w_t\|_{L^{\infty}(Q_{1/2})}\leq C \|w\|_{L^{\infty}(Q_1)}
\end{equation}
for some $C=C(d)>0$. Moreover, if $(0,0)\in \overline{\{w>0\}}$ then, for every $r\in (0,1)$, one has the nondegeneracy property
\begin{equation} \label{eq:w nondegeneracy}
    \sup_{B_r}w(\cdot,-r^2) \geq w(0,0)+ c_d r^2,
\end{equation}
where $c_d\in (0,1)$ is a dimensional constant, see \cite[Lem 5.1]{CaPeSh} (they bound from below the supremum in the entire cylinder, $Q_r^-$, but \eqref{eq:w nondegeneracy} follows given our assumption $w_t \leq 0$).

Given $(x_0,t_0)\in \partial\{w>0\}\cap Q_1$, one may define the rescaled solution
\begin{equation*}
    w^{x_0,t_0,r}(x,t):=r^{-2}w(x_0+rx,t_0+r^2t)
\end{equation*}
The functions $w^{x_0,t_0,r}$ are then uniformly bounded solutions to \eqref{eq:obstacle intro}, and the free boundary points may then be classified into two categories as follows:
\begin{itemize}
    \item $(x_0,t_0)$ is a {\it regular point} if there exists $e \in \mathbb{S}^{d-1}$ and a sequence $r_k\to 0$ such that
    \begin{equation}\label{e:regularblowup}
        \lim_{k\to \infty}w^{x_0,t_0,r_k}(x,t)=\frac{1}{2}(x\cdot e)_+^2.
    \end{equation}
    \item $(x_0,t_0)$ is a {\it singular point} if there exist $m\in [0,1]$ and a non-negative $d \times d$ matrix $A$ such that $\operatorname{tr}(A)=1-m$, and 
    \begin{equation}\label{e:singularblowup}
        \lim_{r\to 0+}w^{x_0,t_0,r}(x,t)=-mt+\frac12Ax\cdot x. 
    \end{equation}    
\end{itemize}
The convergence in \eqref{e:regularblowup} and \eqref{e:singularblowup} happens uniformly on compacta and smoothly on the positivity set of the limiting function.

The next theorem is the standard classification of the free boundary points, and states that the free boundary is a smooth hypersurface moving with finite speed near the regular points (see \cite[Rem. 7.2, Thm. 14.1, Thm. 15.1]{CaPeSh}, \cite[Prop. 1.2]{Bla06}).
\begin{thm} \label{thm:obstacle prelim}
Let $w:Q_1 \to \R$ be a bounded solution to \eqref{eq:obstacle intro}. Then every free boundary point is either regular or singular. The set of regular points is relatively open, and, in some neighborhood $V$ of every regular point, $w_t$ is space-time H\"older continuous in $V$, the free boundary is space-time $C^{\infty}$ hypersurface in $V$, and there exists $C>0$ such that
\begin{equation*}
    |w_t|\leq C|\grad w|\;\;\text{ in }\;\; V. 
\end{equation*}
\end{thm}
 We emphasize that $w_t$ is continuous at regular points but, unlike the melting problem, may be discontinuous at singular points.

 \begin{rem}\label{r:graphs}
Let us clarify the different notions of graphicality and speed we use in the paper. 

We say that $x_0$ is a regular point, i.e. $x_0\in \partial \{w > 0\}\backslash \Sigma$ if around $x_0$ the free boundary can locally be written as a graph of space and time. Namely, we can write $x= (x', x_d)$ and after a rotation we obtain $Q_r(x_0)\cap \partial \{w > 0\} = \{(x,t)\mid x_d = g(x',t)\}$ for some $g\in C^\infty_{x,t}$. 

Another notion of graphicality is the freezing time. Here we use monotonicity to parameterize the entire free boundary as $\partial \{w > 0\} = (x, s(x))$. Now the ``speed'' with which $(x_0, t_0)\in \partial \{w > 0\}$ moves is given by $|\nabla s(x_0)|^{-1}$ (by the inverse function theorem).  In this parametrization, the freezing time $s(x)$ plays a similar role as the hitting time $T(x)$ used to describe the tumor interface in \cite{CoJaKi}.

We can relate the two characterizations as follows: if $t- s(x', g(x',t)) = 0$ then we have that $1- \partial_d s \partial_t g = 0$ and that $\partial_i s + \partial_d s \partial_i g = 0$. We can see then that if $\nabla s(x) = 0$ (i.e. we are advancing with infinite speed) it must be that such a $g$ cannot exist and so $(x, s(x)) \in \Sigma$. In fact, our Theorem \ref{thm:reg intro}, shows that this is a rigorous characterization. Note that this does not follow from the implicit function theorem, and in fact is false for the melting problem, see \cite{figalli}. 
 \end{rem}

The following result follows from \cite[Thm. 1.9]{monneau}, and gives stratified dimensional estimates for the spatial projection of the singular set.
\begin{thm}\label{thm:monneau} Let $w$ be a bounded solution to \eqref{eq:obstacle intro} and, for each $i\in \{1,\ldots,d\}$, let $\Sigma_i$ be the $i^{\operatorname{th}}$ stratum of the singular set as defined by \eqref{eq:stratification}. Then
\begin{equation} \label{eq:monneau cite}
    \dim_{\mathcal{H}}(\pi_x(\Sigma_i))\leq i
\end{equation}   
\end{thm}
We note that, in the case of the top stratum $\Sigma_d$, the bound \eqref{eq:monneau cite} is trivial and provides no information. 

  In summary, the main differences with the analogous properties for the classical melting problem are that: 
  \begin{itemize}\item[(i)] $w_t$ may be discontinuous, \item[(ii)] the nondegeneracy property \eqref{eq:w nondegeneracy} does not hold in the same time slice $t=0$, and, crucially, \item[(iii)] the blow-up profiles at singular points may be time-dependent, with a nonempty stratum $\Sigma_d$.
 \end{itemize} 
   \subsection{Parabolic functionals and monotonicity formulas} We recall in this section some basic algebraic identities and key monotonicity formulas obtained in \cite{figalli}. Let $\mathcal{P}$ denote the set of parabolically $2$--homogeneous polynomials $p(x,t)$ such that 
\[\cH p=1, \quad p(x,t)\geq 0, \quad (x,t)\in \R^d \times(-\infty,0] \]
We note that the elements of $\mathcal{P}$ are precisely those of the form
\[ p(x,t)=-mt+\frac{1}{2}Ax\cdot x, \quad m\in [0,1], \quad A\geq0, \quad \text{tr}(A)=1-m,\]
and can be equivalently characterized as the set of possible blow-up limits at singular points for solutions to \eqref{eq:obstacle intro}.
In particular, 
\begin{equation*}
    Zp=2p \quad \hbox{ for all } p\in \mathcal{P}. 
\end{equation*}
We remark that while most of the formulas below were stated for solutions to the melting problem, they are valid for any non-negative solution to the parabolic obstacle problem. This is because the proofs do not rely on the sign of $w_t$, but rather on the following basic identities, which hold for any $p\in \mathcal{P}$ and any non-negative function $w$ satisfying $\cH w= \chi_{\{w>0\}}$:
\begin{equation}\label{eq:good sign monotonicity formulas} (w-p)\cH (w-p)=p\chi_{\{w=0\}}\geq 0, \quad Z(w-p)\cH (w-p)=Zp\chi_{\{w=0\}}=2p\chi_{\{w=0\}}\geq 0.\end{equation}

Recalling that
\begin{equation*}
 f_r(x,t)=f(rx,r^2t),   
\end{equation*}
we have, for $f,g \in C^{1,1}_x\cap C^{0,1}_t$,
\begin{equation} \label{eq:scaling identities}
    \cH (f_r)=r^2(\cH f)_r, \quad Z(f_r)=(Zf)_r, \quad \langle f,g\rangle_r=\langle f_r,g_r\rangle
\end{equation}
\begin{equation*}\left. \frac{d}{dr}\right\vert_{r=1}\langle f,1\rangle_r=\langle Zf,1\rangle    \end{equation*}
\begin{equation} \label{eq: integration by parts}
    2\langle \grad f, \grad g\rangle = \langle Zf, g \rangle -2\langle \cH f,g\rangle=\langle f,Zg\rangle-2\langle f, \cH g\rangle.
\end{equation}
\begin{equation} \label{eq:H' comp}
    \frac{d}{dr}H(r,f)=\frac{2}{r}\langle f,Zf\rangle_r=\frac{2}{r}\left(D(r,f)+2r^2\langle f, \cH f\rangle_r \right).
\end{equation}
\begin{equation}\label{eq:D' comp}
   \frac{d}{dr}D(r,f)=\frac{1}{r}(2\langle Zf,Zf \rangle_r-4r^2\langle Zf,\cH f\rangle_{r}) 
\end{equation}
  
For the last two equations recall the definition of the height and energy functionals, $H$ and $D$, in \eqref{e:handd}. The key fact about the parabolic frequency functions, stated below, are that they are (almost-) monotone in $r$, as long as $f$ is ``approximately" caloric. Since we mostly work with local solutions, we utilize the spatial cutoff function $\zeta$ defined in \eqref{eq:cutoff defi} to obtain a function defined on all of $\R^d$.

\begin{lem}\cite[Lem. 5.3, Prop. 5.4]{figalli} \label{lem:frequency} Let $w:Q_1 \to [0,\infty)$ be a bounded solution to \eqref{eq:obstacle intro}, assume $(0,0)$ is a singular point, let $p\in \mathcal{P}$, and let $u=\zeta(w-p)$. Then there exists $C>0$  such that,
\begin{equation} \label{eq:freq geq 2}
D(r,u)-2H(r,u)\geq -Ce^{-\frac{1}{r}}, \quad r\in (0,1)
\end{equation}
and, for $\gamma>2$
\begin{equation} \label{eq:freq monotonicity}
    \frac{d}{dr}\phi^{\gamma}(r,u)\geq -Ce^{-\frac{1}{r}}, \quad r^2\langle u,\cH u\rangle_r \geq -Ce^{-\frac{1}{r}}, \quad r\in (0,1),
\end{equation} 
where $C$ depends only on dimension, $\|w\|_{\infty}$, and (only in the case of \eqref{eq:freq monotonicity}) $\gamma.$
\end{lem}

The next lemma captures the other key fact about the frequency, which is that $\phi(0+, f)$ measures the order of vanishing of $f$ at zero (again assuming $f$ is ``approximately caloric'').

\begin{lem}\cite[Lem. 5.6, Cor. 5.9]{figalli} \label{lem:freq H bd}  Let $w:Q_1 \to [0,\infty)$ be a bounded solution to \eqref{eq:obstacle intro}, assume $(0,0)$ is a singular point, let $p\in \mathcal{P}$, let $u=\zeta(w-p)$, and let $\phi^{\gamma}(0+,u)=\lambda$.  Then, for $R\in (0,1)$,
\begin{equation} \label{eq:H frequency bounds upper}
H(r,u) + r^{2\gamma} \leq C \left( \frac{r}{R} \right)^{2\lambda}(H(R,u) + R^{2\gamma}).
\end{equation}
where $C=C(\gamma,\|w\|_{L^\infty})$.  Assume further that there exists $\delta > 0$ such that $\phi^\gamma(r, u) \leq \lambda + \tfrac{\delta}{2}$ for all $r \in (0,R)$. Then 
\begin{equation} \label{eq:H frequency bounds lower}
     (H(R,u) + R^{2\gamma})\left( \frac{r}{R}  \right)^{2\lambda + \delta} \leq C_\delta (H(r,u) + r^{2\gamma}),
\end{equation}
where $C_\delta=C_{\delta}(\gamma,\|w\|_{L^\infty},\delta)>0$.

Moreover, if $\phi^{\gamma_0}(0^+,u)<\gamma_0$ for some $\gamma_0>2$, then $\phi(0+,u)$ exists and satisfies $\phi^{\gamma}(0+,u)=\lambda$ for any $\gamma \geq \gamma_0$. In that case, for every $K>1$, there exists $C_K=C_K(K,d,\gamma_0, \|w\|_{\infty})>1$ such that
\begin{equation} \label{eq: H(theta r) comparable}
    \frac{1}{C_K}\leq \frac{H(r,u)}{H(\theta r,u)}\leq C_K, \quad \theta\in [K^{-1},K], \quad r\in (0,1).
\end{equation}    
\end{lem}

For global solutions, we will prove related estimates (but without the cutoff) in Section \ref{sec:top stratum infinite}.

\section{A priori estimates and quantitative nondegeneracy of the free boundary}
\label{sec:estimates}
\subsection{Asymptotic convexity at the free boundary} It is well known that, for the melting problem (i.e. assuming $w_t\geq 0$), the function $(D^2w)^-$ vanishes continuously at the free boundary. This was proved in \cite{Caf77} by showing that $(D^2w)^-$ has the same type of logarithmic modulus of continuity as $w_t$. In our case, $w_t$ is discontinuous and the classical proof breaks down. 

Nevertheless, we see below that qualitative asymptotic convexity still holds by a blow-up argument. We observe that the argument works for any non-negative solution of $w_t - \Delta w = - \chi_{\{w > 0\}}$, i.e. no sign assumption on $w_t$ is made.

\begin{prop}\label{p:continuityofsecondd}Let $w:Q_1 \to \R$ be a bounded solution to \eqref{eq:obstacle intro}, and assume that $(0,0)\in \partial\{w>0\}$. Then one has \begin{equation*}
    \lim_{(x,t) \to (0,0),\;w(x,t)>0}(D^2w)^{-}(x,t)=0. 
\end{equation*}
In particular, $(D^2w)^{-}$ may be extended continuously to $Q_1$.    
\end{prop}
\begin{proof}
We argue by contradiction. Suppose that there exists $e\in\mathbb{S}^{d-1}$ and a sequence $(x_n,t_n)\in\{w>0\}$ such that $(x_n,t_n)\to (0,0)$ and
\[\lim_{n\to \infty} \partial_{ee} w(x_n,t_n)=\liminf_{(x,t)\to (0,0),w(x,t)>0}\partial_{ee}w(x,t)=:-c_0<0. \]
For each $n$ sufficiently large, let $r_n>0$ be the largest number such that $Q_{r_n}(x_n,t_n)\subset \{w>0\}$, and let 
\[u_n(x,t):=r_n^{-2}w(x_n+xr_n,t_n+tr_n^2).\]
 By taking a subsequence, using the $C^{1,1}_x\cap C_t^0$ estimates \eqref{eq:C11xCt bounds w} and the nondegeneracy property \eqref{eq:w nondegeneracy}, we may assume that there exists a solution $u$ to
\begin{equation}
    \begin{cases} \label{eq:obstacle gen}
        u_t-\Delta u=-\chi_{\{u>0\}} \quad \text{ in } \;\;\R^d \times(-\infty,0],\\
        u\geq 0,\;u_t\leq 0,
    \end{cases}
\end{equation}
such that $u_n \to u$ locally uniformly in $\R^d \times(-\infty,0]$.  Since $u_n >0$ in $Q_1^-$, we infer that $\cH u_n=1$ in $Q_1^-$. Thus, by parabolic regularity the convergence of $u_n$ is locally uniform in $Q_1^-$ in any $C^k$ norm (and $\cH u = 1$ in $Q_1^-$). In particular,
\[\partial_{ee} u(0,0)=\lim_{n\to \infty}\partial_{ee} u_n(0,
0)=\lim_{n \to \infty} \partial_{ee} w(x_n,t_n)=\liminf_{(y,s)\to 0, w(y,s)>0}\partial _{ee}w(y,s)=-c_0.\]
On the other hand, if $(x,t)\in Q_1^-$, then
\begin{equation*}
 \partial_{ee}u(x,t) =\lim_{n\to \infty} \partial_{ee}w(x_n+xr_n,t_n+tr_n^2)\geq \liminf_{(y,s)\to (0,0), w(y,s)>0}  \partial_{ee}w(y,s)=-c_0.
\end{equation*}
By the strong minimum   principle of the heat equation, we infer that
\begin{equation*}
    \partial_{ee}u(x,t)\equiv -c_0 <0,\qquad (x,t)\in Q_1^-,
\end{equation*}
and, by spatial analyticity of caloric functions, $\partial_{ee}u(x,0) \equiv -c_0$ for all $x$ in the connected component of $\{u(\cdot, 0) > 0\}$ containing $(0,0)$.

Consequently, define the $C^{1,1}_{\operatorname{loc}}(\R)$ function $g(s):=u(se,0)$. Since $|\nabla u|, u < \infty$, the constant negative second derivative in the direction $e$ implies that the connected component of $\{u(\cdot, 0) > 0\}$ containing $(0,0)$ is bounded in the $e$ direction. Thus there exists a value $s^* \in \R \setminus \{0\} $ such that $g''\equiv -c_0$ on $(-|s^*|,|s^*|)$, and  $g(s^*)=0$ (i.e. there exists a maximal open interval in the direction $e$, centered at the origin, contained in $\{u > 0\}\cap \{t=0\}$). 

Since $u\geq 0$, we also have $g'(s^*)=0$. In summary, one has $g\geq 0$, $g( s^*)=g'(s^*)=0$ and $g''<0$ in $(-|s^*|,|s^*|)$, which is a contradiction. This proves that, for any $e\in \mathbb{S}^{d-1}$,
\begin{equation} \label{eq:wee finite}
    \lim_{(x,t)\to (0,0), w(x,t) > 0}(\partial_{ee}w(x,t))^-=0.
\end{equation}
From here the estimate on the full Hessian follows from the boundedness of $D^2w$ and the compactness of $\mathbb S^{d-1}$.
\end{proof}

\subsection{Nondegeneracy of the free boundary}
\label{subsec:nondegeneracy estimates}Our next goal will be to show that, near any free boundary point, the free boundary may be locally described as $\{t=s(x)\}$, where the freezing time $s$ is Lipschitz continuous, and to obtain precise quantitative nondegeneracy near points of $\Sigma_{\operatorname{dyn}}$. We also show the semi-convexity of $w$ in time. A similar estimate was essential in the analysis of the melting problem, but our proof is necessarily different due to the discontinuity of $w_t$. Instead, we rely strongly on the nondegeneracy of the free boundary.

Formally, if $w(x, s(x)) \equiv 0$ then $\nabla_x w +\partial_tw \nabla s(x) = 0$. As such a bound on a the ratio $|\nabla_x w|/|\partial_t w|$ gives an upper bound on $\nabla s(x)$, which, by the inverse function theorem, is equivalent to a lower bound on the speed of propagation of the free boundary.

In this vein, we begin by obtaining an {\it a priori} lower bound for the speed of propagation of the boundary and a one-sided estimate for $D^2w$. At free boundary points where $w_t$ does not vanish continuously, we obtain a non-degeneracy estimate on the speed of the free boundary. This will be a crucial estimate when we prove, in Section \ref{sec:dynamic}, that $s$ is $C^1$ at singular points of $\Sigma_{\operatorname{dyn}}$.

Here we strongly use the  condition $\{w_t<0\}=\{w>0\}$ given in \eqref{eq:obstacle intro}, a property inherited from the supercooled Stefan problem.
\begin{lem} \label{lem:eta grad w} Let $r>0$, let $w:Q_{r}\to \R$ be a solution to \eqref{eq:obstacle intro}, let $\eta=-w_t$, and assume that $Q_{r} \subset B_1 \times (-1,1)$. There exists a constant $C>0$ such that
\begin{equation} \label{nondegeneracy general}|\grad w(x,t)| + |(D^2w)^-(x,t)|\leq C\eta(x,t), \quad (x,t)\in Q_{r/2}.    
\end{equation}

In addition, for any $m>0$, there exist constants $\delta\in (0,\frac{1}{10})$ and $K>1$, depending only on $d$, $m$, and $\|w\|_{\infty}$, but not on $r$, such that the following holds. Assume that $w(0,0)=0$ and
\begin{equation} \label{eta large near fast blow-up}
    \eta(x,t)\geq \frac{m}{2}, \quad  (x,t)\in B_{r}\times (-r^2,-(\delta r)^2).
\end{equation}
Then
\begin{equation} \label{nondegeneracy fast blow-ups}
    |\grad w(x,t)| \leq Kr \eta(x,t), \quad (x,t)\in Q_{r\delta }.
\end{equation}
\end{lem}
\begin{proof}    
We fix $(x_1,t_1)\in Q_{r/2}$ and let $D=Q_{r/4}^-(x_1,t_1)$. Fix $e\in \mathbb{S}^{d-1}$, let $c_1, c_{11},c_2\in (0,1)$ be small constants, to be chosen later, and let
\begin{equation} \label{eq:z defi spd pf}z(x,t)=\eta(x,t) - c_1\partial_{e}w(x,t)-c_{11}(\partial_{ee} w(x,t))^-+c_2\left(\frac{1}{2d+1}(|x-x_1|^2+(t_1-t))-w(x,t)\right).\end{equation}
Note that $z$ is supercaloric on $D \cap \{w>0\}$ and, since $w$, $\grad w$ and $(D^2w)^-$ are continuous, satisfies \begin{equation} \label{partk1o1}\liminf_{(x,t)\in D,  \text{dist}((x,t),\partial\{w>0\} )\to 0}z(x,t)\geq 0.\end{equation} Moreover, requiring $(c_1+c_{11})/c_2$ small enough, we may fix $\rho>0$ small enough (depending on $r$ and the $C^{1,1}_x\cap C^{0,1}_t$ norm of $w$, but not on $c_1$, $c_{11}$, $c_2$, or $(x_1,t_1)$), such that
\begin{equation} \label{partk1o2}z\geq -C(c_1\rho+c_{11})-c_2C\rho^2 +c_2\frac{r^2}{16(2d+1)}>0 \;\; \text{ on } \;\; \partial_{p} D\cap(\{w=0\}+Q_{\rho}).\end{equation}
Here we use the notation $\{w =0\}+Q_\rho$ to mean the sum of these two sets, i.e. the set of all points which are contained in a cylinder of radius $\rho$ centered at a point in the zero set of $w$. On the other hand, if $(x,t)\in D\backslash(\{w=0\}+Q_{\rho})$, since $\eta$ is smooth in $\{w>0\}$, from the condition $\{w>0\}=\{\eta>0\}$ in \eqref{eq:obstacle intro}, we have $\eta(x,t) \geq \vep$ for some $\vep>0$ that depends on $\rho$ and $w$, but not on $(x_1,t_1)$. We therefore have
\begin{equation}\label{partk1o3}
    z(x,t)\geq \vep -C(c_1+c_{11}+c_2)>0, \quad (x,t)\in D\backslash(\{w=0\}+Q_{\rho}).
\end{equation}
if $c_1$, $c_{11}$ and $c_2$ are chosen sufficiently small, depending on $\vep$ and the $C^{1,1}_x$ norm of $w$. By the maximum principle, we infer from \eqref{partk1o1}, \eqref{partk1o2}, and \eqref{partk1o3} that 
\[z\geq 0 \text{ on } D .\]
Since $e$ was arbitrary, evaluating at $(x,t)=(x_1,t_1)$, we conclude that
\begin{equation} \label{etasamfa}\eta(x_1,t_1) \geq \frac12c_1|\grad w(x_1,t_1)|+\frac12c_{11}|(D^2w)^-(x_1,t_1)| +c_2 w(x_1,t_1)\geq \frac12c_1|\grad w(x_1,t_1)|+\frac12 c_{11}|(D^2w)^-(x_1,t_1)|  .\end{equation}
This proves \eqref{nondegeneracy general}. 

Assume now that $\delta\in (0,\frac{1}{10})$ is such that \eqref{eta large near fast blow-up} holds, and let $(x_1,t_1)\in Q_{\delta r}$. We now prove that \eqref{nondegeneracy fast blow-ups} holds if $\delta$ is sufficiently small, through a variant of the above barrier argument, where $z$ is defined again by \eqref{eq:z defi spd pf}, but we set $c_{11}=0$ and $c_1=rc_2$ from the start, so that, for $c_2\in (0,\infty)$ to be chosen later,

\begin{equation*}
    z(x,t):=\eta(x,t)-c_2r \partial_{e}w(x,t)+c_2\left(\frac{1}{2d+1}(|x-x_1|^2+(t_1-t))-w(x,t)\right).
\end{equation*}
We let $D= Q_{r/4}^-(x_1,t_1)$ as before, so that \eqref{partk1o1} holds for any choice of $c_2$. Choosing $\rho=2\delta r$, we have
\begin{equation}\label{singpf12w1}
z\geq -c_2Cr^2\delta+c_2\frac{r^2}{16(2d+1)}>0 \;\; \text{ on } \;\; \partial_{p} D\cap(\{w=0\}+Q_{\rho}),
\end{equation}
provided that $\delta$ is sufficiently small, chosen independently of $r$ and $c_2$.
Assume now that $(x,t)\in D\backslash(\{w=0\}+Q_{\rho})$. Since $\{\eta>0\}=\{w>0\}$, we have $\eta>0$ on $Q^-_{\rho}(x,t)=Q^-_{2\delta r}(x,t)$. Moreover, since $B_{2\delta r}(x) \times \{t-(2\delta r)^2\} \subset B_{r} \times (-r^2,-(\delta r)^2)$, we obtain from \eqref{eta large near fast blow-up} that
\[\eta \geq m/2 \quad \text{on} \quad B_{2\delta r}(x,t)\times \{t-(2\delta r)^2\}.\]
The Harnack inequality therefore implies that $\eta(x,t)\geq k_0m$, for some constant $k_0\in (0,1/2)$ depending only on dimension. Hence, for some constant $C_1$ depending only on the $C^{1,1}_x\cap C_t^{0,1}$ norm of $w$, using the fact that $w(0,0)=0$,
\begin{equation}\label{singpf12w2}
    z(x,t)\geq k_0m  -C_1c_2 r^2, \quad (x,t)\in D\backslash(\{w=0\}+Q_{\rho}).
\end{equation}
Choosing $c_2 =\frac{1}{2}\frac{k_0m}{C_1r^2}$, we obtain $z(x,t) > 0$ for $(x,t)\in D\backslash (\{w=0\}+Q_\rho)$ 
and the maximum principle yields
\begin{equation*}
    \eta(x_1,t_1)\geq c_2r |\nabla w(x_1, t_1)| = \frac{1}{2} \frac{k_0m}{C_1r}|\grad w(x_1,t_1)|.
\end{equation*}
\end{proof}

The next result is formally a consequence of the lower bound on the free boundary speed implied by \eqref{nondegeneracy fast blow-ups}, but one needs to be careful in its use due to the fact that $|\nabla w|$ degenerates at free boundary points, and the temperature $\eta$ is discontinuous. 
\begin{prop}[Lipschitz continuity of the freezing time]\label{prop lipschitz}Let $w: Q_1 \to \R$ be a bounded solution to \eqref{eq:obstacle intro}. There exists a constant $C>0$ such that for every $(x_1,t_1)\in  Q_{\frac12} \cap {\{w=0\}} $ and every $(x_2,t_2) \in Q_{\frac12} \cap \overline{\{w>0\}}$ with $t_1<t_2$, one has
\begin{equation*}
0\leq  t_2-t_1 \leq C|x_2-x_1|   
\end{equation*}
In particular if $(0,0)\in \partial \{w>0\}$, there exists a unique Lipschitz continuous function $s: B_{\frac{1}{2C}} \to (-1,1)$ such that \[\{w>0\}\cap (B_{\frac{1}{2C}}\times (-1,1))=\{t<s(x)\}.\]
\end{prop}
\begin{proof}

By Lemma~\ref{lem:eta grad w},  there exists $\tilde{C}>0$ such that
\begin{equation}\label{e:gradwest}
    |\grad w(x,t)|\leq \tilde{C}|w_t(x,t)|, \quad (x,t)\in Q_{\frac12}.
\end{equation}
Let $(x_1,t_1),(x_2, t_2)$ be as in the statement of the proposition. By approximation, we may assume that $w(x_2,t_2)=:\vep>0$, and by monotonicity we then have $x_1 \neq x_2$. Let
\begin{equation*}
    \phi(s)=w(x_2-s(x_2-x_1),t_2-s\tilde{C}|x_2-x_1|), \quad s\in [0,b], \quad b=\min((t_2-t_1)/(\tilde{C}|x_2-x_1|),1).
\end{equation*}

Since $w$ is smooth in $\{w>0\}$, $\phi$ is smooth near any $s\in \{\phi>0\}$, and in light of \eqref{e:gradwest},
\begin{equation*}
    \phi'(s)=\grad w\cdot (x_1-x_2)-\tilde{C}w_t|x_2-x_1|\geq 0. 
\end{equation*}
(Note that $(x_2 - s(x_2-x_1), t_2 - s\tilde{C}|x_2 - x_1|) \in Q_{1/2}$ for all $s\in [0,b]$). Assume, by contradiction, that $t_2-t_1 > \tilde{C}|x_1-x_2|$. Then $b = 1$, and, by monotonicity of $w$,
\begin{equation*}
0<\phi(1)=w(x_1,t_2-\tilde{C}|x_2-x_1|)\leq w(x_1,t_1),
\end{equation*}
contradicting the fact that $(x_1,t_1)\in \{w=0\}$.

\end{proof}

We now exploit the regularity of the free boundary to obtain a time semi-convexity estimate for $w$.  This estimate is the main ingredient required to obtain a ``quadratic cleaning'' result (see Proposition \ref{prop:quadratic cleaning sigma}), which is analogous to those obtained in the melting problem (see \cite[Lem. 6.7]{colombo} and \cite[Lem. 8.1]{figalli}).
\begin{prop} \label{prop: time semiconvexity local}Let $w:Q_1 \to \R $ be a bounded solution to \eqref{eq:obstacle intro}. There is a constant $C=C(d,\|w\|_{L^{\infty}})>0$ such that
\[w_{tt}\geq -C \quad\hbox{ in } Q_{\frac12}.\]
\end{prop}
\begin{proof}
As mentioned above, the main issue relative to the melting problem (see \cite{figalli}) is that $\eta=-w_t$ is, in general, discontinuous. To address this issue, we substitute continuity of $\eta$ by the continuity and nondegeneracy of the free boundary provided by Proposition \ref{prop lipschitz}.
    Let $h>0$ be a small number, and let
    \[z(x,t)=\frac{\eta(x,t+h)-\eta(x,t)}{h}.\]
    We will show below that $z^+$ is subcaloric in $Q_1$. If this is true, then following \cite[Prop. 3.4]{figalli}, one may then use the parabolic Krylov-Safonov $L^{\vep}$ estimate, the Calder\'on-Zygmund estimates, and the time monotonicity of the sets $\{w(\cdot,t)>0\}$ to obtain
    \[\|z^+\|_{L^{\infty}(Q_\frac{1}{2})} \leq C\|w\|_{L^1(Q_1)}.\]
   We can then conclude by letting $h \to 0$ (note that, distributionally, we get an estimate on $\eta_t^+ = (-w_{tt})^+ = (w_{tt})^-$).
    
    It remains then to show that $z^+$ is subcaloric in $Q_1$. Note that $z$ is subcaloric in $\{w>0\}$ since, in that set, $(\chi_{\{x\mid \eta(x,t) > 0\}})_t = 0$, and so, by the monotonicity of $w$,
    \[z_t-\Delta z=h^{-1}(\chi_{\{\eta(t+h)>0\}})_t \leq0.\]
     To consider $z^+$ in $\{w=0\}$, define    
  \begin{equation*}
      A:=\text{int}(\{(x,t)\in Q_1:w(x,t+h)=0\}).
  \end{equation*}
    We claim that $\{w=0\}\cap Q_1 \subset A$. Indeed, let $(x_0,t_0)\in Q_1$ with $w(x_0,t_0)=0$.  We may assume that $(x_0,t_0)\in \partial \{w>0\}$, since otherwise trivially $(x_0,t_0)\in A$ by monotonicity of $w$. Let $r<\sqrt{h/2}$ be small enough so that $Q_r(x_0,t_0)\subset Q_1$.  By Proposition \ref{prop lipschitz}, there exist $c>0$ and a Lipschitz function $s:B_{cr}\to (t_0-r^2,t_0+r^2)$ such that
    \[ \{w>0\}\cap (B_{cr}(x_0)\times (t_0-r^2,t_0+r^2))=\{(x,t):t<s(x)\}.\] In particular, we have 
    \[B_{cr}(x_0) \times [t_0+r^2,1) \subset \{w=0\}.\]
   If $(x,t)\in B_{cr}\times (t_0-r^2,t_0+r^2)$, then $t+h \geq t_0-r^2+h \geq t_0+r^2$, hence $w(x,t+h)=0$. This proves that $B_{cr}(x_0)\times (t_0-r^2, t_0 +r^2) \subset \{(x,t):w(x,t+h)=0\}$, and therefore $(x_0,t_0)\in A$. 

   We have shown that $A$ is an open neighborhood containing the set $\{w=0\}$. But 
   \begin{equation*}
       z^+(x,t)=h^{-1}(\eta(x,t+h)-\eta(x,t))^+=h^{-1}(-\eta(x,t))^+=0, \quad (x,t)\in A
   \end{equation*}
   and thus $z^+$ is trivially subcaloric in $A$. Moreover, since $z$ is subcaloric in $\{w>0\}$, $z^+$ is subcaloric in $B=\{w>0\}$. Noting that
   \begin{equation*}
       Q_1=A \cup B,
   \end{equation*}
   since $z^+$ is subcaloric in both open sets $A$ and $B$, and subcaloricity is a local property, we conclude that $z^+$ is subcaloric in $Q_1$. 
    
\end{proof}

\subsection{Energy estimates}
We now obtain the energy estimates that will be needed to justify the existence, under certain conditions, of a second blow-up for $w$. Since the blow-up profiles can be time-dependent, some arguments that may work in the melting case break down here in subtle ways, and some of the estimates that hold there are simply false for the supercooled problem. 

For example, \cite[Corollary 6.2]{figalli} obtains an $L^\infty$ bound on $u_t$. In the context of the freezing problem, such a bound fails. Indeed, in the case that $w$ blows-up to $p = (-t)^+$ at $(0,0)$, an $L^\infty$ bound on $u_t$ would imply (in the notation of the result below) that $u$ (appropriately rescaled, see Proposition \ref{prop:second blow-up}) converges uniformly to the second blow-up at $(0,0)$ on the time slice $\{t = 0\}$ (this follows from Arzela-Ascoli). This would contradict the algebraic structure of the second blow-up given by \eqref{eq:sigma_d^2 blow-up form}. Accordingly, we only prove below a bound on $u_t^-$.

Due to these subtleties, we derive the bounds below in full detail. 
 
\begin{lem}\label{lem energy} Let $w: Q_2^- \to \R$ be a bounded solution to \eqref{eq:obstacle intro}, and let $p\in \mathcal{P}$. Let $e\in \ker (D^2 p)$ and, given $r\in (0,1)$, let 
\begin{equation}\label{udefst} u(x,t):=r^{-2}(w-p)(rx,r^2t).\end{equation} Then there exists a constant $C=C(d)>0$ such that
\begin{equation} \label{eq:energy}
    \|u\|_{L^{\infty}(Q_1^-)}+\|u_t^-\|_{L^{\infty}(Q_1^-)}+\|u_t\|_{L^{1}(Q_1^-)}+\|\nabla u\|_{L^2(Q_{1}^-)}+\|\cH u\|_{L^1(Q_1^-)}+\|(\partial_{ee}u)^-\|_{L^{\infty}(Q_1^-)}\leq C \|u\|_{L^2(Q_{2}^-)}.
\end{equation}
Additionally, if either
\begin{enumerate}
    \item[(i)] $p$ is space-independent, or
    \item[(ii)] $p$ is time-independent and $\dim(\ker(D^2p))=d-1$,   
\end{enumerate}
then we also have
\begin{equation}\label{grad bd u}
  \|\nabla u\|_{L^{\infty}(Q_1^-)}\leq C\|u\|_{L^2(Q_2^-)}.  
\end{equation}
Moreover, the same statements hold if $Q_1^-$ and $Q_2^-$ are replaced by $Q_1$ and $Q_2$ throughout.
\end{lem}

\begin{proof}
    Note that $u$ satisfies
    \[u_t-\Delta u = -\chi_{\{w_r>0\}}+1=\chi_{\{w_r=0\}}.\]
  In particular, $u$ is supercaloric, so $u^-$ is subcaloric. Moreover, $u^+$ is caloric in $\{u^+> 0\} \subset \{w_r>0\}$, so $u^+$ is subcaloric as well.  By the parabolic Harnack inequality, this means that
    \[\|u\|_{L^{\infty}(Q_{7/4}^-)}\leq C\|u\|_{L^{1}(Q_{2}^-)}\leq C\|u\|_{L^2(Q_{2}^-)}. \]
    The $L^2$ estimate on $\grad{u}$ follows by a standard energy computation, by observing that
    \[u_tu -u\Delta u = \chi_{\{w_r=0\}}u=-p\chi_{\{w_r=0\}}\leq0.\]
     It remains to estimate $u_t=(-\eta-(p)_t)(rx,r^2t)$. First note that $z=-u_t$ is subcaloric, because $(p)_t$ is constant and $\eta$ is subcaloric. Therefore, $z^+=(-u_t)^+$ is subcaloric, and, by \cite[Thm. 4.16]{wang} and interpolation
    \begin{equation} \label{20q-3edi0w221}\sup_{Q_1^-} (z^+) \leq C \left( \int_{Q_{5/4}^-} |z^+|^{\frac12}\right)^2\leq C\sup_{\theta>0}\theta|\{u_t^->\theta\}\cap Q_{5/4}^-|.\end{equation}
    By the parabolic Calder\'on-Zygmund estimate (see \cite[Thm. 3.5]{figalli}), 
    \begin{equation}\label{20q-3edi0w222}
        \sup_{\theta>0}\theta |\{|u_t|+|D^2u|>\theta\}\cap Q_{5/4}^-|\leq C(\|\cH u\|_{L^1(Q_{3/2}^-)}+\|u\|_{L^1(Q_{3/2}^-)}).
    \end{equation}
    Since $\cH u=-\chi_{\{w_r=0\}}\leq 0$, it readily follows by testing the identity $\cH u =\Delta u-u_t$ against a smooth, non-negative bump function that
    \begin{equation}\label{20q-3edi0w223}
        \|\cH u\|_{L^1(Q_{3/2}^-)}\leq C\|u\|_{L^{\infty}(Q_{7/4}^-)}\leq C \|u\|_{L^2(Q_2^-)}.
    \end{equation}
Thus, we conclude from \eqref{20q-3edi0w221}, \eqref{20q-3edi0w222}, and \eqref{20q-3edi0w223} that
\begin{equation*}
    \|u_t^{-}\|_{L^{\infty}(Q_1^-)} \leq C \|u\|_{L^2(Q_2^-)}.
\end{equation*}
In particular, we have
\[|u_t|=u_t+2u_t^- \leq u_t +C\|u\|_{L^{2}(Q_2^-)}.\]
Thus, the quantity $\|u_t\|_{L^{1}(Q_{1}^-)}$ can then be estimated in terms of $\|u\|_{L^2(Q_2^-)}$ in the same way as $\|\cH u\|_1$. That is, by testing against a bump function $\xi(x,t)$ supported in $Q_{3/2}^-$ such that $\xi \equiv 1$ on $Q_{1}^-$, we have
\begin{equation*}\begin{aligned}\int_{Q_{3/2}^-}\xi|u_t|&\leq  \int_{Q_{3/2}^-} u_t\xi +C\|u\|_{L^2(Q_2^-)}  \leq-\int_{Q_{3/2}^-}u\xi_t +C\|u\|_{L^2(Q_2^-)}\\ &\leq C\|u\|_{L^{\infty}(Q_{3/2})} +C\|u\|_{L^2(Q_2^-)} \leq C\|u\|_{L^2(Q_2^{-})}.
\end{aligned} \end{equation*}
To estimate $\partial_{ee}u^-$, we let $h\in (0,1)$ and let
\begin{equation*}
z^e(x,t):=\frac{u(x+eh,t)+u(x-eh,t)-2u(x,t)}{h^2}.
\end{equation*}
The map $z^e$ satisfies, in $\{w_r>0\}$,
\begin{equation*}
    z^e_t-\Delta z^e=h^{-2}(\chi_{\{w_r(\cdot+eh)=0\}}+\chi_{\{w_r(\cdot-eh)=0\}})\geq 0.
\end{equation*}
Moreover, since $e\in \ker (A)$, we have $z^e\geq 0$ in $\{w_r=0\}$, and thus
$(z^e)^-$ is subcaloric. One then obtains from \eqref{20q-3edi0w222} and \eqref{20q-3edi0w223}, as in \eqref{20q-3edi0w221}, that
\begin{equation*}
    \|(z^e)^-\|_{L^{\infty}(Q_1)}\leq C\|u\|_{L^{\infty}(Q_2)}.
\end{equation*}
Recalling that $e\in \ker (D^2p)$ and letting $h \downarrow 0$, we conclude
\begin{equation} \label{eq:uee-bd pf}
     \|(\partial_{ee}u)^-\|_{L^{\infty}(Q_1^-)}\leq C\|u\|_{L^{\infty}(Q_2^-)}.
\end{equation}
If $p$ is space-independent, then \eqref{eq:uee-bd pf} is a uniform semiconvexity bound for $u$, and thus \eqref{grad bd u} follows. On the other hand, if $p$ is time-independent and $\dim(\ker(D^2 p))=d-1$, then, since $w_t\leq0$, we get
\begin{equation*}
\Delta u=u_t-\chi_{\{w_r=0\}} =w_t - 1 \leq 0.
\end{equation*}
Together with \eqref{eq:uee-bd pf},  this implies $\|(\partial_{ee} u)^+\|_{L^\infty(Q_1^-)} \leq C\|u\|_{L^\infty(Q_2^-)}$ for $e\in \ker(D^2p)^{\perp}$. Combining this with the semiconvexity along the remaining directions given by \eqref{eq:uee-bd pf}, \eqref{grad bd u} follows once more. 

Finally, the same estimates with $Q_1^-$ and $Q_2^-$ replaced by $Q_1$ and $Q_2$, respectively, may be obtained with the same proofs, replacing for each $r\in \{1,3/2,5/4,7/4,2\}$, the parabolic cylinder $Q_{r}^-$ with $Q_r$, 
\end{proof}
Next, we remark that, at singular points where the  frequency defined in \eqref{e:frequency} is finite, the height function $H(r,\zeta(w-p))$ is comparable with the standard $L^2$ norm. 

\begin{lem} \label{lem: H L2 comp}
   
Let $w : \mathbb{R}^{d}\times(-1,1)\to[0,\infty)$ be a bounded solution of \eqref{eq:obstacle gen}, let $(0,0)$ be a singular point with blow-up profile $p_2$, and let $u:=\zeta(w-p_{2})$.  Assume that $\phi(0+,u)=\lambda\in (0,\infty)$ exists.
Then, for all $r\in(0,1)$,
\[
\frac{1}{C}\,H(r,u)^{1/2}\;\le\;\|u_{r}\|_{L^{2}(Q_{1}^{-})}\;\le\;C\,H(r,u)^{1/2},
\]
for some constant $C>1$ depending on $d$, $\lambda$, and $\|w\|_{L^{\infty}}$.
\end{lem}
This is a trivial generalization of \cite[Lem. 6.3]{figalli}, and has the same proof, which we omit, only noting that Lemma \ref{lem energy} serves as a substitute for \cite[Cor. 6.2]{figalli}. However, it is important to emphasize that the constant $C$ depends on the frequency $\lambda$, and degenerates as $\lambda \to \infty$. While this is not an issue in the melting case, where one always has $\lambda < 4$, the supercooled problem may have singular points of infinite frequency where Lemma \ref{lem: H L2 comp} becomes inapplicable (see Sections \ref{sec:top stratum superquadratic} and \ref{sec:top stratum infinite}). For this purpose, we compute a one-sided bound that is valid at such points where $\phi(0+)=\infty$ (or, equivalently, points where $\phi^{\gamma}(0+)=\gamma$ for all $\gamma>2$).

\begin{lem} \label{lem:comparability sigma d} Let $w:Q_1^- \to \R$ be a bounded solution to \eqref{eq:obstacle intro}, let $\gamma>2$, and let $(0,0)$ be a singular point with blow-up profile $p_2$. Let $u=w-p_2$, $\zeta$ be the fixed spatial cutoff function defined in \eqref{eq:cutoff defi}, and assume that
\begin{equation*}
   \phi^{\gamma}(0+,\zeta u)=\gamma. 
\end{equation*}
 There exists a constant $K>0$, depending only on dimension, $\gamma$, and $\|w\|_{\infty}$, such that
\begin{equation*}
    \|u\|_{L^2(Q_r^-)}\leq K|Q_r|^{\frac12}r^{\gamma}, \quad r\in (0,1/4).
\end{equation*}    
\end{lem}
\begin{proof}
Fix $\gamma$ and $10\cdot 2^{\gamma + \frac{d+2}{2}} \geq N>2^{\gamma+\frac{d+2}{2}}$. We claim that there exists $C_\gamma>1$ such that, for any $r\in (0,1/2)$, if
\begin{equation} \label{eq:freqinfregpf1dc}
        \|u\|_{L^2(Q_{2r}^-)}\leq  N \|u\|_{L^2(Q_r^-)},
\end{equation}
then
\begin{equation}\label{eq:freqinfregpf2dc}
    \|u\|_{L^2(Q_{r}^-)}\leq C_{\gamma}r^{\gamma+\frac{d+2}{2}}.
\end{equation}
Indeed, assume that \eqref{eq:freqinfregpf1dc} holds. By Lemma \ref{lem energy}, we have, for some constant $C_0>1$,
\begin{equation*}
    \|u\|_{L^{\infty}(Q_r^-)}\leq |Q_r|^{-1/2} C_0 \|u\|_{L^2(Q_{2r}^-)}\leq C_0N|Q_r|^{-1/2} \|u\|_{L^2(Q_r^-)}.
\end{equation*}
We then have, for $\vep \in (0,1)$,
\begin{equation*}
    \int_{B_r\times(-r^2,-\vep r^2)}u^2\geq \int_{Q_r^-}u^2-\vep |Q_r|\|u\|^2_{L^{\infty}(Q_r^-)}\geq (1-\vep C_0^2N^2) \int_{Q_r^-}u^2.
\end{equation*}
Thus, choosing $\vep=\frac12 C_0^{-2}N^{-2}$, we get
\begin{equation} \label{eq: u comparab pf1}
    \|u\|_{L^2(Q_r^-)}^2\leq 2\int_{B_r \times (-r^2,-\vep  r^2)}u^2\leq C_1r^{d}\int_{B_r \times (-r^2,-\vep r^2)}u^2G(x,t)dxdt\leq C_1 r^d\int_{-r^2}^{ -\vep r^2}H(|t|^{\frac12},\zeta u)dt,
\end{equation}
where $C_1>1$ depends only on $C_0$ and $N$.
On the other hand, by \eqref{eq:H frequency bounds upper}, we have, for some generic constant $C$ depending only on dimension, $\gamma$, and $\|w\|_{\infty}$,
\begin{equation*}
    H(\rho,\zeta u) \leq \frac{1}{c}(H(1/2,\zeta u)+(1/2)^{2\gamma})\rho ^{2\gamma}\leq C\rho ^{2\gamma}\leq Cr^{2\gamma}, \quad \rho \in (0,r).
\end{equation*}
Thus, \eqref{eq: u comparab pf1} yields
\begin{equation*}
   \|u\|_{L^2(Q_r^-)} \leq  \sqrt{C C_1} |Q_r|^{\frac12}r^{\gamma}:=C_{\gamma} r^{\gamma+\frac{d+2}{2}},
\end{equation*}
which proves that \eqref{eq:freqinfregpf1dc} implies \eqref{eq:freqinfregpf2dc}.

Assume, by contradiction, that
\begin{equation*}
    \|u\|_{L^2(Q_r^-)}>C_{\gamma}r^{\gamma+\frac{d+2}{2}}.
\end{equation*}
Then by the claim just proved,
\begin{equation*}
    \|u\|_{L^2(Q_{2r}^-)}>N\|u\|_{L^2(Q_r^-)}\geq C_{\gamma}2^{-(\gamma+\frac{d+2}{2})}N(2r)^{\gamma+\frac{d+2}{2}}.
\end{equation*}
Since $\frac{N}{2^{\gamma+\frac{d+2}{2}}}>1$, we may apply the same claim again with $r:=2r$. Iterating this process, we obtain
\begin{equation*}
    \|u\|_{L^2(Q_{2^{k}r}^-)}\geq C_\gamma\left(2^{-(\gamma+\frac{d+2}{2})}N\right)^k(2^kr)^{\gamma+\frac{d+2}{2}}.
\end{equation*}
Choosing $k\geq 1$ such that $\frac14\leq 2^kr \leq \frac12$, and using that $C_\gamma>1$, we infer that
\begin{equation*}
    \left(2^{-(\gamma+\frac{d+2}{2})}N\right)^k \leq C_d\|u\|_{\infty},
\end{equation*}
where $C_d>0$ is a dimensional constant. Since $N>2^{\gamma+\frac{d+2}{2}},$ this yields a contradiction for sufficiently large $k$. Therefore, we have
\begin{equation*}
    \|u\|_{L^2(Q_r^-)}\leq C_{\gamma}r^{\gamma+\frac{d+2}{2}}
\end{equation*}
for sufficiently small $r$. Up to increasing the value of $C_{\gamma}$, this yields the result for all $r\in (0,\frac14)$.
\end{proof}
 
\section{Stationary blow-up profiles, nucleation, and second blow-up}
\label{sec:stationary}The goal of this section is to analyze the set $\Sigma_{\operatorname{stat}}$ of singular free boundary points $(x_0,t_0)$ with a time-independent blow-up profile. In particular, we will show that the bottom and top strata, $\Sigma_{\operatorname{stat},0}$ and $\Sigma_{\operatorname{stat}, d-1}$, are both empty, where we denote
\begin{equation*}
\Sigma_{\operatorname{stat},k}:=\Sigma_k \cap \Sigma_{\operatorname{stat}}, \quad 0\leq k \leq d-1.
\end{equation*}

To rule out the bottom stratum, an important step of independent interest is to show that any nucleation event occurs at a singular point with a time-dependent blow-up profile. For ruling out $\Sigma_{\operatorname{stat}, d-1}$ we characterize the second blow-up profile $q$ at a given point in the set to lead to a contradiction. It is crucial here to show that the degree of $q$ is at most two,  in contrast to what is known for the melting problem. 

\subsection{Nucleation and the bottom stationary stratum}
We begin by characterizing the points where the temperature variable is continuous, which is of independent interest.

\begin{prop}[Points of continuous temperature] \label{prop: eta continuous} Let $w:Q_1 \to \R$ be a bounded solution to \eqref{eq:obstacle intro}, and assume that $(0,0)\in \partial\{w>0\}$. Then $\eta=-w_t$ is continuous at $(0,0)$ if and only if $(0,0)$ is a regular point or $(0,0)\in \Sigma_{\operatorname{stat}}$.
\end{prop}
\begin{proof}
Assume first that $(0,0)\in \Sigma_{\operatorname{dyn}}$ is a singular point with a time-dependent blow-up profile
\begin{equation*}
    p(x,t)=-mt+\frac12Ax\cdot x, \quad m\in (0,1].
\end{equation*}
In view of Proposition \ref{prop: time semiconvexity local}, $\eta_t$ is bounded above, and thus $\lim_{t\uparrow 0}\eta(0,t)$ exists. We then have
\begin{equation*}
    0=\lim_{t\uparrow 0} \frac{w(0,t)-p(0,t)}{t}=\lim_{t \uparrow 0}(-\eta(0,t)+m),
\end{equation*}
where the first equality is by the definition of blow-up. Thus
\begin{equation*}
    \lim_{t\uparrow 0}\eta(0,t)=m\neq 0=\eta(0,0),
\end{equation*}
so $\eta$ is discontinuous at $(0,0)$.

Conversely, assume that $(0,0)\notin \Sigma_{\operatorname{dyn}}$. If $(0,0)$ is a regular point, then $\eta$ is continuous at $(0,0)$ by Theorem \ref{thm:obstacle prelim}, so we may assume that $(0,0)\in \Sigma_{\operatorname{stat}}$ has a stationary blow-up profile $p(x)$. Let $r\in (0,1/4)$, and let $u:Q_1\to \R$ be as in \eqref{udefst}. Since $p$ is stationary, we have $u_t =w_t\leq 0$. Lemma \ref{lem energy} yields
\[\|u_t\|_{L^{\infty}(Q_1)}=\|u_t^-\|_{L^{\infty}(Q_1)}\leq C\|u\|_{L^{2}(Q_2)}.\]
By the definition of blow-up, $u \to 0$ uniformly in $Q_2$ as $r\to 0$.  We therefore have $u_t\to 0$ uniformly in $Q_{1}$ as $r\to  0$. Translating this statement back to $w$, and using once more the fact that $p$ is stationary, we infer that
\[\eta=-\partial_t w\to 0 \;\;\text{ as }\;\; |x|+|t|^{\frac12} \to 0,\]
which proves that $\eta$ is continuous at $(0,0)$.
\end{proof}

A distinct feature of the supercooled Stefan problem that is not present in the melting problem, is that solutions may exhibit spontaneous {\it nucleation } away from the boundary (see \cite{CKM25} for examples). This notion is made precise by the following definition.
\begin{defn}[Nucleation]\label{def:nucleation} Let $w:Q_1 \to \R$  be a bounded solution to \eqref{eq:obstacle intro}, and assume that $(0,0)\in \partial \{w>0\}$. We say that $(0,0)$ is a {\it nucleation point} of $w$ if there exists $r>0$ such that 
\begin{equation*}
(\overline{B_r}\times (-r^2,0])\cap\{w=0\} \subset B_r \times \{0\}.  
\end{equation*}    
\end{defn}
We begin by showing that when nucleation happens, it must occur at singular points where $\eta=-w_t$ is discontinuous and thus, by the above, the blow-up profile is time-dependent: namely, once nucleation happens, freezing near the point occurs faster than parabolically. This description of nucleation is of independent interest.
\begin{prop}[Discontinuity at nucleation points]\label{prop:nucleation}Let $w:Q_1 \to \R$ be a bounded solution to \eqref{eq:obstacle intro}, and assume that $(0,0)$ is a nucleation point of $w$. Then $\eta=-w_t$ is discontinuous at $(0,0)$. In particular, $(0,0)\in \Sigma_{\operatorname{dyn}}$. 
\end{prop}
\begin{proof} 
Since $(0,0)$ is a nucleation point, there exists $r>0$ such that $w>0$ in $Q_r^-\backslash \{t=0\}$. Thus $\cH w=1$ in $Q_r^-$, and by parabolic regularity $\eta$ has a smooth, caloric extension to $Q_r^-$. If $\eta \geq 0$ were to be continuous at $(0,0)$, then the fact that $\eta(0,0)=0$ would contradict the strong minimum principle of the heat equation. Therefore, $\eta$ is discontinuous at $(0,0)$, and, by Proposition \ref{prop: eta continuous}, we have $(0,0)\in \Sigma_{\operatorname{dyn}}$. 
\end{proof}

This observation, plus monotonicity, implies that if $p$ is a blow-up with $\ker(D^2 p) = \{0\}$ then $p$ must be time dependent. 

\begin{prop}[Bottom stratum of stationary singular profiles]\label{prop:sigma stat 0} Let $w:Q_1 \to \R$ be a bounded solution to \eqref{eq:obstacle intro}. Then
\begin{equation*}
    \Sigma_{\operatorname{stat},0}=\emptyset.
\end{equation*}    
\end{prop}
\begin{proof} Assume, by contradiction, that $(x_0,t_0)\in \Sigma_{\operatorname{stat},0}$ for some $(x_0,t_0)\in Q_1$. After a change of coordinates, we may assume that $(x_0,t_0)=(0,0)$. Let $p(x)$ be the blow-up profile of $w$ at $(0,0)$. By the definition of blow-up, we have
\begin{equation*}
    w(x,0) \geq p(x)+o(|x|^2)\; \text{ as }\; x\to 0.
\end{equation*}
Since $\ker (D^2 p_2)=\{0\}$ it follows that, for some $\lambda \in(0,1)$, and $r>0$ sufficiently small,
\begin{equation*}
    w(x,0)\geq \lambda|x|^2+o(|x|^2)\geq \frac{\lambda}{2}|x|^2>0, \quad x\in B_r\backslash \{0\}.
\end{equation*}
Therefore, recalling that $w_t \leq 0$, we see either that $(0,0)$ is a nucleation point (recall Definition \ref{def:nucleation}) or that $\{w = 0\}\cap Q_r^- = \{(0, t)\mid 0 \geq t \geq \epsilon\}$ for some $\epsilon \geq-r^2$. The former situation contradicts Proposition \ref{prop:nucleation}, whereas the latter situation contradicts the Lipschitz regularity of the freezing time (Proposition \ref{prop lipschitz}).

\end{proof}

\subsection{Second blow-up and the top stationary stratum}\label{ss:topstationarystrata}

Now we turn to our proof that $\Sigma_{\operatorname{stat}, d-1}$ is empty. We recall that this is a crucial step in our larger argument, and that the points in $\Sigma_{\operatorname{stat}}$ cause the most difficulty in \cite{figalli}. 

Our strategy is to characterize the second blow-up profiles at points in this set. Following the ideas in the melting problem \cite{figalli}, we first show an interesting pointwise estimate, which says that at such a point, $\eta = -w_t$ grows faster than linear rate in all but one direction. Then we will use this to show that the second blow-ups are at most cubic in $\Sigma_{\operatorname{stat},d-1}$. The rest of the argument, which is unique to the supercooled problem, shows that the cubic blow-ups are not possible.
\begin{lem}[Nondegeneracy away from the thin set] \label{lem: d-1 linear nondegeneracy} Let $w:Q_1 \to \R$ be a bounded solution to \eqref{eq:obstacle intro}, let $\eta=-w_t$, and assume that $(0,0)\in \Sigma_{\operatorname{stat},d-1}$ has the blow-up profile  $p_2(x)=\frac12 x_d^2$. Then there is a constant $c_0>0$ such that the following holds: for every $\delta>0$, there is $r_0=r_0(\delta) \in (0,1)$ such that 
\begin{equation*}
    \eta(x,t)\geq c_0|x_d|, \quad (x,t)\in Q_r^-, \quad |x_d|\geq \delta r \quad \hbox{ for all } r\leq r_0.
\end{equation*}
\end{lem}
\begin{proof} By definition of blow-up, we have
\begin{equation*}
    w(x,t)=\frac12 x_d^2+o(|x|^2+|t|), \quad (x,t)\in Q_1^-.
\end{equation*}
Thus, in particular, there exists $r_0\in (0,1)$ such that 
\begin{equation*}
    Q_{2r}^-\cap \{|x_d|\geq r\delta/2\} \subset \{w>0\},\quad  r\leq r_0.
\end{equation*}
We fix $(x,t)\in Q_r^-$ such that $|x_d|\geq \delta r$. Letting $\rho:=|x_d|/4$, we then have
\begin{equation*}
    Q^-_{\rho}(x,t)\subset \{w>0\}\cap Q_{2r}^-.
\end{equation*}
Thus, by the nondegeneracy property \eqref{eq:w nondegeneracy} and the mean value theorem, there exists $(y,s)\in Q_{\rho}^-(x,t)$ such that
\begin{equation*}
 \rho|\nabla w(y,s)|+\rho^2|w_t(y,s)|\geq c\rho^2.
\end{equation*}
Here $c$ is a small positive constant that may decrease in the steps below, but is independent of $\rho,$ $\delta$, and $r_0$. Lemma \ref{lem:eta grad w} then implies that
\begin{equation*}
 \sup_{Q_{\rho}^{-}(x,t)}\eta\geq c\rho.    
\end{equation*}
 By Proposition \ref{prop: time semiconvexity local}, $\eta_t$ is bounded above, and thus, provided that $r_0$ is sufficiently small, we infer that
\begin{equation*}
 \sup_{y\in B_{\rho}(x)}\eta(y,t-\rho^2)\geq c\rho.    
\end{equation*}
Finally, the Harnack inequality applied to the cylinder $Q_{\rho}^-(x,t)$ yields
\begin{equation*}
    \eta(x,t)\geq c\rho = \frac{c}{4}|x_d|.
\end{equation*}
\end{proof}

The next proposition shows  the existence of, subsequential, second blow-ups, which will also be used later in Section 6. However, part (c) in particular characterizes the second blow-up profiles at $\Sigma_{\operatorname{stat}, d-1}$, crucially restricting its degree to $(2,3]$, with the help of Lemma~\ref{lem: d-1 linear nondegeneracy}.

 It will be helpful to recall from \eqref{eq:scaling identities} the scaling identities
\[H(\rho,f_{r})=H(\rho r,f), \quad f_r(x,t)=f(rx,r^2t),\quad  \rho,r>0.\]

\begin{prop}[Second blow-up at singular points with finite frequency]\label{prop:second blow-up} Let $w:Q_1 \to \R$ be a bounded solution to \eqref{eq:obstacle intro}. Assume that $(0,0)\in\Sigma$ with blow-up profile $p_2$, and that $\lambda:=\phi^{\gamma}(0+,\zeta(w-p_2))$ for some fixed $\gamma>4$. Let
$$u^{(r)}(x,t) := \frac{(w-p_2)_r}{H(r,(w-p_2) \zeta)^{\frac12}}.$$ Then, for every sequence $r_k \downarrow 0$ there exists a subsequence $\{r_{k_l}\}$ and a function $q$ such that
\begin{equation*}
    u^{(r_k)} \to q \quad \text{and}  \quad \grad u^{(r_k)}\rightharpoonup \nabla q \quad \text{ in }\;L^2_{\operatorname{loc}}(\R^d\times (-\infty,0]).
\end{equation*}
Moreover,
    \begin{enumerate}
    \item[(a)] If $(0,0)\in \Sigma_{\operatorname{dyn}}$ and $\lambda<\gamma$, then $\lambda\geq2$ is an integer, $q$ is a caloric polynomial of degree $\lambda$, and $u^{(r_{k_l})} \to q$ in $C^{\infty}_{\operatorname{loc}}(\R^d\times (-\infty,0))$. Moreover, if $(0,0)\in \Sigma_d$ and $\lambda=2$, then there exist $a>0$ and a $d \times d$ matrix $A\geq 0$ such that
    \begin{equation} \label{eq:sigma_d^2 blow-up form}
     q(x,t)=-at -\frac12 Ax\cdot x, \quad \operatorname{tr}( A)=a,
    \end{equation}
    where 
\begin{equation} \label{eq:sigma d^2 q nondeg}
    \frac{1}{C}\leq a \leq C, \quad C=C(d,\|w\|_{\infty}).
\end{equation}
\item[(b)] If $(0,0)\in \Sigma_{\operatorname{stat},i}$ for some $i\leq d-2$, then $\lambda=2$ and $q$ is a caloric polynomial of degree $2$.
\item[(c)] If $(0,0)\in \Sigma_{\operatorname{stat},d-1}$, then $\lambda=3$, $u^{(r_{k_l})} \to q$ in $C^0_{\operatorname{loc}}(\R^d\times (-\infty,0])$, $q\in W^{1,\infty}_{\operatorname{loc}}(\R^d \times (-\infty,0])$, $\langle q,1\rangle <\infty$, and $q$ is a parabolically $\lambda$--homogeneous solution to the parabolic thin obstacle problem
\begin{equation} \label{eq:signorini}
    \begin{cases}
        \cH q\leq 0,\; q\cH q=0 & \text{ in }\;\; \R^d\times (-\infty,0)\\
        \cH q=0  & \text{ in } \;\;\R^d\times(-\infty,0) \setminus \{p_2=0\}\\
        q \geq 0 & \text{on} \;\; \{p_2=0\}\\
        \partial_t q \leq 0 & \text{ in }\;\; \R^d \times (-\infty,0).
    \end{cases}
\end{equation}
    \end{enumerate}
\end{prop}
\begin{proof} By Lemmas \ref{lem:freq H bd}, \ref{lem: H L2 comp}, and \ref{lem energy}, 
\begin{equation} \label{eq:energy urk blowup pf}\|u^{(r_k)}\|_{L^{\infty}(Q_R^-)}+R^2\|(u^{(r_k)})_t^-\|_{L^{\infty}(Q^-_R)}+ \left( \fint_{Q_R^-} R^2|\grad u^{(r_k)}|^2+R^2|\partial_tu^{(r_k)}|+R^2|\cH u^{(r_k)}| \right)^{\frac12} \leq C_{\delta,\gamma}R^{\lambda+\delta},    
\end{equation}
for $\delta>0$ and $1\leq R \ll r_k^{-1}$. Thus \cite[Thm. 2.4.2]{droniou}, up to a subsequence (which is not relabeled for simplicity), there exists $q$ such that \begin{equation*} u^{(r_k)} \to q\; \text{ in }\;L^2_{\text{loc}}(\R^d\times (-\infty,0)). \end{equation*} By the definition of blow-up, $r^{-2}w_r \to p_2$ locally uniformly in $\R^d \times (-\infty,0]$. Therefore, since $\cH u^{(r_k)}=0$ on $\{w_{r_k}>0\}$, we infer that $q$ is caloric in $\R^d\times (-\infty,0)\setminus \{p_2=0\}$. 

Assume first that $(0,0)\notin \Sigma_{\operatorname{stat},d-1}$. Then the set $\{p_2=0\}\cap \{-\infty < t < 0\}$ has zero capacity, which implies that $q$ is caloric. 

The proof of part (b), that is, the fact that $\lambda=2$ when $(0,0)\in \Sigma_{\operatorname{stat}}\backslash \Sigma_{\operatorname{stat},d-1},$ follows exactly as in \cite[Lem. 5.8]{figalli}, and so we omit it.

 We turn to the proof of part (a), and in particular assume that $\lambda < \gamma$. In view of \eqref{eq:scaling identities}, Lemma \ref{lem:freq H bd}, and the fact that  $\lambda<\gamma$, we have for $ r_k \ll r <1$,

\begin{equation}\label{120dkq0wkxq2s}
    H(r,\zeta_{r_k} u^{(r_k)})=\frac{H(r,(\zeta (w-p_2))_{r_k})}{H(r_k,\zeta (w-p_2))}=\frac{H(rr_k,\zeta (w-p_2))}{H(r_k,\zeta (w-p_2))}\leq C\frac{H(rr_k,u)+(rr_k)^{2\gamma }}{H(r_k,u)+r_k^{2\gamma}}\leq Cr^{2\lambda}.
\end{equation}
In the second to last inequality, we used the fact that, by Lemma \ref{lem:freq H bd}, $r^{2\gamma}\ll H(r,\zeta (w-p_2))$ for $r\ll1$.
Note also that,
\begin{equation} \label{eq:q energy pf blowup}
    H(1, q)=\lim_{k\to \infty}H(1,\zeta_{r_k}u^{(r_k)})=\lim_{k\to \infty}1=1.
\end{equation}
On the other hand,
\begin{equation*}
    \phi(r,\zeta_{r_k} u^{(r_k)})=\frac{D(r,\zeta_{r_k} u^{(r_k)})}{H(r,\zeta_{r_k} u^{(r_k)})}=\frac{D (r,(\zeta(w-p_2))_{r_k})}{H(r,(\zeta(w-p_2))_{r_k})}=\frac{D(rr_k,\zeta(w-p_2))}{H(rr_k,\zeta(w-p_2))}=\phi(rr_k,\zeta(w-p_2))
\end{equation*}
Letting $k\to \infty$, and using the polynomial $L^{\infty}_{\operatorname{loc}}$ bounds and the $L^2_{\text{loc}}$ convergence of $u^{(r_k)}$, as well as the lower semicontinuity of $u \mapsto D(r,u)$, we get
\begin{equation*}
    \phi(r,q) \leq \phi(0+,u)=\lambda.
\end{equation*}
We claim that $\phi(r,q) \equiv \lambda$. To see this, observe that $q$ is caloric in $\mathbb R^d\times (-\infty, 0)$ and as such $r\mapsto \phi(r, q)$ is monotone increasing, and is constant if and only if $q$ is a homogeneous caloric polynomial with homogeneity $\phi(1, q)$ (this is well known, see \cite[Page 525]{Poon}). Furthermore if at some $r_0 > 0$ we have $\phi(r_0, q) = \mu$ then $H(r,q) \geq r^{2\mu}(r_0^{-2\mu}H(r_0, q_0))$ for all $r < r_0$ (see \cite[Page 525]{Poon}). Using the strong $L^2$ convergence of $u^{(r_k)}$ to $q$ and the estimate \eqref{120dkq0wkxq2s},we see that $\mu \geq \lambda$ and thus we have that $\phi(r,q) \equiv \lambda$ and that $q$ is a parabolically $\lambda$-homogeneous caloric polynomial.

Let us now return to a more general situation, that is we assume that $p_2$ is the blowup at $(0,0)$ but make no assumption on what $p_2$ is except that we will assume there exists a $\gamma > 2$ such that if $u = (w-p_2)$ as above then $\phi^\gamma(0+,\zeta u) = \lambda < \gamma$.
We now fix $p\in\mathcal P$ and set $U^p:=\zeta\,(w-p)$. By \eqref{eq: integration by parts} and Lemma~\ref{lem:frequency}, we have that $r^{-4}H(r, U^p)$ is ``almost-monotone" in the sense that \[
\frac{d}{dr}\Bigl(r^{-4}H(r,U^p)\Bigr)
\ge\;-\frac{C}{r^5}e^{-1/r}.
\]
Most relevantly this factor, $e^{-1/r}r^{-5}$, is integrable at zero so for any $\rho < r \ll 1$ we have 

\begin{equation*}\
\rho^{-4}H\!\left(1,(U^p)_\rho\right) \;\le\; r^{-4}H\!\left(1,(U^p)_r\right) \;+\; \Psi(r),
\end{equation*} where $r^{-\alpha}\Psi(r) \downarrow 0$ as $r\downarrow 0$ for any $\alpha > 0$ (to see this, note that $e^{-1/r}r^{-\alpha}$ is integrable at zero for any $\alpha > 0$).

Using that $r^{-2}w_r \rightarrow p_2$ and the almost-monotonicity described above we get 

\begin{align}
\int_{\{t=-1\}} (p_2-p)^2 G
&\le \int_{\{t=-1\}}r_k^{-4}
\bigl(\zeta_{r_k}(H_k^{\frac12} u^{(r_k)}+r_k^2(p_2-p))\bigr)^2\,G \;+\;\Psi(r_k)\nonumber\\
&\leq\;\!\int  (p_2-p)^2 G
\;+\;2H_k^{\frac12} r_k^{-2}\!\int \zeta_{r_k}^{2} (p_2-p)u^{(r_k)}G
\;+\;r_k^{-4}H_k\!\int \zeta_{r_k}^2 (u^{(r_k)})^2 G
\;+\;\Psi(r_k),\label{eq:expanded-corr}
\end{align}
where $\zeta_{r_k}(x) = \zeta(r_kx)$.

Divide by $H_k^{\frac12} r_k^{-2}$ and let $k\to \infty,$. We note that $H_k^{\frac12}r_k^{-2} \to 0$ by the definition of blow-up. If we can show that $H^{-\frac{1}{2}}_k r_k^2\Psi(r_k)\rightarrow 0$ then we will have obtained the key inequality \begin{equation}\label{e:signedmonneau}
\int_{\{t=-1\}} (p_2-p)\,q\,G \;\ge\; 0, \;\; \forall p\in \mathcal P.
\end{equation}
But indeed, by Lemma \ref{lem:freq H bd}, as long as $\phi^\gamma(0+, \zeta u) = \lambda < \gamma$ for some $\gamma > 2$ (which was assumed above) we have that $H_k \gtrsim r^{2\lambda}$ for some $\lambda < \gamma$ and thus by our observation above we have $H^{-\frac{1}{2}}_k r_k^2\Psi(r_k)\rightarrow 0$. A more useful form of \eqref{e:signedmonneau}, which follows from a simple application of \eqref{eq: integration by parts}, is
\begin{equation} \label{eq:optimality pf}
 0\leq    \langle p_2-p,q\rangle = \frac1\lambda \langle p_2-p,Zq\rangle =\frac{2}{\lambda}\langle \nabla (p_2-p), \nabla q\rangle.
\end{equation}

To finish the proof of part (a), assume again that $(0,0)\in \Sigma_d$ and $\lambda=2$, we may write 
\begin{equation} \label{eq:q eq 2nd blowup pf}
 q=-at-\frac{1}{2}Ax\cdot x
\end{equation}    
 where $A$ is a symmetric matrix and $\text{tr}(A)=a$. Then \eqref{eq:optimality pf} implies that for any $p(x)=\frac12 Bx\cdot x$, with $B\geq 0$ and $\operatorname{tr}(B)=1$,
\begin{equation*}
    0\leq \int_{\{t=-1\}} Bx \cdot  Axdx.
\end{equation*}
Evaluating the integral, we infer that
\begin{equation*}
    \operatorname{tr}(BA)\geq0,
\end{equation*}
and thus $A\geq 0$. But, in view of \eqref{eq:q energy pf blowup}, since all norms are equivalent in the space of quadratic polynomials, we also have \eqref{eq:sigma d^2 q nondeg}.

It remains to treat the case where $(0,0)\in \Sigma_{\operatorname{stat},d-1}$, i.e. part (c). By Lemma \ref{lem: d-1 linear nondegeneracy}, we have, for sufficiently small $r$ and a constant $c>0$,
\begin{equation*}
\int_{B_r}|(w(x,-r^2)-p_2(x)-(w(x,-r^2/4)-p_2(x))|dx= \int_{-r^2}^{-r^2/4}\int_{B_r}\eta(x,t)dx\geq cr^3 |B_r|.
\end{equation*}
Thus, by Lemmas \ref{lem energy}, \ref{lem: H L2 comp} and \ref{lem:freq H bd},
\begin{equation*}
    cr^3\leq CH(r,(w-p_2)\zeta)^{\frac12}\leq Cr^{\lambda},
\end{equation*}
which implies that $\lambda\leq 3$.

We may assume with no loss of generality that $p_2=\frac{1}{2}x_d^2$.  Since $\partial_t p_2\equiv 0$ and $w_t\leq 0$, we have $\partial_t u^{(r_k)}\leq 0$ and thus $\partial_t q \leq 0$. By Lemma \ref{lem energy}, we also obtain
\begin{equation*}
    R\|\nabla u^{(r_k)}\|_{L^{\infty}(Q_R^-)}+R^2\|u^{(r_k)}_t\|_{L^{\infty}(Q_R^-)}\leq C_{\delta,\gamma}R^{\lambda+\delta},
\end{equation*}
for $\delta>0$ and $1\leq R \ll r_k^{-1}$, which implies that $u^{(r_k)}\to q$ in $C^0_{\operatorname{loc}}(\R^d\times (-\infty,0])$. In particular, by uniform convergence, since $u^{(r_k)}\geq0$ on $\{p_2=0\}$, we have  $q \geq 0$ on $\{p_2=0\}$. 
We also have
\begin{equation*}
 \cH u^{(r_k)}=  -\chi_{\{w_{r_k}=0\}} r_k^2H_k^{-1/2}\leq 0 \Longrightarrow \cH q \leq 0,
\end{equation*}
and 
\begin{equation*}
    u^{(r_k)}\cH u^{(r_k)}=r_k^2H_k^{-1/2}(p_2)_{r_k}\chi_{\{w_{r_k}=0\}}\geq 0 \;\Longrightarrow q\cH q \geq0.
\end{equation*}
Moreover, since $q$ is caloric outside of $\{p_2=0\}$, the measure $\cH q \leq 0$ is supported on $\{p_2=0\} \subset \{q\geq 0\}$, which implies that
\begin{equation}\label{eq:qHq=0}
    q\cH q=0.
\end{equation}
That is to say, $q$ satisfies a parabolic Signorini-type problem. 

The $\lambda$--homogeneity of $q$ can be shown similarly to before, using \eqref{eq:qHq=0} as a replacement for caloricity (see, e.g. \cite[Proof of Theorem 7.3]{DanGarPetTo}).
It remains to show that $\lambda=3$. By \cite[Thm. 3.1]{colombo}, it follows that $\lambda \notin(2,3)$.  Noting that $\lambda\geq2$ by \eqref{eq:freq geq 2}, and since we have already shown that $\lambda \leq 3$, we only need to rule out the case $\lambda=2$. Assuming by contradiction that $\lambda=2$, we see from \cite[Lem. 12.4]{DanGarPetTo} that $q$
 must be a non-zero caloric polynomial, and $q$ must have the form \eqref{eq:q eq 2nd blowup pf}, where $A$ is a symmetric matrix and $\operatorname{tr}(A)=a$.  Since the frequency is bounded, we have \eqref{eq:optimality pf} for any $p=\frac{1}{2}Bx\cdot x$, $B\geq 0$, $\operatorname{tr}(B)=1$, which yields
 \begin{equation*}
     0\leq \int_{\{t=-1\}} -(e_d \otimes e_d-B)x\cdot Ax G.
 \end{equation*}
 Evaluating the integral, we infer that
 \begin{equation} \label{eq:matrixeq2 2blowup pf}
     \operatorname{tr}(BA)\geq \operatorname{tr}(A(e_d\otimes e_d))=A_{dd}.
 \end{equation}
 Setting $B=e_i \otimes e_i$ for $i\in \{1,\ldots d-1\}$ yields
 \begin{equation*}
     A_{ii}\geq A_{dd}.
 \end{equation*} 
 Since  $q\geq 0$ on $\{p_2=0\}$, we have $A_{ii}\leq 0$, so that $A_{dd}\leq 0$, and $a=\operatorname{tr}(A)\leq 0$. On the other hand, since $q_t\leq 0$, we have $a\geq 0$, and thus $a=0$ and $A_{dd}=0$. But then \eqref{eq:matrixeq2 2blowup pf} implies that $A\geq 0$, which forces $A=0$, and thus $q=0$, a contradiction.
\end{proof}

Having shown that $q$ is cubic, we now show that $q$ is quantitatively time-dependent, a property it inherits from the fact that $w_t \leq 0$. We will subsequently use this observation to rule out cubic second blow-ups entirely.

\begin{lem}\label{l:upperqtbound} Let $w$, $q$, $\lambda$, and $p_2$ be as in Proposition \ref{prop:second blow-up}, and assume that $p_2(x)=\frac12 x_d^2$ and $\lambda=3$. Then there exists $c>0$ such that
\begin{equation}\label{e:upperqtbound}
    q_t(x,t)\leq -c|x_d|, \quad (x,t)\in Q_1^-.
\end{equation}
\end{lem}
\begin{proof}
By Lemma \ref{lem: d-1 linear nondegeneracy}, for any fixed $(x,t)\in Q_1^-$ with $x_d \neq 0$ and sufficiently small $r$, we have  
\begin{equation*}
    \partial_t(w-p_2)_r \leq -cr^3|x_d|.
\end{equation*}
Thus, dividing by $H^{\frac12}$, using \eqref{eq:H frequency bounds upper}, and letting $r\to 0$, the claim follows (recall that, away from $|x_d| = 0$, we have $w_r > 0$ for $r > 0$ small enough and thus the convergence to the second blow-up is smooth).
\end{proof}

Our next lemma is an algebraic observation, that odd reflections of solutions to the Signorini problem must be caloric. 

\begin{lem}\label{l:oddreflection} Let $q$ be a locally Lipschitz continuous solution to the Signorini problem \eqref{eq:signorini}, where $p_2=\frac12x_d^2$. If $q\equiv0$ on $\{x_d=0\}$, then \begin{equation}\label{e:qtilde}\tilde{q}(x', x_d, t):= \begin{cases} q(x', x_d ,t), \qquad x_d > 0\\
        -q(x', -x_d, t), \qquad x_d \leq 0.
    \end{cases}\end{equation}
    is a caloric function.
\end{lem}

\begin{proof} Since $q$ solves the Signorini problem, $\nabla q$ is locally H\"older continuous in $\overline{\{x_d>0\}}$ \cite[Lem. 3.3]{DanGarPetTo}, and thus $\tilde{q}$ and $\nabla \tilde{q}$ are continuous across the thin set $\{x_d=0\}$. The fact that $\tilde{q}$ is caloric then follows from integration by parts. 
\end{proof}

Finally, we are able to rule out the existence of degree-three homogeneous second blow-ups. 

\begin{lem}\label{l:notimedeptcubics}
    There exists no $3$--homogeneous solution to \eqref{eq:signorini} satisfying the condition \eqref{e:upperqtbound}.
\end{lem}

\begin{proof}
    By \cite[Lemma 9.2]{figalli} (which does not use the sign of $\partial_t q$), we have that $q\equiv 0$ on $\{x_d = 0\}$ for any $3$--homogeneous solution to the Signorini problem \eqref{eq:signorini}. Using Lemma \ref{l:oddreflection}, since $\tilde{q}$ is 3--homogeneous, it must equal a cubic caloric polynomial on each side of $\{x_d=0\}$. Namely, there exist constants $a^L, a^R, b^L,b^R$ and matrices $A^L, A^R$ such that 
    $$
    \begin{aligned} q|_{\{x_d < 0\}} =& x_d(a^Lx_d^2 + b^Lt + x'A^Lx'),\\
    q|_{\{x_d > 0\}} =& x_d(a^Rx_d^2 + b^Rt + x'A^Rx').\end{aligned}$$
    Additionally, since $\Delta q - \partial_t q \leq 0$, integration by parts yields 
    \[\partial^+_n q|_{\{x_d = 0\}} - \partial^-_nq|_{\{x_d =0\}} \leq 0,\]
    that is,
    \[(b^R-b^L)t \leq 0.\]
    Since $t<0$, we infer that $b^R \geq b^L$. But Lemma \ref{l:upperqtbound} implies that $b^L > 0$ and $b^R <0$, which is a contradiction.
    
\end{proof}

We may now rule out the top stationary stratum altogether.

\begin{thm}\label{t:nofullkernel} Let $w:Q_1 \to \R$ be a bounded solution to \eqref{eq:obstacle intro}, and let $\Sigma$ be its set of singular points. Then
\begin{equation*}
    \Sigma_{\operatorname{stat},d-1}=\emptyset.
\end{equation*}
\end{thm}
\begin{proof}
    Assume, by contradiction (and up to translation and rescaling), that $(0,0)\in \Sigma_{\operatorname{stat},d-1}$, and let $q$ be as in Proposition \ref{prop:second blow-up}. Such $q$ must satisfy the nondegeneracy condition \eqref{e:upperqtbound}, which is a contradiction to Lemma \ref{l:notimedeptcubics}.
\end{proof}

\subsection{The intermediate stationary strata}
Having shown that the top and bottom strata of the set $\Sigma_{\operatorname{stat}}$ are both empty, we now focus on the remaining singular points with stationary blow-up. We begin by briefly explaining why our present results are already sufficient to obtain the sharp parabolic Hausdorff dimension estimate for the set $\Sigma_{\operatorname{stat}}$, when combined with the same tools that work for the melting problem. Then, we finish this section by proving a non-degeneracy estimate near the points of $\Sigma_{\operatorname{stat}}$ that, much like its dynamic counterpart \eqref{nondegeneracy fast blow-ups}, will be needed later to obtain the space-time $C^1$ regularity of the free boundary.

We will make use of the following basic GMT lemma (see e.g. \cite[Lem. 8.9]{figalli}).
\begin{lem} \label{lem: FRS cleaning} Let $E \subset \R^d \times (-1,1)$ with
\[\dim_{\mathcal{H}}(\pi_x(E)) \leq \beta.\]
Assume that for any $(x_0,t_0)\in E$, there exists $C=C(x_0,t_0)>0$ and $\rho=\rho(x_0,t_0)>0$ such that
\begin{equation} \label{pardim condition}
    \{(x,t) \in B_{\rho}(x_0)\times (-1,1): t-t_0 <- C|x-x_0|^{2}\} \cap E =\emptyset.
\end{equation}
Then $\dim_{\operatorname{par}}(E)\leq \beta.$
    
\end{lem}
The following quadratic cleaning result is proved in exactly the same way as \cite[Lem. 6.6, Lem. 6.7]{colombo}, in view of Proposition \ref{prop:second blow-up}, the spatial semiconvexity estimates of Lemmas \ref{lem energy} and \ref{lem:eta grad w}, the time semiconvexity estimate of Proposition \ref{prop: time semiconvexity local},  and the fact that $\Sigma_{\operatorname{stat},d-1}=\emptyset$.

\begin{prop} \label{prop:quadratic cleaning sigma} Let $w:Q_1 \to \R$ be a bounded solution to \eqref{eq:obstacle intro}, with $(0,0) \in \Sigma_{\operatorname{stat}}$. Let $p_2=p_2(x)$ be the blow-up profile for $w$ at $(0,0)$, and assume that $e_d:=(0,\ldots,0,1)$ is an eigenvector corresponding to the maximal eigenvalue for $D^2p_2$. Then there exist constants $c>0$ and $C>0$ such that
\begin{equation*}
    \fint _{B_r \cap \{|x_d|\geq\frac{r}{10}\}}|w_t(\cdot,-r^2)|\geq cr^{-2}H^{\frac12}(r,w-p_2), \quad r\in (0,1),
\end{equation*}
and
\begin{equation*}
    (B_{r/2}\times (-1,-Cr^2]) \cap \{w=0\}=\emptyset.
\end{equation*}    
\end{prop}

\begin{cor}\label{cor:stat_par} We have
\begin{equation*}
    \dim_{\operatorname{par}}(\Sigma_{\operatorname{stat}})\leq d-2.
\end{equation*}
\end{cor}
\begin{proof} By Theorem \ref{t:nofullkernel}, $\Sigma_{\operatorname{stat},d-1}=\emptyset$. Thus, by Theorem \ref{thm:monneau}, $\dim_{\mathcal{H}}(\pi_{x}(\Sigma_{\operatorname{stat}}))\leq d-2$. The result then follows by combining  Lemma \ref{lem: FRS cleaning} and Proposition \ref{prop:quadratic cleaning sigma}.
\end{proof}

For the last estimate of this section, we will require the following barrier construction from \cite[Lem. 5.8]{figalli}.
\begin{lem} \label{lem: OU barrier Sigma < d-1} Let $p_2=p_2(x)$ be a quadratic polynomial in $\R^d$, and let $L=\operatorname{ker}(D^2 p_2)$. Assume that $\dim(L)\leq d-2$, and for $\delta>0$, let
\begin{equation*}
    \mathcal{C}_{\delta}:=\{(x,t):x\in \R^d, \;t<0,\;\operatorname{dist}(x,L)\geq \delta(|x|+|t|^{\frac12})\}.
\end{equation*}
Then for every $\vep>0$ there exist $\delta,\alpha\in (0,\vep)$ and a positive, parabolically $\alpha$-homogeneous  $\Phi:\mathcal{C}_{\delta} \to(0,\infty)$ such that 
\begin{equation*}
    \cH \Phi =0 \text{ in } \mathcal{C}_{\delta}, \quad \Phi =0 \text{ on } \partial \mathcal{C}_{\delta },\quad \int_{\R^d} \Phi (y,-1)^2e^{-|y|^2/4}dy=1.
\end{equation*}
\end{lem}
We now show a non-degeneracy estimate  that is tailored for stationary singular points, where there is no lower bound on $\eta$ given from the blow-up profile. The proof is based on the above Lemma as well as revisiting the proof of Lemma~\ref{lem:eta grad w}. 
\begin{lem} \label{lem: eta nondegeneracy space blow-ups}Let $w:Q_1 \to \R$ be a bounded solution to \eqref{eq:obstacle intro}  with $(0,0)\in \Sigma_{\operatorname{stat}} $. For every $\vep>0$, there exists a constant $r_0\in (0,1)$ such that
\begin{equation} \label{eq: eta nondegeneracy space blow-ups}
    |\grad w(x,t)|\leq r^{1-\vep}\eta(x,t), \quad (x,t)\in Q_{r}, \quad  0<r<r_0.
\end{equation}
\end{lem}
\begin{proof}Assume that $(0,0)\in \Sigma_{\operatorname{stat}}$, and let $p_2$ be the blow-up profile for $w$. By Theorem \ref{t:nofullkernel}, $p_2$ satisfies the assumptions of Lemma \ref{lem: OU barrier Sigma < d-1}. Let $\vep>0$, and let $\delta \in (0,\vep/2)$, $\alpha \in (0,\vep/2)$, and $\Phi:\mathcal{C}_{\delta}\to (0,\infty)$ be  given by Lemma \ref{lem: OU barrier Sigma < d-1}, taken with $\vep :=\vep/2$.  By uniform convergence to the blow-up, for $r_0$ sufficiently small, 
\begin{equation*}
\mathcal{C}_{\delta/2}\cap  Q_{r_0} \subset \{w>0\}.
\end{equation*}
Observing that it suffices to prove \eqref{eq: eta nondegeneracy space blow-ups} for $\vep$ sufficiently small, we may assume that  $\vep$, $\alpha$ and $\delta$ are as small as needed. Since $\Phi$ is caloric on $\mathcal{C}_{\delta}$ and vanishes continuously on $\partial \mathcal{C}_{\delta}$, $\Phi$ is subcaloric on $\R^d \times (-\infty,0)$. Thus, by comparison with the heat flow on $\R^d$,
\begin{multline}
    \|\Phi\|_{L^{\infty}(Q_{1/2}^-)}\leq \sup_{(x,t)\in Q_{1/2}^-}\int_{\R^d} \Phi(y,-1)e^{-|x-y|^2/4(1-|t|)}\leq  \sup_{|x|<1/2}\int_{\R^d} \Phi(y,-1)e^{-|x-y|^2/4}dy\\
    \leq \sup_{|x|<1/2} e^{-|x|^2/4}\left(\int_{\R^d}e^{x\cdot y}e^{-|y|^2/4}dy\right)^{\frac12}\left(\int_{\R^d}\Phi(y,-1)^2e^{-|y|^2/4}dy\right)^{\frac12}\leq C.
\end{multline}
Since $\eta$ is positive and caloric in $\{w>0\}$, the comparison principle yields, for some generic constant $c>0$,
\begin{equation*}
    \eta \geq c\Phi, \quad \text{in}\quad  Q_{r_0}^-\cap \mathcal{C}_{\delta}.
\end{equation*}
In particular, this implies that
\begin{equation*}
    \eta(x,t)\geq cr^{\vep/2} \quad\hbox{ for }  (x,t)\in  Q_{2r}^-,\;\; \operatorname{dist}(x,L)\geq 8\delta r,  \quad r\in (0,r_0/2). 
\end{equation*}
 By the Harnack inequality, we then have
\begin{equation*}
  \eta(x,t)\geq cr^{\vep/2} \quad \hbox{ for } (x,t)\in  Q_{2r},\;\;\operatorname{dist}(x,L)\geq 9\delta r,  \quad r\in (0,r_0/2).
\end{equation*}
We now show \eqref{eq: eta nondegeneracy space blow-ups} by a variant of the proof of Lemma \ref{lem:eta grad w}. We let, for $(x_0,t_0)\in Q_{r}$ and $e\in \mathbb{S}^{d-1}$,
\begin{equation*}
    z(x,t):=\eta(x,t) - r^{\vep-1}\partial_e w+r^{\vep-2}\left( \frac{(|x-x_1|^2+(t_1-t))}{2d+1}-w\right), \quad (x,t)\in D:=Q_{r}(x_0,t_0),
\end{equation*}
Note that, by uniform convergence, since $\Delta p_2(x)=1$, we have, for any $(x,t) \in Q_{2r}$, 
\begin{equation*}
    w(x,t)\leq p_2(x) + o(1)r^2 \leq \frac12\operatorname{dist}(x,L)^2+o(1)r^2,
\end{equation*}
and
\begin{equation*}
    |\partial_e w(x,t)| \leq |\nabla p_2(x)| +o(1)r\leq \operatorname{dist}
(x,L)+o(1)r,\end{equation*}
where $o(1) \to0$ as $r\downarrow 0$.
We thus have, for any $\rho>0$ and any $(x,t)\in \partial_pQ_r(x_0,t_0)$ such that $\operatorname{dist}(x,L)< 9\delta r$,
\begin{equation}
  z(x,t)\geq  r^{\vep}\left(- ( 9\delta +o(1))- (81\delta^2/2+o(1)) +\frac{1}{2d+1}\right) >0
\end{equation}
provided that $r_0$ and $\delta$ are sufficiently small. On the other hand, for any $(x,t)\in \partial_p Q_r(x_0,t_0)$ such that $\operatorname{dist}(x,L)\geq9\delta r$, we have
\begin{equation*}
    z(x,t)\geq cr^{\vep/2}- C r^{\vep} >0
\end{equation*}
if $r_0=r_0(\vep)$ is sufficiently small. We then conclude by the maximum principle that
$z\geq0$ in $D$ and, in particular,
\begin{equation*}
   \eta(x_0,t_0)\geq r^{\vep-1}|\partial_e w(x_0,t_0)|. 
\end{equation*}
This holds for all $(x_0,t_0)\in Q_r$ and all $r<r_0/2$, which implies the result.
\end{proof}

\section{Time-dependent blow-up profiles and regularity of the freezing time}\label{sec:dynamic}
 In this section, we begin the study of the set $\Sigma_{\operatorname{dyn}}$ of singular points with a time-dependent blow-up profile, and prove two of our main results. First, we complete the remaining parabolic dimension estimates necessary to prove Theorem \ref{thm: dim par}. Next, we exploit the non-degeneracy results of Lemma \ref{lem:eta grad w} and Lemma \ref{lem: eta nondegeneracy space blow-ups} to obtain the $C^1$ regularity of the freezing time and prove Theorem \ref{thm:reg intro}. 

 We start by noting that, near points of $\Sigma_{\operatorname{dyn}}$,  the lower bound of $\eta$ necessary for Lemma \ref{lem:eta grad w} is satisfied.

\begin{lem}[Superquadratic cleaning of $\Sigma_{\operatorname{dyn}}$]\label{lem: freeze non-stationary}
Let $w:Q_1\to \R$ be a bounded solution to \eqref{eq:obstacle intro} with $(0,0)\in \Sigma_{\operatorname{dyn}}$, and let $p_2=-mt+\frac12Ax\cdot x$ be the corresponding blow-up profile. Then, for every $\delta>0$, there exists $r_0\in (0,1)$ such that $\eta:=-w_t$ satisfies
\begin{equation} \label{eq:eta > m/2 proved}
    \eta(x,t)\geq \frac{m}{2}, \quad (x,t)\in B_r \times (-r^2,-(\delta r)^2), \quad r\in (0,r_0). 
\end{equation}
In particular, there exists a modulus of continuity $\omega$ and a radius $r_\ast>0$
such that
\begin{equation*}
  \{(x,t): |x|<r,\ t<-\omega(r)\,r^2\}
\cap \{w=0\}=\emptyset
\quad\text{for all }0<r<r_\ast.  
\end{equation*}
\end{lem}

\begin{proof}
By the definition of blow-up, we have $r^{-2}w_r \to p_2$ uniformly in $Q_1^-$. Since $p_2(x,t)>0$ for $t<0$, this implies that $r^{-2}w_r>0$ on $B_1 \times [-1,-\delta^2]$ for $r$ sufficiently small and thus $r^{-2}w_r-p_2$ is caloric. By parabolic regularity, $\eta(rx,r^2t)=-\partial_tr^{-2}w_r\to m$ uniformly in $B_1 \times [-1,-\delta^2]$, which implies \eqref{eq:eta > m/2 proved}. For the second assertion, fix a decreasing sequence $(\delta_j)_{j\ge1}$ with
$\delta_j \downarrow 0$. For each $j$,
let $r_j \in (0,1)$ be the radius given by \eqref{eq:eta > m/2 proved} with $\delta=\delta_j$,
so that
\[
\eta(x,t)\ge \frac{m}{2}
\qquad\text{for all }(x,t)\in B_r\times(-r^2,-(\delta_j r)^2)
\quad\text{whenever }0<r\le r_j.
\]
Shrinking $r_j$ if necessary, we may assume that $r_{j+1}\le r_j/2$ for all $j$,
and we set $r_\ast := r_1$. Since $\{w>0\} = \{\eta>0\}$, the inequality above, together with the fact that $w_t\leq 0$, implies
\[
B_r \times (-1,-(\delta_j r)^2) \subset \{w>0\}
\qquad\text{for every }0<r\le r_j.
\]
Now define $\omega:(0,r_\ast)\to(0,1)$ by
\[
\omega(r) := \delta_j^2 \qquad\text{whenever } r\in(r_{j+1},r_j],\ j\ge1,
\]
with the convention $r_{j+1}<r\le r_j$ and $r_1=r_\ast$. By construction,
$\omega$ is nondecreasing and $\omega(r)\to0$ as $r\downarrow0$, so $\omega$
is a modulus of continuity. For any $0<r<r_\ast$, choose $j$ such that
$r\in(r_{j+1},r_j]$; then $r\le r_j$ and the inclusion above with $\delta_j$
gives
\[
B_r\times(-1,-\omega(r) r^2)
=
B_r\times(-1,-(\delta_j r)^2)
\subset\{w>0\}.
\]

\end{proof}

In view of the results on $\Sigma_{\operatorname{stat}}$ from the previous section, we may now show that the singular set is low-dimensional for almost all times.
\begin{prop} \label{prop: dim slices a.e. d-2}  Let $w:Q_1 \to \mathbb{R}$ be a bounded solution to \eqref{eq:obstacle intro}. Then, for almost every $t$,
\begin{equation*}
    \mathcal{H}^{d-2}(\{x\in B_1: (x,t)\in\Sigma\} )=0.
\end{equation*}
\end{prop}
\begin{proof}By Theorem \ref{t:nofullkernel} and Theorem \ref{thm:monneau}, the spatial projection of $\Sigma_{\operatorname{stat}}$  is locally contained in a $(d-2)$--dimensional manifold. Since (by monotonicity of $w$) $\Sigma_{\operatorname{stat}}$ is also contained in the graph of a function, this readily implies that, for all but countably many $t$,
\begin{equation*}
    \mathcal{H}^{d-2}(\{x: (x,t)\in \Sigma_{\operatorname{stat}}\})=0.
\end{equation*}
On the other hand, Lemma \ref{lem: freeze non-stationary} implies that, for every $(x_0,t_0)\in \Sigma_{\operatorname{dyn}}$, there exists a modulus of continuity $\omega_{x_0,t_0}$ such that
\begin{equation*}
    \{(x,t):t< t_0-\omega_{x_0,t_0}(|x-x_0|)|x-x_0|^2\} \cap \Sigma_{\operatorname{dyn}} = \emptyset.
\end{equation*}
Since we trivially have $\mathcal{H}^{d}(\pi_x(\Sigma_{\operatorname{dyn}}))<\infty$, \cite[Prop. 7.7]{FRS20} implies that, for a.e. $t$,
\begin{equation*}
  \mathcal{H}^{d-2}(\{x: (x,t)\in \Sigma_{\operatorname{dyn}}\})=0,  
\end{equation*}
which proves the claim.
\end{proof}
We estimate below the parabolic Hausdorff dimension of each dynamic stratum $\Sigma_{\operatorname{dyn},k}$, where
\begin{equation*}
    \Sigma_{\operatorname{dyn},k}:=\Sigma_{\operatorname{dyn}}\cap \Sigma_k, \quad 0\leq k \leq d.
\end{equation*}
\begin{prop}[Parabolic Hausdorff dimension of dynamic singular points]\label{prop: Hausdorff time dependent} Let $w:Q_1 \to \mathbb{R}$ be a bounded solution to \eqref{eq:obstacle intro}. Then \begin{equation} \label{eq:dimpar k}
    \dim_{\operatorname{par}}(\Sigma_{\operatorname{dyn},k})\leq k.
\end{equation} In particular, one has
\begin{equation}\label{eq:dimpar general}
    \dim_{\operatorname{par}}(\Sigma_{\operatorname{dyn}})\leq d
\end{equation}
and, if $\eta:=-w_t<1$, then $\dim_{\operatorname{par}}(\Sigma_{\operatorname{dyn}})\leq d-1$.
\end{prop}
\begin{proof} Let $(x_0,t_0)\in \Sigma_{\operatorname{dyn}}$, let $p_2$ be its blow-up profile. By Lemma \ref{lem: freeze non-stationary}, there exists $\rho=\rho(x_0,t_0)$ such that if $r\leq\rho(x_0,t_0)$, then $w>0$ on $B_r(x_0)\times [t_0-r^2,t_0-r^2/2]$. Recalling the time monotonicity of $w$, that means
\begin{equation}\label{e:quadraticclearing}
    \left\{(x,t)\in B_{\rho}(x_0) \times (-1,1): t-t_0 < -\frac12|x-x_0|^2\right\}\cap \Sigma_{\operatorname{dyn}}=\emptyset.
\end{equation}
By Theorem \ref{thm:monneau}, we have $\dim_{\mathcal{H}}(\pi_x(\Sigma_{\operatorname{dyn,k}}))\leq k$ for each $k\in \{0,\ldots,d\}$. Thus, we infer from Lemma \ref{lem: FRS cleaning} that $\dim_{\text{par}}(\Sigma_{\operatorname{dyn},k})\leq k$. Finally, note that if $(x_0,t_0)\in \Sigma_d$, then $p_2=-t$. But this is impossible if $\eta=-w_{t}<1$. Thus $\Sigma_d=\emptyset$ and $\dim_{\operatorname{par}}(\Sigma_{\operatorname{dyn}})\leq d-1$.
\end{proof}

We may now prove the first two main results of the paper.
\begin{proof}[Proof of Theorem \ref{thm: dim par}] We decompose
\begin{equation*}
    \Sigma=\Sigma_{\operatorname{stat}}\cup \Sigma_{\operatorname{dyn}},
\end{equation*}
and conclude by applying Proposition \ref{prop: Hausdorff time dependent} and Corollary \ref{cor:stat_par}.    
\end{proof}
\begin{proof}[Proof of Theorem \ref{thm:reg intro}] Since $\partial \{w>0\}$ is smooth near any regular point, it suffices to assume that $(0,0)$ is a singular point, and to prove that the freezing time $s$ given by Proposition \ref{prop lipschitz} is differentiable in a neighborhood of $x=0$, with $\lim_{x\to 0} \grad s(x)=0$. If $(0,0)\in \Sigma_{\operatorname{stat}}$, then we have by Lemma \ref{lem: eta nondegeneracy space blow-ups} that, for $r$ sufficiently small,
\begin{equation} \label{eqsqrtc1spko}
    |\nabla w(x,t)|\leq r^{\frac{1}{2}}|w_t|, \quad (x,t)\in Q_r.\quad 
\end{equation}
On the other hand, if $(0,0)\in \Sigma_{\operatorname{dyn}}$, we let $p_2(x,t)=-mt+\frac12 Ax \cdot x$ be the blow-up profile of $(0,0)$, and let $\delta=\delta(m)\in (0,1)$ and $K=K(m)>1$ be as in Lemma \ref{lem:eta grad w}. By Lemma \ref{lem: freeze non-stationary}, there exists $r_0>0$ such that 
\[\eta(x,t)\geq m/2, \quad (x,t)\in   B_r\times [-r^2,-(\delta r)^2], \quad 0<r\leq r_0.\]
We therefore obtain, by Lemma \ref{lem:eta grad w} applied at scale $r:=r/\delta$, shrinking $r_0$ if necessary, that
\begin{equation*}  |\nabla w(x,t)|\leq \frac{K}{\delta}r\eta(x,t), \quad (x,t)\in Q_{r}, \quad r\leq r_0.\end{equation*}
In particular, we infer that \eqref{eqsqrtc1spko} holds for $r$ sufficiently small, regardless of whether $(0,0)\in \Sigma_{\operatorname{stat}}$ or $(0,0)\in \Sigma_{\operatorname{dyn}}$.  Thus,  since $|w_t|=\eta > 0$ on $\{w>0\}$, we have, choosing  $r_0$ sufficiently small, by Proposition \ref{prop lipschitz} and the implicit function theorem, that
$\{w=\vep\}\cap B_{r^2}\times (-r^2,r^2)$ is a smooth graph $\{t=s_{\vep}(x)\}$, with $|\nabla s_{\vep}(x)|\leq r^{\frac12}$, for all $\vep>0$ sufficiently small. Letting $\vep \downarrow 0$, the Arzel\'a-Ascoli theorem implies that $s_{\vep}\to s$ uniformly in $B_{r^2}$, and 
\begin{equation*}
    |\nabla s(x)|\leq Cr^{\frac12} \quad \text{for a.e. } x\in B_{r^2}.
\end{equation*}
Letting $r \downarrow 0$, this implies that $s$ is differentiable at $x=0$, with $\grad s(x)=0$. A posteriori, repeating this argument centered at any other singular points near $(0,0)$, we infer that $s$  is everywhere differentiable near $x=0$, with $\lim_{x\to 0}\grad s(x)=0$.   
\end{proof}

\section{Quadratic second blow-up in the top stratum: the set $\Sigma_d^2$}\label{sec:top stratum quadratic}
For the remaining sections, we will focus exclusively on the top stratum $\Sigma_{d} \subset \Sigma_{\operatorname{dyn}}$ of singular points with blow-up profile $p_2(x,t)=-t$. A key feature of singular points with a time-dependent blow-up profile is that $w-p_2$ may have infinite frequency. For each integer $k\geq 2$, we define the points of frequency $k$ by
\begin{equation*}
    \Sigma_{d}^k:=\{(x_0,t_0)\in \Sigma_d: \phi^{\gamma}(0+,\zeta(w(x_0+\cdot,t_0+\cdot)+t))=k \;\;\text{for some $\gamma>k$}\}.
\end{equation*}
By Lemma \ref{lem:freq H bd}, the sets $\Sigma_d^k$ are disjoint, and one has $\phi(0+,\zeta(w(x_0+\cdot,t_0+\cdot)+t))=k$ for any $(x_0,t_0)\in \Sigma_d^k$. Letting
\begin{equation}\label{e:defsigmak}
    \Sigma_{d}^{\geq k}:=\{(x_0,t_0)\in \Sigma_d: \phi^{k}(0+,\zeta(w(x_0+\cdot,t_0+\cdot)+t))=k\},
\end{equation}
(equivalently, $\phi^{\gamma}(0+)=\gamma$ for all $\gamma \in (2,k]$) we then define the points of infinite frequency by
\begin{equation*}    \Sigma_d^{\infty}:=\bigcap_{k=2}^{\infty}\Sigma^{\geq k}_d.
\end{equation*}
By Proposition \ref{prop:second blow-up}, one then has the disjoint decomposition
\begin{equation*}    \Sigma_d=\left(\dot\bigcup_{k=2}^{\infty}\Sigma^k_{d}\right)\dot \cup \Sigma_{d}^{\infty}. 
\end{equation*}

In this section, we show that the set $\Sigma_d^2$ of points with frequency $\lambda=2$, where one expects to have a quadratic second blow-up $q,$ is dimensionally no larger than the penultimate stratum $\Sigma_{d-1}$. That is, we obtain the estimate
\begin{equation*}
    \dim_{\operatorname{par}}(\Sigma_d^{2})\leq d-1,
\end{equation*}
thereby proving that jump-type phenomena does not occur in $\Sigma_d^2$.
 The main difficulty in analyzing these points  is the slow rate of convergence of $r^{-2}w_r(x,t)\xrightarrow{r\downarrow 0} -t$ (slower than any power $r^{\alpha}$, see Lemma \ref{lem:freq H bd}). In particular, the second blow-up obtained in Proposition \ref{prop:second blow-up} may be only subsequential, so that the caloric polynomial $q$ may not be unique.

We begin with a simple lemma which expresses a unique feature of the supercooled problem at its top stratum, which is that superquadratic cleaning occurs both forward and backward in time. The faster-than-parabolic backward cleaning simply follows from the fact that the blow-up profile is time-dependent (see Lemma \ref{lem: freeze non-stationary}). The forward cleaning, on the other hand, is explained by the interaction between obstacle problem non-degeneracy, the fact that $p_2(x,0)\equiv 0$, and the fact that $w$ decreases in time (see \eqref{eq:forward cleaning pf} below. The fact that $w_t\leq0$ is already implicit in \eqref{eq:w nondegeneracy}). 
\begin{lem}[Two-sided superquadratic cleaning] \label{lem:two-sided cleaning} Let $w:Q_2 \to \R$ be a bounded solution to \eqref{eq:obstacle intro}, and for $c_d\in (0,1)$ as given in \eqref{eq:w nondegeneracy}, let $\delta\in (0,\frac{c_d}{10})$. Suppose that $(0,0)\in \Sigma_d$ and that for some $r\in (0,1)$
\begin{equation} \label{eq:supquadcl linft cond}
    \|w-(-t)\|_{L^{\infty}(Q_{2r}^-)}\leq \delta r^2.
\end{equation}
Then
\begin{equation} \label{eq:quad cleaning sigma_d^2}
 B_r \times [-r^2,-c_d^{-1}\delta r^2] \subset \{w>0\},\quad  B_r\times [c_d^{-1}\delta r^2,r^2]\subset \{w=0\}.
\end{equation} 
\end{lem}
\begin{proof}
  If $(x_1,t_1)\in B_r \times [-r^2,-c_d^{-1}\delta r^2]$, then, by \eqref{eq:supquadcl linft cond},
 \begin{equation*}
     w(x_1,t_1) \geq (-t_1)-\|w-(-t)\|_{L^{\infty}(Q_{2r}^-)} \geq (-t_1)-\delta r^2\geq c_{d}^{-1}(1-c_d)\delta r^2>0.
 \end{equation*}
On the other hand, if $w(x_1,t_1)>0$ for some $(x_1,t_1)\in B_r \times [0,r^2]$, then applying \eqref{eq:w nondegeneracy} (at scale $\sqrt{t}$), we infer that
 \begin{equation} \label{eq:forward cleaning pf}
   c_dt_1 \leq \sup_{B_{\sqrt{t_1}}(x)}w(\cdot,0)\leq \|w+t\|_{L^{\infty}(Q_{2r}^-)}\leq \delta r^2 \implies t_1\leq   c_d^{-1}\delta r^2.
 \end{equation}    
\end{proof}

We now show that, even if $q$ is not unique, the function $w$ qualitatively inherits, away from $t=0$, some of the concave behavior of its subsequential limits. This appears as a sign rigidity at small scales, albeit with a possibly scale-dependent direction of ``concavity''. 
\begin{lem}[Sign rigidity of $\Sigma_d^2$] \label{lem: sign rigidity} Let $w:Q_1 \to \R$ be a bounded solution of \eqref{eq:obstacle intro} with $(x_0,t_0)\in \Sigma_{d}^2\cap Q_1$, and let $\delta>0$. Then there exists $n=n(x_0,t_0,\delta)\in \mathbb{N}$ such that, for every $r\in(0,1/n)$ such that $Q_{2r}(x_0,t_0) \subset Q_1$,
\begin{equation} \label{eq:rfnb linft cond}
\|w-(t_0-t)\|_{L^{\infty}(Q_{2r}^-(x_0,t_0))}\leq c_d\frac{\delta r^2}{10} \quad \hbox{ and } 
\end{equation}
\begin{equation} \label{eq: por ineq 1}
    w(x_0,t_0+r^2t)>-r^2t,\quad t\in [-1,0),
\end{equation}
where $c_d\in (0,1)$ is the dimensional constant in \eqref{eq:w nondegeneracy}. Furthermore,  there exists $e=e(r)\in \mathbb{S}^{d-1}$ such that
\begin{equation} \label{eq: por ineq 2}
w(x_0+rx,t_0+r^2t)<-r^2t , \quad (x,t)\in \overline{B_1}\times[-1,-\delta/2], \quad |x\cdot e|^2\geq 4d|t|.
\end{equation}  
\end{lem}
\begin{proof}
The fact that \eqref{eq:rfnb linft cond} holds for sufficiently large $n$  follows simply from the definition of blow-up. We note next that  showing \eqref{eq: por ineq 1} for all $t\in [-1,0)$ and all sufficiently small scales $r$ is equivalent to showing that, for $r$ sufficiently small,
\begin{equation} \label{eq: por ineq 11}
    w(x_0,t_0-r^2)>r^2.
\end{equation}
 Assume, by contradiction, that there exists a sequence $r_k \to 0$ such that for each $k$, either \eqref{eq: por ineq 11} fails to hold for $r=r_k$, or \eqref{eq: por ineq 2} fails to hold for $r=r_k$ and every $e\in \mathbb{S}^{d-1}$. After translating $(x_0,t_0)$ to the origin, let $p_2=-t$ and let $u^{(r)}$ be as in Proposition \ref{prop:second blow-up}. Up to a subsequence, we may assume that $u^{(r_k)}\to q$ in $L^2_{\text{loc}}(\R^d\times (-\infty,0])$ for some caloric polynomial satisfying \eqref{eq:sigma_d^2 blow-up form} and \eqref{eq:sigma d^2 q nondeg}.
By Lemma \ref{lem: freeze non-stationary}, we have, for $k$ sufficiently large, 
\begin{equation*}
B_{2r}(x_0)\times [t_0-(2r)^2,t_0-(\delta/10)r^2] \subset \{\eta>0\}=\{w>0\},
\end{equation*}
and $u^{(r_k)}$ is caloric in  $B_2\times [-4,-\delta/10]$. Since $u^{(r_k)}\to q$ in $L^2_{\text{loc}}$, parabolic regularity implies that $u^{(r_k)} \to q$ uniformly in $B_1 \times [-1,-\delta/2]$. Recalling, from \eqref{eq:sigma_d^2 blow-up form}, that we can write $q(x,t) = -at - \frac{1}{2}Ax\cdot x$, 
 let $\lambda_{\max}$  be the largest eigenvalue of $A$ and $e\in \mathbb{S}^{d-1}$ be a corresponding eigenvector.  
Then, for $|x\cdot e|^2\geq 4d|t|$ and $t\in (-1,-\delta/2]$, using the inequality $\lambda_{\max}\geq \frac{1}{d}\operatorname{tr}(A)=\frac{a}{d}$, we get from \eqref{eq:sigma d^2 q nondeg} that
\begin{equation*}
 q(x,t)=-at-\frac12Ax\cdot x\leq a|t| -\frac12\lambda_{\max}(x\cdot e)^2\leq a|t| -\frac{1}{2d}a(x\cdot e)^2\leq-a|t|\leq -\frac12a\delta <0.   
\end{equation*}
Hence, by uniform convergence, if $k$ is sufficiently large, $u^{(r_k)}(0,-1)>0$ and $u^{(r_k)}(x,t)<0$ for $(x,t)\in \overline{B_1}\times [-1,-\delta/2]$. But this says, respectively, for $r=r_k$, that \eqref{eq: por ineq 11} holds and \eqref{eq: por ineq 2} holds for the particular eigenvector $e$, which contradicts the assumption about the sequence $\{r_k\}$.
\end{proof}
We note that the slack condition $\delta>0$ in \eqref{eq: por ineq 2} is inevitable due to the slow convergence near $t=0$. Indeed, if one could take $\delta=0$, then \eqref{eq: por ineq 2}  would trivially imply that $\Sigma_d^2=\emptyset$, but this is false in general (see Proposition \ref{p:sigma2example}).

 Next, we use the rigidity result of Lemma \ref{lem: sign rigidity} and the forward cleaning of Lemma \ref{lem:two-sided cleaning} to show that, after stratifying $\Sigma_d^2$ at countably many scales, one can obtain a one-sided-in-time Reifenberg--type condition. 
\begin{lem}[Countable slab trapping for $\Sigma_d^2$]\label{lem:Sn cleaning}Let $\delta\in (0,\frac12)$, and, for each $n\in \mathbb{N}$, under the notation of Lemma \ref{lem: sign rigidity}, let
 \begin{equation}\label{eq:Sn defi}S_n=\{(x_0,t_0)\in \Sigma_{d}^2\cap \overline{Q_{1-1/n}}: n(x_0,t_0,\delta)\leq n\}.\end{equation} 
Then, for every $(x_1,t_1)\in S_n$ and every $r\in(0,\frac{1}{n})$, there exists $e=e(r, x_1,t_1)\in \mathbb{S}^{d-1}$ such that
\begin{equation*}
S_n\cap (B_r (x_1)\times [t_1,1]) \subset \{(x,t):|(x-x_1)\cdot e|\leq 2(d \delta )^{\frac12}r\}.  
\end{equation*}
\end{lem}
\begin{proof} 
For $(x_1,t_1) \in S_n$ and $0<r<\frac{1}{n}$, let $e=e(r,x_1,t_1)\in \mathbb{S}^{d-1}$ be a vector such that \eqref{eq: por ineq 2} holds with $x_0:=x_1$. Assume, by contradiction, that $(x_2,t_2)\in S_n \cap (B_{r}(x_1)\times[t_1,1])$ and $|(x_2-x_1)\cdot e|> 2(d\delta)^{\frac12} r$. Then, by \eqref{eq:rfnb linft cond}, applying Lemma \ref{lem:two-sided cleaning} with $\delta:=\frac{c_d\delta}{10}$, we have
\begin{equation*}
    B_r(x_1)\times [t_1+(\delta/10)r^2,t_1+r^2] \subset \{w=0\}.
\end{equation*}
Thus, since $w_t\leq 0$ and $(x_2,t_2)\in \partial\{w>0\}$, we must have 
\begin{equation*}
   t_2 \in [t_1,t_1+\delta r^2/10].
\end{equation*}
Taking $(x_0,t_0):=(x_1,t_1)$, $x:=(x_2-x_1)/r$, and $t:=-\delta+(t_2-t_1)/r^2\in [-1,-\delta/2]$ in \eqref{eq: por ineq 2}, we obtain
\begin{equation*}
 w(x_2,t_2-r^2\delta)<r^2\delta-(t_2-t_1)\leq r^2\delta. 
\end{equation*}
But taking $(x_0,t_0):=(x_2,t_2)$ and $t=-\delta$ in \eqref{eq: por ineq 1}, we obtain
\begin{equation*}
  w(x_2,t_2-r^2 \delta)> r^2\delta,  
\end{equation*}
a contradiction.
\end{proof}
We may now use the ``future only'' Reifenberg-type condition to obtain a local (non-centered) slab-trapping result for the entire spatial projection of $\Sigma_d^2$.
\begin{lem}[Countable slab trapping for $\pi_x(\Sigma_d^2)$]\label{lem:pi_xSn cleaning}Let $\delta\in (0,\frac12)$ and let $S_n$ be given by \eqref{eq:Sn defi}. Then, for every $x_0\in \pi_x(S_n)$, and every $r\in(0,\frac{1}{2n})$, there exist $x_*\in \overline{B_{r}(x_0)}$ and $e=e(r,x_0)\in \mathbb{S}^{d-1}$ such that
\begin{equation} \label{eq:reifenberg projection sigma_d^2.}
    \pi_x(S_n) \cap B_r(x_0) \subset \{x\in \R^d: |(x-x_*)\cdot e|\leq 4(d\delta)^{\frac12}r\}.
\end{equation}
\end{lem}
\begin{proof} Let $x_0\in S_n$ and $r\in (0,\frac{1}{2n})$ be fixed throughout the proof, and consider the quantity
\begin{equation} \label{eq:mini problem sigmad2 pf}
    t_*=\operatorname{inf}_{(x,t)\in S_n\cap( \overline{B_r}(x_0)\times(-1,1))}t.
\end{equation}
Assume first that \eqref{eq:mini problem sigmad2 pf} is attained at  $(x_*,t_*)\in S_n$. Then by Lemma \ref{lem:Sn cleaning} there exists $e\in \mathbb{S}^{d-1}$ such that
\begin{equation*}
    S_n \cap (B_{2r}(x_*)\times [t_*,1))\subset \{(x,t):|(x-x_*)\cdot e|\leq 4(d\delta)^{\frac12}r\}.
\end{equation*}
By the minimality of $t_*$, we have
\begin{equation} \label{eq:cleanbelowpf12}
    S_n \cap (B_{r}(x_0)\times (-1,t_*))=\emptyset.
\end{equation}
On the other hand, since $B_r(x_0)\subset B_{2r}(x_*)$, we deduce
\begin{equation*}
    S_n\cap (B_{r}(x_0)\times (-1,1))=S_n\cap (B_{r}(x_0)\times [t_*,1))\subset S_n\cap (B_{2r}(x_*)\times [t_*,1)) \subset \{x:|(x-x_*)\cdot e|\leq4(d\delta)^{\frac12}r\},
\end{equation*}
which is precisely \eqref{eq:reifenberg projection sigma_d^2.}.

Now, assume instead that the infimum \eqref{eq:mini problem sigmad2 pf} is not attained, and let $\{(x_k,t_k)\}\in S_n\cap (\overline{B_r}\times(-1,1))$ be a minimizing sequence. Up to extracting a subsequence, and recalling \eqref{eq:Sn defi}, we may assume that
\[(x_k,t_k)\xrightarrow{k\to \infty}(x_*,t_*)\in \overline{B_r}(x_0)\times \left[-(1-1/n)^2,(1-1/n)^2\right].\] 
By Lemma \ref{lem:Sn cleaning}, for each $k\in \mathbb{N}$ there exists $e_k\in \mathbb{S}^{d-1}$ such that
\begin{equation*}
    S_n \cap (B_{2r}(x_k)\times [t_k,1))\subset \{(x,t):|(x-x_k)\cdot e_k|\leq 4(d\delta)^{\frac12}r\}.
\end{equation*}
Up to extracting a further subsequence, we may assume that $e_k\to e$ for some $e\in \mathbb{S}^{d-1}$. Letting $k\to \infty$, it follows that
\begin{equation*}
    S_n \cap (B_{2r}(x_*)\times (t_*,1))\subset \{(x,t):|(x-x_*)\cdot e|\leq 4(d\delta)^{\frac12}r\}.
\end{equation*}
Moreover, since the infimum \eqref{eq:mini problem sigmad2 pf} is not attained, we must have
\begin{equation*}
    S_n \cap (B_{r}(x_0)\times \{t_*\})=\emptyset,
\end{equation*}
and by the infimality of $t_*$, we again have \eqref{eq:cleanbelowpf12}. Therefore, we conclude similarly to the first case, that
\begin{equation*}
    S_n\cap (B_{r}(x_0)\times (-1,1))=S_n\cap (B_{r}(x_0)\times (t_*,1))\subset S_n\cap (B_{2r}(x_*)\times (t_*,1)) \subset \{x:|(x-x_*)\cdot e|\leq4(d\delta)^{\frac12}r\}.
\end{equation*}
\end{proof}

We now recall some basic facts about porosity and Hausdorff dimension.
\begin{defn}Given $\vep>0$  we say that $S\subset \R^d$ is {\it $(1-\vep)$--porous} if, for every $x\in S$, there exists $r_0>0$ such that, for every $r<r_0$ there exists $y\in B_r(x)$ that satisfies 
\begin{equation*}
    B_{(1-\vep)r/2}(y) \subset B_r(x) \setminus S.
\end{equation*}    
\end{defn}
The following is a classical estimate for porous sets \cite[Thm. 3.8.1, Rem. 3.8.2]{salli}.
\begin{lem}[Porosity bound] \label{lem:porosity} Let $S \subset \R^d$ be a $(1-\vep)$--porous set for some $\vep>0$. Then there is $A=A(d)>0$ such that
\begin{equation*}
    \dim_{\mathcal{H}}(S)\leq d-1+\frac{A}{\ln(\vep^{-1})}
\end{equation*}
\end{lem}
With this result in hand, we may finally prove our main estimate for $\Sigma_d^2$.

\begin{prop} \label{prop:-t dimension bound quadratic} Let $w:Q_1 \to \R$ be a bounded solution to \eqref{eq:obstacle intro} with singular set $\Sigma$. Then
\begin{equation*}\operatorname{dim}_{\operatorname{par}}(\Sigma_{d}^2)\leq d-1.\end{equation*}    
\end{prop}
\begin{proof}
By Lemma \ref{lem:pi_xSn cleaning}, letting $\delta\in (0,\frac12)$ be small enough that $\vep=8(d\delta)^{\frac12}<1$, it follows that the set $\pi_x(S_n)$ is $(1-\vep)-$porous. By Lemma \ref{lem:porosity}, we have
\begin{equation*}
\dim_{\mathcal{H}}(\pi_x(S_n))\leq d-1+\frac{A}{\ln(\vep^{-1})}.    
\end{equation*}
By Lemma \ref{lem: sign rigidity}, we have 
\begin{equation*}
 \Sigma_{d}^2=\cup_{n=1}^{\infty}S_n.   
\end{equation*}
Therefore, we have
\begin{equation*}
    \dim_{\mathcal{H}}(\pi_x(\Sigma_{d}^2))\leq d-1+\frac{A}{\ln(\vep^{-1})}.
\end{equation*}
Letting $\delta \to 0$, we get
\begin{equation*}
    \dim_{\mathcal{H}}(\pi_x(\Sigma_{d}^2))\leq d-1,
\end{equation*}
and we conclude as in Proposition \ref{prop: Hausdorff time dependent} by an application of Lemma \ref{lem: FRS cleaning}.
\end{proof}

\section{Superquadratic second blow-up in the top stratum: the set $\Sigma_d^{\geq 3}$}\label{sec:top stratum superquadratic}
 In this section, we turn our attention to the set $\Sigma_{d}^{\geq3}$  of points with at least a cubic rate of convergence to the blow-up profile $p_2=-t$.  Since $p_2(x,t)>0$ for all $t<0$ and $(r^{-2}w_r- p_2)\to 0$, the set $\{w_r>0\}$ converges, in the limit, to the full space $\R^d \times (-\infty,0)$. The key observation is that, if the convergence is sufficiently fast, we expect $(r^{-2}w_r- p_2)$ to behave like a solution to the heat equation, since \begin{equation*}
    \cH (r^{-2}w_r-p_2)=\chi_{\{w_r=0\}}.
\end{equation*}
 As it turns out, the cubic rate of convergence is sufficient to disrupt the parabolic scaling in the blow-up limit and obtain a full Taylor expansion by a Campanato-type iteration. In particular, we will be able to uniquely identify the second blow-up profile. Our Lemmas \ref{lem: compactness caloric} and \ref{lem poly d} below should be compared to \cite[Lem. 13.6, Prop. 13.10]{figalli}, with caloric polynomials here playing a similar role to the Ansatz manifold of approximate profiles that was used to study a certain subset of $\Sigma_{d-1}$ in the melting case.
  
  We will then study the Taylor expansion systematically, and use it to obtain an upgraded polynomial cleaning estimate, which will lead to the parabolic dimension bound through a careful application of the sharp \L ojasiewicz inequality for harmonic polynomials.

We begin by recording the fact that, at points of $\Sigma_d^{\geq3}$, the convergence rate to the blow-up profile is at least cubic.
\begin{lem}\label{lem:cubic convergence} Let $w:Q_2 \to \R$ be a bounded solution to \eqref{eq:obstacle intro}. Then there is a constant $C=C(d,\|w\|_{\infty})>0$ such that, for every $(x_0,t_0)\in Q_1 \cap \Sigma_{d}^{\geq 3}$ and every $r\in (0,1)$,
\begin{equation*}
 \|w-(t_0-t)\|_{L^{\infty}(Q_r^-(x_0,t_0))}\leq Cr^3.   
\end{equation*}    
\end{lem}
\begin{proof}
 By Lemma \ref{lem energy}, we have, for small $r$,
\begin{equation*}
    \|w-(t_0-t)\|_{L^{\infty}(Q_r^-(x_0,t_0))}\leq C|Q_{2r}^-|^{-1/2}\|w-(t_0-t)\|_{L^{2}(Q_{2r}^-(x_0,t_0))}.
\end{equation*}
On the other hand, applying Lemma \ref{lem:comparability sigma d} with $\gamma=3$, we obtain
\begin{equation*}
\|w-(t_0-t)\|_{L^{2}(Q_{2r}^-(x_0,t_0))}\leq C|Q_r^-|^{\frac12}r^3, 
\end{equation*}
which yields the claim.
\end{proof}

Given $\beta>0$, $M_0>0$, we define
\begin{equation*}
    \Omega^{\beta}_{M_0}:=\{(x,t)\in \R^d \times (-\infty,0]: |x|<M_0|t|^{\frac{1}{2+\beta}}\}.
\end{equation*}
We now obtain an elementary compactness estimate for caloric functions on the cuspidal domains $\Omega^{\beta}_{M_0}$, which will later be the basic iterative tool to approximate $w+t$ by caloric polynomials.
\begin{lem}\label{lem: compactness caloric} For given $k\in \mathbb{N}$, $ \alpha, \beta\in (0,1)$, and $\vep_0 >0$, there exists a positive constant $M_0=M_0(d,\alpha,\beta, k,\vep_0)$ such that the following holds. Suppose that $u: Q_{2^{M_0}}^-\to \R$ satisfies
\begin{equation} \label{wq-dlo0spkc}\left(\int_{\Omega_{M_0}^{\beta}\cap Q_{1}^-}|u_{2^{m}}|^2+|\grad u_{2^{m}}|^2+|\partial_t(u_{2^m})|^2dxdt\right)^{\frac12}\leq (2^m)^{k+\alpha}, \quad 0\leq m \leq M_0, 
\end{equation}
where $u_{2^m}(x,t):=u(2^{m}x,4^{m}t)$. Assume also that
\begin{equation} \label{poasxkas,p13}
    \cH u=0 \; \text{ in } \;\Omega^{\beta}_{M_0}\cap Q_{2^{M_0}}^-,
\end{equation}
and for any caloric polynomial $p$ of degree $\leq k$,\begin{equation}\label{orth cond}    \int_{-2}^{-1}\int_{B_{M_0}}(up G)(x,t) dxdt=0. 
\end{equation}
 Then we have
\begin{equation*}
    \left( \int_{\Omega^{\beta}_{M_0}\cap Q_1^-}u^2\right)^{\frac12}\leq \vep_0.
\end{equation*}
\end{lem}
\begin{proof} Proceeding by contradiction, assume that there is a sequence $\{u_n\}$ such that
\[\left( \int_{\Omega^{\beta}_{n}\cap Q_1^-}u_n^2\right)^{\frac12}> \vep_0,\]
    \begin{equation*}\left(\int_{\Omega_{n}^{\beta}\cap Q_{1}^-}|(u_{n})_{2^{m}}|^2+|\grad (u_{n})_{2^{m}} |^2+|\partial_t(u_{n})_{2^{m}}|^2dxdt\right)^{\frac12}\leq (2^m)^{k+\alpha}, \quad 0\leq m \leq n,    
\end{equation*}
\[\cH u_n=0\;\; \text{ in }\; \Omega^{\beta}_{n},  \quad\int_{-2}^{-1}\int_{B_{n}}u_npG(x,t)dxdt=0\]
for any caloric polynomial $p$ of degree $\leq k$. Then, noting that $\Omega^{\beta}_{n} \uparrow \R^d \times (-\infty,0)$ as $n\to \infty$, along a subsequence $u_n\to u$ in $L^{2}_{\text{loc}}(\R^{d}\times (-\infty,0])$, where $u$ is caloric with
\begin{equation} \label{lvasd,psax}\left(\int_{Q_{1}^-}|u_{2^m}|^2+|\grad (u_{2^m}) |^2+|\partial_t(u_{2^m})|^2dxdt\right)^{\frac12}\leq (2^m)^{k+\alpha}, \quad m\geq0, \end{equation}
 and 
 \begin{equation} \label{ctraoedco5321as}\left( \int_{ Q_1^-}u^2\right)^{\frac12}\geq \vep_0, \quad \int_{-2}^{-1}\int_{\R^d}upG(x,t)dxdt=0.      
 \end{equation}
 for every caloric polynomial $p$ of degree $\leq k$. By Liouville's theorem, \eqref{lvasd,psax} implies that $u$ is a caloric polynomial of degree $\leq k$. But then \eqref{ctraoedco5321as} yields a contradiction by taking  $p=u$.
\end{proof}
The following energy estimate is a slight variation of \cite[Lem. 13.7]{figalli}, and has the same proof, which consists of explicitly constructing a suitable cutoff function that is supported on $\Omega^{\beta}_R$. It is important to observe that this is only possible because $\beta>0$, so that the set $\Omega^{\beta}_R$ approximates the full space $\R^d \times (-\infty,0)$ at small parabolic scales. 
\begin{lem} \label{lem: energy cone} Let $\beta \in (0,1)$ and $\theta\in (1/2,1)$. There exist constants $R_0>0$, $C>0$, depending on $d$, $\beta$, and $\theta$, such that, for any caloric function $u$ on $\Omega^{\beta}_R\cap Q^-_2$, we have
\[\int_{\Omega^{\beta}_{\theta R}\cap Q_1^-}|\nabla u|^2+|u_t|^2 \leq C \int_{\Omega^{\beta}_{R}\cap Q_2^-}u^2, \quad R\geq R_0.\]
\end{lem}

Combining the localization and energy estimates above, we can now implement a Campanato-type iteration which approximates $w+t$ near points of $\Sigma_d^3$ by caloric polynomials of arbitrarily high degree in a one-sided time region.

\medskip

\begin{lem}[Campanato iteration] \label{lem poly d} Assume that $(0,0)\in \Sigma_{d}^{\geq3}$. Then for  every $k\geq 3$, $\alpha \in (0,1)$ and $\beta\in (\frac12,1)$, there is a caloric polynomial $p$ of parabolic degree $\leq k$ and $r_0=r_0( d, \|w\|_{L^{\infty}}, k, \alpha,\beta)\in (0,1)$ such that
\begin{equation} \label{poly approx -t}\|w+t- p\|_{L^{\infty}(B_r\times (-r^2,- r^{2+\beta}))}\leq r^{k+\alpha},  \quad r\leq r_0.\end{equation}

\end{lem}
\begin{proof}
 We fix $k\geq 3$, $\alpha$ and $\beta$, and let $M_0>10$ be a large constant to be chosen later. By definition, we have
\begin{equation} \label{ombkoqd2w1k}
    (rx,r^2t) \in \Omega_{K}^{\beta} \iff |x|< Kr^{-\frac{\beta}{2+\beta}}|t|^{\frac{1}{2+\beta}} \iff (x,t)\in  \Omega_{r^{-\beta/(2+\beta)}K}^{\beta}, \quad r>0, \;\;K>0.
\end{equation} 
Thus, if we let
\[\theta =2^\frac{-\beta}{2+\beta}<1,\]
then
\begin{equation} \label{doubling2qk3wqdx}
 (x,t)\in  \Omega_{M_0}^{\beta} \iff \left(x/2,t/4\right) \in \Omega_{\theta M_0}^{\beta}. 
\end{equation}
By Lemma \ref{lem:cubic convergence}, there exists a small $r_1\in (0, 1)$, depending on $d$ and $M_0$, such that the function
\begin{equation} \label{udefomwe21ee}u(x,t):=\frac{(w+t)(r_1x,r_1^2t)}{r_1^2}\end{equation}
satisfies
\begin{equation} \label{basecasedq0-eodq1} \left(\int_{ Q_{1}^-}|u(2^{-m}x,4^{-m}t)|^2dxdt\right)^{\frac12}\leq (2^{-m})^{k+\alpha}, \quad 0\leq m \leq M_0+1.  \end{equation}
By Lemma \ref{lem:cubic convergence}, there exists a constant $c_0>0$ such that
\begin{equation} \label{supwdqwxo0k}
   \Omega_{c_0}^{\beta} \cap Q_1^- \subset \{w>0\}.
\end{equation}
Thus, by requiring $r_1^{\frac{\beta}{2+\beta}} \leq \frac{c_0}{ M_0} $ so that $\Omega_{M_0}^{\beta}\subset   \Omega_{r_1^{-\beta/(2+\beta)}c_0}^{\beta} $, it follows from \eqref{ombkoqd2w1k}, \eqref{udefomwe21ee},  and \eqref{supwdqwxo0k} that
\begin{equation*}
    \cH u=0, \quad (x,t)\in \Omega^{\beta}_{M_0} \cap Q_1^-.
\end{equation*}
A posteriori, using \eqref{ombkoqd2w1k} once more with $r=2^{-m}$, we have $(2^{-m}x,4^{-m}t)\in \Omega^{\beta}_{M_0}$ whenever $(x,t)\in \Omega^{\beta}_{M_0} \subset \Omega^{\beta}_{r^{-\beta/(2+\beta)} M_0} $, and thus
\begin{equation} \label{caloricscalingsasxopa}
    \cH (u(2^{-m}x,4^{-m}t))=0, \quad (x,t)\in \Omega^{\beta}_{M_0} \cap Q_{2^m}^-, \quad m\geq 0.
\end{equation}
We now prove that there exists a sequence of caloric polynomials $\{p_m\}_{m=0}^{\infty}$ of degree $\leq k$ such that, for any $m\in\mathbb{N}$,
\begin{equation} \label{induction hyp s201-dk}\left(\int_{ \Omega^{\beta}_{M_0}\cap Q_{1}^-}|(u-p_m)(2^{-m}x,4^{-m}t)|^2dxdt\right)^{\frac12}\leq (2^{-m})^{k+\alpha}. \end{equation}
We proceed by induction, where we will show that \eqref{induction hyp s201-dk} holds for $1\leq m \leq j$ for every $j\in\mathbb{N}$. Observe that, by \eqref{basecasedq0-eodq1}, \eqref{induction hyp s201-dk} holds up to $j =  \lceil  M_0 \rceil$ with $p_m:=0$ for $m\leq \lceil M_0 \rceil$. Assume that \eqref{induction hyp s201-dk} holds for some $j\geq \lceil M_0\rceil$. Note that $K_0=B_{1/2} \times (-1,-1/2)\subset \Omega_{M_0}^{\beta}$, so for $0\leq m\leq j-1$, letting $2^{-(m+1)}K_0$ $=\{(2^{-(m+1)}x,4^{-(m+1)}t):(x,t)\in K_0\}$, we have from \eqref{induction hyp s201-dk} that
\[  \left(\fint_{2^{-(m+1)}K_0}|u-p_m|^2dxdt\right)^{\frac12}+ \left(\fint_{2^{-(m+1)}K_0}|u-p_{m+1}|^2dxdt\right)^{\frac12}\leq C(2^{-m})^{k+\alpha},\]
so that
\[\left(\fint_{2^{-(m+1)}K_0}|p_{m+1}-p_m|^2dxdt\right)^{\frac12}\leq C(2^{-m})^{k+\alpha},\]
namely
\[\left(\int_{K_0}|(p_{m+1}-p_m)(2^{-(m+1)}x,4^{-(m+1)}t)|^2dxdt\right)^{\frac12}\leq C(2^{-m})^{k+\alpha}. \]
Since the space of polynomials of degree $\leq k$ is finite dimensional, this implies that the coefficient vectors $a_m^{i}$ of $p_{m}$ of parabolic degree $i$ satisfy
\begin{equation}\label{coefqewosdmo}|a_m^i-a_{m+1}^i|\leq C(2^{-m})^{k+\alpha-i}, \quad 0\leq m\leq j-1,\quad 0\leq i \leq k.\end{equation}
We define $p_{j+1}$ through $L^2_{G}$--orthogonal projection onto the space of caloric polynomials of degree $\leq k$:
\begin{equation} \label{leastsquares}p_{j+1}:=\argmin_{p\text{ caloric, }\deg(p)\leq k}\int_{-2}^{-1}\int_{B_{M_0}}(u-p)^2(2^{-j}x,4^{-j}t)G(x,t)dxdt. \end{equation}
We will now show that \eqref{induction hyp s201-dk} for $m=j+1$ holds with $p_{j+1}$.

Summing \eqref{coefqewosdmo} from $m$ to $j-1$, we obtain
\begin{equation} \label{coefqewosdmo2}|a_m^i-a_{j}^i|\leq C (2^{-m})^{k+\alpha-i}, \quad 0\leq m\leq j-1, \quad 0\leq i \leq k.\end{equation}
Thus, in particular, \[\left(\int_{\Omega^{\beta}_{M_0}\cap Q_1^-}|(p_m-p_j)(2^{-m}x,4^{-m}t)|^2dxdt\right)^{\frac12}\leq C(2^{-m})^{k+\alpha}, \quad 0\leq m\leq j. \]
 Since \eqref{induction hyp s201-dk} holds for $0\leq m\leq j$, we have
\begin{equation}\label{12-deklwxoqk3e0}    \int_{\Omega^{\beta}_{M_0}\cap Q_1^-}|(u-p_j)(2^{-m}x,4^{-m}t)|^2dxdt\leq C(2^{-m})^{2(k+\alpha)}, \quad 0\leq m\leq j,
\end{equation}
or, equivalently, recalling \eqref{ombkoqd2w1k}, 
\begin{equation*}
   \int_{\Omega^{\beta}_{\theta^{-l}M_0}\cap Q_{2^{l}}^-}|(u-p_j)(2^{-j}x,4^{-j}t)|^2dxdt\leq C2^{l\gamma}(2^{-j})^{2(k+\alpha)}, \quad 0\leq l\leq j, 
\end{equation*}
where $\gamma=\gamma(d,k)$. Now we decompose above integral to estimate with the weight $G(x,t)$. Observing that  $B_{M_0} \times (-2,-1) \subset \Omega^{\beta}_{M_0}$,  we get 
\[\int_{B_{M_0}\times (-2,-1)\cap (Q_{2^l}^-\backslash Q_{2^{l-1}}^-)}|(u-p_j)(2^{-j}x,4^{-j}t)|^2dxdt\leq C2^{l\gamma}(2^{-j})^{2(k+\alpha)} ,\quad 0\leq l\leq j, \]
so, since $G(x,t)\leq Ce^{-2^{l}}$ on  $(Q_{2^l}^-\backslash Q_{2^{l-1}}^-)\cap\{t\in (-2,-1)\},$ and $2^{j}\geq 2^{M_0}>M_0,$ summing up the integrals over $0\leq l\leq j$, we have 
\[\left(\int_{B_{M_0}\times (-2,-1)}|(u-p_j)(2^{-j}x,4^{-j}t)|^2G(x,t)dxdt\right)^{\frac12}\leq C(2^{-j})^{k+\alpha}, \]
where $C$ remains independent of $M_0$. By minimality of $p_{j+1}$,
\[\left(\int_{B_{M_0}\times (-2,-1)}|(u-p_{j+1})(2^{-j}x,4^{-j}t)|^2G(x,t)dxdt\right)^{\frac12}\leq C(2^{-j})^{k+\alpha}. \]
Applying the triangle inequality to the last two inequalities,
\[\left(\int_{B_{M_0}\times (-2,-1)}|(p_j-p_{j+1})(2^{-j}x,4^{-j}t)|^2dxdt\right)^{\frac12}\leq C(2^{-j})^{k+\alpha}. \]
Therefore, using again the finite dimensionality,
\[|a^i_{j}-a^i_{j+1}|\leq C  (2^{-j})^{k+\alpha-i}, \quad 0\leq i \leq k. \]
which together with \eqref{coefqewosdmo2} implies
\[|a^i_{j+1}-a^i_{m}|\leq C  (2^{-m})^{k+\alpha-i}, \quad 0\leq m \leq j, \quad  0\leq i \leq k.\]
Thus, we infer from \eqref{12-deklwxoqk3e0} that
\begin{equation} \label{-1qldo-wskqpax}\left(\int_{\Omega^{\beta}_{M_0}\cap Q_1^-}|(u-p_{j+1})(2^{-m}x,4^{-m}t)|^2dxdt\right)^{\frac12}\leq C_1(2^{-m})^{k+\alpha}, \quad 0\leq m\leq j. \end{equation}
In view of \eqref{-1qldo-wskqpax}, Lemma \ref{lem: energy cone} yields
\begin{multline}\left(
\int_{\Omega^{\beta}_{\theta M_0}\cap Q_1^-}(4^{-m}|\grad (u-p_{j+1})(2^{-m}x,4^{-m}t)|^2 +16^{-m}|\partial_t(u-p_{j+1})(2^{-m}x,4^{-m}t)|^2)dxdt\right)^{\frac12} \\
\leq C_2  (2^{-m})^{k+\alpha}, \quad 0\leq m\leq j,
\end{multline}
where $C_1, C_2$ are both independent of $M_0$ and $j$. Letting
\[\tilde{u}(x,t)=\frac{(u-p_{j+1})(2^{-j}x,4^{-j}t)}{(C_1+C_2)(2^{-j})^{k+\alpha}},\]
we obtain
\begin{equation*}\left(
\int_{\Omega^{\beta}_{\theta M_0}\cap Q_{1}^-}(|\tilde u_{2^l}|^2+|\grad \tilde{u}_{2^{l}}|^2 +|\partial_t\tilde{u}_{2^{l}}|^2)\right)^{\frac12} \leq 2^{l(k+\alpha)}, \quad 0\leq l \leq j.
 \end{equation*}
Since $j\geq  M_0 $, it follows that $\tilde{u}$ satisfies \eqref{wq-dlo0spkc}. Moreover, \eqref{caloricscalingsasxopa} shows that $\tilde{u}$ satisfies \eqref{poasxkas,p13}, whereas \eqref{orth cond} is just the optimality condition of \eqref{leastsquares}. Thus, taking $\vep_0=2^{-(k+\alpha+(d+2)/2)}/(C_1+C_2)$ in Lemma \ref{lem: compactness caloric}, if $M_0$ is chosen sufficiently large, we have
\begin{equation*}\left(
\int_{\Omega^{\beta}_{\theta M_0}\cap Q_{1}^-}(u-p_{j+1})^2(2^{-j}x,4^{-j}t)dxdt\right)^{\frac12} 
\leq 2^{-(d+2)/2}(2^{-(j+1)})^{k+\alpha}.
\end{equation*}
Changing the variables and recalling \eqref{doubling2qk3wqdx}, we obtain
\begin{equation*}\left(
\int_{\Omega^{\beta}_{M_0}\cap Q_{1}^-}(u-p_{j+1})^2(2^{-(j+1)}x,4^{-(j+1)}t)dxdt\right)^{\frac12} 
\leq(2^{-(j+1)})^{k+\alpha}.
\end{equation*}
This completes our induction and shows that \eqref{induction hyp s201-dk} holds for all $m\in \mathbb{N}$. 

\medskip 

Observe that, in view of \eqref{coefqewosdmo2}, the polynomials $p_{j}$ converge locally uniformly, as $j\to \infty$ to a caloric polynomial $p_{\infty}$ of degree $\leq k$. Letting $j\to \infty$ in \eqref{-1qldo-wskqpax}, we obtain
\begin{equation*}\left(\int_{\Omega^{\beta}_{M_0}\cap Q_1^-}|(u-p_{\infty})(2^{-m}x,4^{-m}t)|^2dxdt\right)^{\frac12}\leq C(2^{-m})^{(k+\alpha)}, \quad m\geq 0.\end{equation*}
Thus, 
\begin{equation} \label{L2 approx 2e-1qo}\int_{\Omega^{\beta}_{M_0}\cap Q_1^-}|(u-p_{\infty})(rx,r^2t)|^2dxdt\leq Cr^{2(k+\alpha)}, \quad r\in [0,1].\end{equation}
Finally, let $\hat p_{\infty}$ and $\hat M_0$ be such that \eqref{L2 approx 2e-1qo} holds for $\hat \beta:=\frac{1}{2}(1+\beta)>\beta$. Then, if $r$ is sufficiently small, depending only on $\beta$, we have $B_{2r}\times [-(2r)^2,-r^{2+\beta}]\subset \Omega_{\hat M_0}^{\hat \beta}$. Therefore, the map $z(x,t)=(u-\hat{p}_{\infty})(rx,r^2t)$ is caloric in $D=B_2\times(-4,-r^{\beta}]$ and satisfies $\|z\|_{L^{2}(D)}\leq Cr^{k+\alpha}$, so \eqref{poly approx -t} follows by interior parabolic regularity.
\end{proof} 

We now show that at points of $\Sigma_{d}^{\geq 3}$, the solution admits a genuine one-sided Taylor expansion in space–time with caloric polynomial coefficients. This yields a uniquely defined higher-order profile at each such point, which will be crucial in the analysis to follow.
\begin{prop}[One-sided Taylor expansion at $\Sigma_d^{\geq 3}$]\label{prop:taylor} Let $w: Q_2 \to \R$ be a bounded solution to \eqref{eq:obstacle intro}. Then, for every $X=(x_0,t_0)\in \Sigma_d^{\geq3}$, there exists a unique sequence $\{p^{(j)}_X\}_{j=3}^{\infty}$ of caloric polynomials such that each $p_j$ is parabolically homogeneous of degree $j$, and one has, for any $\beta\in (\frac12,1)$ and $k\geq 3$, the expansion
\begin{equation} \label{eq:taylor}
  w(x-x_0,t-t_0)=t_0-t+\sum_{j=3}^{k}p^{(j)}_X(x,t) + E^{(k)}_{X}(x,t),  \quad (x,t)\in Q_1,\quad t<t_0, \quad |t_0-t|\geq |x-x_0|^{2+\beta},
\end{equation}
where, for some constant $C=C(d,k,\beta,\|w\|_{\infty})>0$,
\begin{equation*}
|E^{(k)}_X(x,t)|\leq C_k(|x-x_0|^{k+1}+|t-t_0|^{\frac{k+1}{2}}).
\end{equation*}  
Equivalently, we have
\begin{equation} \label{eq:taylor refble}
     \left\|w-(t_0-t)-q_{X}^{(k)}\right\|_{L^{\infty}(B_r(x_0) \times [t_0-r^2,t_0-r^{2+\beta}])}
     \leq C_kr^{k+1}, \quad q^k_X:=\sum_{j=3}^{k}p_{X}^{(j)}, \quad r\in (0,1).
\end{equation}
\end{prop}
\begin{proof}Up to translation and rescaling, we may assume that $X=(x_0,t_0)=(0,0)$. For each $k\geq3$,  Lemma \ref{lem poly d} applied with $\alpha=\frac12$ implies that there exists a polynomial $q_{X}^{(k)}$ of parabolic degree $\leq k$ such that
\begin{equation} \label{eq:poly pk pf}
    \|w+t-q_{X}^{(k)}\|_{L^{\infty}(B_r \times [-r^2,-r^{2+\beta}])}\leq r^{k+\frac12}, \quad r\leq r_{0,k},
\end{equation}
where $r_{0,k}\in(0,1)$ depends on $k$, $d$, $\|w\|_{\infty}$, and $\beta$. With no loss of generality, we may also assume that the sequence $\{r_{0,k}\}_{k=1}^{\infty}$ is decreasing. Letting $q_X^{(2)}:=0$, we see by Lemma \ref{lem:cubic convergence} that \eqref{eq:poly pk pf} also holds for $k=2$ and an appropriate choice of $r_{0,2}$. Setting $r=r_{0,k}$ in \eqref{eq:poly pk pf}, since all norms in the space of polynomials of parabolic degree $\leq k$ are comparable, we infer that 
\begin{equation} \label{eq:pk bdd pf}
    \|q_{X}^{(k)}\|_{L^{\infty}(Q_1)}\leq C_k, 
\end{equation}
where $C_k$ depends on $k$, $d$, $\|w\|_{\infty}$ and $r_{0,k}$. Using the triangle inequality in \eqref{eq:poly pk pf}, we get
\begin{equation*}
 \|q_{X}^{(k)}-q_{X}^{(k+1)}\|_{L^{\infty}(B_r \times [-r^2,-r^{2+\beta}])}\leq r^{k+\frac12}, \quad r\leq \min(r_{0,k},r_{0,k+1})=r_{0,k+1}. 
\end{equation*}
This implies that $q_{X}^{(k)}$ and $q_{X}^{(k+1)}$ must differ by a homogeneous term of degree $k+1$, that is,
\begin{equation} \label{eq:recursive pk pf}
    q_{X}^{(k+1)}(x,t)=q_{X}^{(k)}(x,t)+p^{(k+1)}_{X}(x,t),
\end{equation}
where each $p_X^{j}$ is a (possibly zero) parabolically homogeneous polynomial of degree $j$. We therefore have
\begin{equation*}
    q_{X}^{(k)}=\sum_{j=3}^{k}p_X^{(j)},
\end{equation*}
and, by \eqref{eq:poly pk pf}, \eqref{eq:pk bdd pf}, and \eqref{eq:recursive pk pf},
\begin{multline} \label{eq:taylor pf refble}
     \|w+t-q_{X}^{(k)}\|_{L^{\infty}(B_r \times [-r^2,-r^{2+\beta}])}\\
     \leq \|w+t-q_{X}^{(k+1)}\|_{L^{\infty}(B_r \times [-r^2,-r^{2+\beta}])}+\|p^{(k+1)}_X\|_{{L^{\infty}(B_r \times [-r^2,-r^{2+\beta}])}}\\
    \leq r^{k+1+\frac12}+Cr^{k+1} \leq C'r^{k+1} \quad 0<r\leq r_{0,k+1}.
\end{multline}
Up to increasing the value of $C$, depending only on $\|w\|_{\infty}$, we see that \eqref{eq:taylor pf refble} holds for $r\in (0,1)$, which shows \eqref{eq:taylor refble}.  Finally, uniqueness of the sequence $\{q^{(k)}_X\}_{k\geq 3}$ (and thus of $\{p^{(k)}_X\}_{k\geq 3}$) follows directly from \eqref{eq:taylor refble} and the triangle inequality.
\end{proof}

The next proposition shows that the polynomial $q$ obtained as a second blow-up in Proposition \ref{prop:second blow-up} coincides, up to normalization, with the first nontrivial term in the expansion of Proposition \ref{prop:taylor}, and that the points of $\Sigma_d^{\infty}$ are precisely those where all such higher-order coefficients vanish.
\begin{prop}[Characterization of the second blow-up and $\Sigma_d^{\infty}$]\label{prop:second blow-up char taylor} Let $w:Q_2\to \R$ be a bounded solution to \eqref{eq:obstacle intro}, assume that $(0,0)\in \Sigma_d^{\geq3}$, and let $\{p_{(0,0)}^{(j)}\}_{j=3}^{\infty}$ be the polynomials given by Proposition \ref{prop:taylor}. Then the following holds:
\begin{itemize}
    \item[(i)] If $(0,0)\in \Sigma_d^{k}$ for some $k\in \{3,4,\ldots\}$, then for any $\gamma>k$, the polynomial $q$ of Proposition \ref{prop:second blow-up} is independent of the subsequence and of $\gamma$, the full sequence converges to $q$, and
    \begin{equation*}
    q=\frac{1}{H^{\frac12}(1,p^{(k)}_{(0,0)})}p_{(0,0)}^{(k)},  
    \end{equation*} where $p^{(k)}_{(0,0)}$ is the first non-zero element of the sequence $\{p_{(0,0)}^{(j)}\}_{j=3}^{\infty}$. 
    \item[(ii)] Moreover, we have
\begin{equation} \label{eq:sigma infty char}
    (0,0)\in \Sigma_d^{\infty} \quad \text{if and only if } \quad p^{(j)}_{(0,0)}\equiv 0 \text{ for all } j\geq 3.    
\end{equation}
\end{itemize}    
\end{prop}
\begin{proof} Assume first that $(0,0)\in \Sigma_d^{k}$. By  \eqref{eq:H frequency bounds upper}, Lemma \ref{lem: H L2 comp}, and Lemma \ref{lem energy}, we have  
\begin{equation} \label{eq:enoijultlessim}
 \|w+t\|_{L^{\infty}(Q_r^-)}\lesssim r^{k}, \quad r\ll1.   
\end{equation} 
By, \eqref{eq:taylor}, this implies that $p_{(0,0)}^{(j)}\equiv 0$ for $j<k$.  Similarly, by \eqref{eq:H frequency bounds lower} and Lemma \ref{lem: H L2 comp},
\begin{equation*}
  |Q_r|^{-\frac12}\|w+t\|_{L^{2}(Q_r^-)}\gg  r^{k+\frac12}, \quad r\ll1,   
\end{equation*}
which, by \eqref{eq:enoijultlessim} and \eqref{eq:taylor}, implies that $p_{(0,0)}^{(k)} \not \equiv0$.

Given $r\in (0,1)$, let $u_r(x,t)=(w+t)_r(x,t)=w(rx,r^2t)+r^2t$.
From \eqref{eq:taylor}, we have, for any fixed $R>0$, as $r\downarrow 0$,
\begin{equation} \label{eq:r-kurconvpf1bc}
       \|r^{-k}u_r-p^{(k)}_{(0,0)}\|_{L^{\infty}(B_R\times (-R^2,-1/R^2))}\leq Cr=o(1).
   \end{equation}   
   Letting $r\downarrow 0$, we infer that
   \[
   r^{-k}u_r \xrightarrow{r\downarrow 0} p_{(0,0)}^{(k)}\quad  \text{ locally uniformly in} \quad \R^d\times (-\infty,0).\]
   In particular, recalling the polynomial growth estimate \eqref{eq:energy urk blowup pf}, we have
   \begin{equation}\label{eq:energyconvpf2bc}
       H^{\frac12}(1,\zeta_{r} r^{-k}u_{r}) \xrightarrow{r\downarrow 0}H^{\frac12}(1,p_{(0,0)}^{(k)}).
   \end{equation}
   On the other hand, by \eqref{eq:H frequency bounds upper} and Proposition \ref{prop:second blow-up}(a),
   \begin{equation*}
       \|r_j^{-k}u_{r_j}-r_j^{-k}H^{\frac12}(1,\zeta_{r_j} u_{r_j})q\|_{L^{\infty}(B_1\times (-1,-1/2))}= r_j^{-k}H^{\frac12}(r_j,\zeta u)o(1)\leq o(1), \quad j\geq 1.
   \end{equation*} 
Thus, by \eqref{eq:r-kurconvpf1bc} and \eqref{eq:energyconvpf2bc}, we infer that
\begin{equation*}
    H(1,p_{(0,0)}^{(k)})^{\frac12}q=p_{(0,0)}^{(k)},
\end{equation*}
which proves part (i). 

Part (i) implies, in particular, that if $(0,0)\notin \Sigma_d^{\infty}$, then the sequence $\{p_{(0,0)}^{j}\}_{j\geq3}$ is not identically zero, which shows one implication of (ii). Conversely, if $(0,0)\in \Sigma_d^{\infty}$, then, as before, using Lemma \ref{lem:comparability sigma d}, we conclude that \eqref{eq:enoijultlessim} holds for all $k\geq 3$, and thus $p_{(0,0)}^{(k)}\equiv 0$ for all $k\geq 3$. 

\end{proof}

We next upgrade the expansion of Proposition \ref{prop:taylor} to a quantitative stability statement along  $\Sigma_d^{\geq 3}$: nearby base points $X_1$, $X_2$ have jets $q_{X_i}^{(k)}$
  whose coefficients vary Lipschitz-continuously as in \eqref{prop:second blow-up char taylor}, and, crucially, after re-centering by the corresponding space–time translation, estimate \eqref{eq: 2nd blow up variation} shows that the truncations agree up to an $O(r^{k+1})$ error on a common parabolic cylinder of radius $r \simeq |x_1-x_2|$.
 \begin{lem} \label{lem: -t Taylor continuous dependence} Let $w:Q_2 \to \R$ be a bounded solution to \eqref{eq:obstacle intro}. Let us fix $k\geq 3$. For $X=(x,t)\in \Sigma_{d}^{\geq3}\cap Q_1$,
let
\begin{equation*}
    q^{(k)}_X=\sum_{j=3}^{k}p^{(j)}_X,
\end{equation*}
where $p^j_X$ are the unique polynomials of Proposition \ref{prop:taylor}, and let $\| \cdot \|$ be any norm in the space of polynomials of degree $\leq k$. Then there exists a constant $C>0$, depending only on $d$, $k$, and $\|w\|_{\infty}$,  such that
\begin{equation} \label{eq: 2nd blow up lip}
\|q_{X_1}^{(k)}-q_{X_2}^{(k)}\|\leq  C|x_1-x_2|, \quad X_1=(x_1,t_1), \quad X_2=(x_2,t_2), \quad X_1,X_2 \in  \Sigma_{d}^{\geq 3}\cap Q_1.  
\end{equation}
Furthermore, setting $\delta X=(\delta x, \delta t):=(x_1-x_2,t_1-t_2)$ and $r:=4|x_1-x_2|$, one has

\begin{equation} \label{eq: 2nd blow up variation}
    \|q_{X_1}^{(k)}-(q_{X_2}^{(k)}(\cdot +\delta X)-\delta t)\|_{L^{\infty}(Q_r)}\leq Cr^{k+1}.
\end{equation}
\end{lem}
\begin{proof} Observe first that, from \eqref{eq:pk bdd pf}, $\|q_{X}^{(k)}\|$ is bounded, uniformly in $X$. With no loss of generality, we may then assume that $|x_1-x_2|<r_0/100$. Moreover, by Lemmas \ref{lem:two-sided cleaning} and \ref{lem:cubic convergence}, we may also assume that $|t_1-t_2|<r^2/8$. 
Letting $q_i:=q_{X_i}^{(k)}$ and  letting $C>0$ be a constant that may increase at each step, we have
\begin{multline}\label{contfin2111}
   \|-\delta t+q_2(\cdot+\delta X)-q_2\|\leq C|\delta t|+  C\|q_2(\cdot+\delta X)-q_2\|_{L^{\infty}(Q_1)}\\\leq Cr^2+ C\|(\grad q_2,\partial_t q_2)\|_{L^{\infty}(Q_2)} |\delta X| \leq Cr^2+ C\|q_2\| |\delta X|\leq Cr.
\end{multline}
Applying \eqref{eq:taylor refble}, we have
\begin{equation} \label{eq:kapdlo1wqq}
 \|w(X_i+\cdot)+t-q_i\|_{L^{\infty}(B_r \times [-r^2,-r^2/2])}\leq Cr^{k+1}, \quad i\in \{1,2\}.   
\end{equation}
When $i=2$, \eqref{eq:kapdlo1wqq} may be equivalently written as
\begin{equation} \label{eq:kapdlo1wqq2}
 \|w(X_1+\cdot)+t+\delta t-q_2(\cdot+\\\delta X)\|_{L^{\infty}((B_r-\delta x) \times [-r^2-\delta t,-r^2/2-\delta t])}\leq Cr^{k+1}.   
\end{equation}
Hence, recalling that $|\delta x|=r/4$ and $|\delta t| <r^2/8$, we may combine the inequalities \eqref{eq:kapdlo1wqq} for $i=1$ and \eqref{eq:kapdlo1wqq2}, in their common domains, to infer that
\begin{equation} \label{eq:kapdlo1wqq3}
\|q_1-(q_2(\cdot +\delta X)-\delta t)\|_{L^{\infty}(B_{r/2}\times [-7r^2/8,-5r^2/8] )} \leq Cr^{k+1}.
\end{equation}
 By comparability of norms in finite-dimensional spaces (applied to the $L^{\infty}$ norms in the rescaled domains $B_{1/2}\times [-7/8,-5/8]$ and $Q_1$), up to increasing the constant $C$, this proves \eqref{eq: 2nd blow up variation}. Since $q_1$ and $q_2$ are polynomials of parabolic degree at most $k$,  \eqref{eq:kapdlo1wqq3} forces
\begin{equation} \label{contfin2112}
 \|q_1-(q_2(\cdot +\delta X)-\delta t)\|\leq Cr.   
\end{equation}
The claim then follows from \eqref{contfin2111} and \eqref{contfin2112}.  
\end{proof}

The basic properties of parabolic frequency, combined with the superquadratic cleaning property of Lemma \ref{lem:two-sided cleaning}, imply that any two points $(x_1,t_1), (x_2,t_2)\in \Sigma_{d}^{\geq k}$ satisfy the $k^{\text{th}}$-order cleaning estimate $|t_1-t_2|=O(|x_1-x_2|^k)$. We next show that this bound can be upgraded by one degree, which will be essential to obtain the parabolic dimension drop.
\begin{prop}[Refined cleaning of $\Sigma_d^k$] \label{prop: accelerated decay k>2} Let $k\geq 3$, and let $(x_1,t_1) \in \Sigma_d^{\geq k} \cap Q_1$.  Then there exists a constant $C=C(\|w\|_{L^{\infty}}, d, k)$ such that, for every $(x_2,t_2) \in \Sigma_{d}^{k}\cap Q_1$,
\begin{equation*}
    |t_2-t_1|\leq C|x_2-x_1|^{k+1}.
\end{equation*}

\end{prop}
\begin{proof}
Let $X_1=(x_1,t_1), \;X_2=(x_2,t_2)$, and let $q_1:=q_{X_1}^{(k)}$ and $q_2:=q_{X_2}^{(k)}$ be as in Proposition \ref{prop:taylor}. By Proposition \ref{prop:second blow-up char taylor}, $q_{2}$ must be  a homogeneous polynomial of degree $k$, and $q_1$ is either of degree $\geq k$ or identically zero. Therefore, by \eqref{eq: 2nd blow up lip}, we have, for $r=4|x_2-x_1|$,
\begin{equation} \label{eq:acc decay pf1}
\|q_1-q_2\|_{L^{\infty}(Q_r)}\leq Cr^{k+1}.   
\end{equation}
On the other hand, we see from \eqref{eq: 2nd blow up variation} that
\begin{equation}\label{eq:acc decay pf2}
    \|q_1-(q_2(\cdot +\delta X)-\delta t)\|_{L^{\infty}(Q_r)} \leq Cr^{k+1}.
\end{equation}
We deduce from \eqref{eq:acc decay pf1} and \eqref{eq:acc decay pf2} that
\begin{equation*}
         \|q_2-(q_2(\cdot +\delta X)-\delta t)\|_{L^{\infty}(Q_r)} \leq Cr^{k+1}.
\end{equation*}
Evaluating the above inequality at $X:=(0,0)$ and at $X:=(\delta x, \delta t)$, we obtain, respectively,
\begin{equation} \label{eq2delt2blo1}
    |-q_2(\delta X)+\delta t|\leq Cr^{k+1},
    \end{equation}
and    
\begin{equation} \label{eq2delt2blo}|q_2(2\delta X)-q_2(\delta X)+\delta t|\leq Cr^{k+1}.
\end{equation}
On the other hand, we have
\begin{equation} \label{eq:eq:accdecaypf23e1}
q_2(2\delta X)=q_2(2\delta x, 4\delta t)+(q_2(2\delta x,2\delta t)-q_2(2\delta x,4\delta t))=    q_2(2\delta x, 4\delta t)+O(|\delta t|r^{k-2}),
\end{equation}
and, by Lemma \ref{lem:cubic convergence} and Lemma \ref{lem:two-sided cleaning}, we infer that
\begin{equation*}
    |\delta t|=|t_1-t_2|\leq O(r^{3}).
\end{equation*}
Thus, by \eqref{eq:eq:accdecaypf23e1} and the parabolic homogeneity of $q_2$, we have
\begin{equation*}
 q_2(2\delta X)=2^{k}q_2(\delta X)+O(r^{k+1}),  
\end{equation*}
and, therefore, by \eqref{eq2delt2blo},
\begin{equation} \label{eq2delt2blo3}
  |(2^{k}-1)q_2(\delta X)+\delta t|\leq Cr^{k+1}.  
\end{equation}
The result now follows from \eqref{eq2delt2blo1}, \eqref{eq2delt2blo3}, and the triangle inequality.
\end{proof} 

We will make use of the following inequality for harmonic polynomials \cite[Thm. 3.1]{BadEngTor}. 
\begin{lem}[Sharp \L ojasiewicz inequality for harmonic polynomials]\label{lem:loja sharp}There exists constant $c=c(d,k)$ such that, for every nonconstant harmonic polynomial $h:\R^d\to \R$ of degree $\leq k$ and every $x_0\in \{h=0\}$, we have
\begin{equation*}
    |h(x)|\geq c\|h\|_{L^{\infty}(B_1(x_0))}\operatorname{dist}(x,\{h=0\})^k, \quad x\in B_{1/2}(x_0).
\end{equation*}    
\end{lem}
In degree two, we will also require the following elementary variant for the non-harmonic case.
\begin{lem}[\L ojasiewicz inequality for quadratic forms] \label{lem:loja 2} Let $f(x)= Ax\cdot x$, where $A$ is a non-zero, symmetric $d\times d$ matrix, and let \begin{equation} \label{min |eig|}
    \alpha=\min\{|\lambda|: \lambda\;\; \emph{is a non-zero eigenvalue of} \;\;A \}.
\end{equation}
Then, for every $x\in \R^d$,
\begin{equation*} 
    \operatorname{dist}(x,\{f=0\})\leq \frac{1}{\sqrt\alpha}|f(x)|^{\frac12}.
\end{equation*}
\end{lem}
\begin{proof} Up to a rotation, we may assume that $A$ is a diagonal matrix, so that $f(x)=\sum_{i=1}^{d}\lambda_ix_i^2$. Note first that, for all $x\in \R^d$,
\begin{equation*}
   |\grad f(x)|^2=4|Ax|^2=4(A^2x) \cdot x=\sum_{i=1}^d\lambda_i^2x_i^2\geq 4\alpha \sum_{i=1}^d|\lambda_i|x_i^2\geq 4\alpha \left|\sum_{i=1}^d\lambda_ix_i^2 \right|=4\alpha|f(x)|.
\end{equation*}
Fix $x_0 \neq 0$, and assume with no loss of generality that $f(x_0)>0$.   Consider the gradient flow
\begin{equation*}
    \dot \gamma(t)=-\grad f(\gamma(t)), \quad t\geq 0, \quad \gamma(0)=x_0.
\end{equation*}
Then
\begin{equation*}
   \frac{d}{dt}f(\gamma(t))=-|\grad f(\gamma(t)) |^2,
\end{equation*}
and there exists $T\in (0,\infty]$ such that
\begin{equation*}
    \lim_{t\to T}f(\gamma(t))=0.
\end{equation*}
We therefore have
\begin{equation*}
\text{dist}(x_0,\{f=0\})\leq \text{length}(\gamma(0,T)) \leq \int_{0}^{T} \frac{|\grad f(\gamma(t))|^2}{|\grad f(\gamma(t))|}dt\leq \frac{1}{\sqrt{\alpha}}\int_0^{T} -\frac{df/dt}{2f^{\frac12}}=\frac{f(x_0)^{\frac12}}{\sqrt{\alpha}}.
\end{equation*}
\end{proof}
We are now ready to prove the main result of this section.
\begin{prop} \label{prop: -t dim bound all} Let $w:Q_1\to \R$ be a bounded solution to \eqref{eq:obstacle intro}. Then
        \begin{equation*}
        \dim_{\operatorname{par}}(\Sigma_{d} \backslash\Sigma_{d}^{\infty})\leq d-1.
    \end{equation*}
\end{prop}
\begin{proof}
By Proposition \ref{prop:-t dimension bound quadratic}, Lemma \ref{lem: FRS cleaning}, and the same arguments as Proposition \ref{prop: Hausdorff time dependent}, it suffices to show that 
\begin{equation*}\operatorname{dim}_{\mathcal{H}}(\pi_x(\Sigma_{d} \backslash(\Sigma_{d}^2 \cup \Sigma_{d}^{\infty})))\leq d-1.\end{equation*}
Assume that $X_0=(x_0,t_0)\in \Sigma_d^k$ for some $k\geq 3$. Then, by Proposition \ref{prop:second blow-up char taylor}, $q_k:=q_{X_0}^{(k)}$ is a homogeneous, caloric polynomial of degree $k$. Therefore, there exists a non-zero homogeneous polynomial $h(x)$ of degree $k$ and an integer $j=j(x_0) \in \{0,1,\ldots,\lfloor k/2 \rfloor\}$ such that 
\begin{equation} \label{12eolpk defi}
 q_k(x,t)=\sum_{i=0}^{j} \frac{t^i }{i!} \Delta^{i}h(x),\quad (x,t)\in \R^d\times \R, \quad   \Delta^jh\not \equiv 0, \quad \Delta^{j+1}h(x)\equiv0.
\end{equation}
For each pair of integers $(k,j)$ such that $k\geq 3$ and $0\leq j \leq \lfloor k/2 \rfloor$, let 
\begin{equation*}
  S_{k,j}=\{(x_0,t_0) \in \Sigma_{d}^k:  j(x_0)=j\}. 
\end{equation*}
The problem is now reduced to proving that
\begin{equation} \label{1dskjdimest}
    \dim_{\mathcal{H}}(\pi_x(S_{k,j}))\leq d-1, \quad k\geq 3, \quad 0\leq j \leq \lfloor k/2 \rfloor.
\end{equation}
Assume that, for some small $r_0>0$, $X_1:=(x_1,t_1) \in S_{k,j}\cap (B_{r_0}(x_0)\times (-1,1))$, and let $\tilde{q}_k:=q_{X_1}^{(k)}$. Letting $r=|x_1-x_0|$, by Proposition \ref{prop: accelerated decay k>2}, we must have $(t_1-t_0)\in  (-Cr^{k+1},Cr^{k+1})$ for some constant $C>0$. Thus, we have from \eqref{eq: 2nd blow up variation} that
\begin{equation*}
 \|\tilde{q}_k-q_k(\cdot+(x_1-x_0),\cdot+(t_1-t_0))\|_{L^{\infty}(B_{2r}\times [-14r^2,-10r^2])}\leq Cr^{k+1}.   
\end{equation*}
and, in particular, 
\begin{equation} \label{2dsoakddimest2}
   \|\tilde{q}_k(0,\cdot)-q_k(x_1-x_0,\cdot + ( t_1-t_0))\|_{L^{\infty}( [-14r^2,-10r^2])} \leq Cr^{k+1}. 
\end{equation}
 Recalling \eqref{12eolpk defi}, depending on the parity of $k$, we must have
\begin{equation} \label{1e1pktil}
\tilde{q}_k(0,\cdot)\equiv 0 \quad\text{ or } \quad \tilde{q}_k(0,t)=\frac{t^{k/2}}{(k/2)!}    \Delta^{k/2}\tilde{q}_k(\cdot,0)=:\frac{t^{k/2}}{(k/2)!}\tilde{c}    .
\end{equation}
Assume first that $j<k/2$, and let $f(x):=\Delta^jh(x)$. Then the coefficient of  $t^j$ of $\tilde{q}_k(0,t)$ vanishes, and we have
\begin{equation} \label{eq:expanqklojpf}
    q_k(x_1-x_0,t+(t_1-t_0))=\sum_{i=0}^{j}\frac{1}{i!}(t+(t_1-t_0))^i\Delta^ih(x_1-x_0).
\end{equation}
 Thus, recovering the coefficient of $t^j$ in \eqref{2dsoakddimest2}, we get 
\begin{equation} \label{fprelojdq3}
    |f(x_1-x_0)|\leq Cr^{b+1}, \quad b:=k-2j=\deg(f)\geq1.
\end{equation}
 By the choice of $j$, $f$ is a  non-constant harmonic polynomial of degree $b$. The \L ojasiewicz inequality for harmonic polynomials (Lemma \ref{lem:loja sharp}) then yields
\begin{equation} \label{eq: higher order reifenberg}
\text{dist}(x_1-x_0,\{f=0\})\leq C|f(x_1-x_0)|^{\frac1b}\leq Cr^{\frac{b+1}{b}}, \quad x_1 \in S_{k,j}\cap B_r(x_0).  
\end{equation}
By Lemma \ref{lem: -t Taylor continuous dependence}, the coefficients of $f$ depend continuously on $x_0$, and therefore, the constant $C$ may be chosen locally uniformly in $(x_0,t_0)\in S_{k,j}$. Since $(b+1)/b>1$, \eqref{1dskjdimest} follows from standard arguments (for instance, it follows directly from \cite[Cor. 8.7]{BadLew} and \cite[M. Thm.]{wongkew}).

It remains to treat the case where $k$ is even and $j=k/2$. This time, we consider $f(x)=\Delta^{j-1}h(x)$. Since $h$ has degree $k$, we must have $f(x)=Ax\cdot x$ for some $d\times d$, non-zero matrix $A$. Since $j-1<k/2$, \eqref{2dsoakddimest2},  \eqref{1e1pktil}, and \eqref{eq:expanqklojpf} imply that
\begin{equation*}
    |f(x_1-x_0)+\Delta f(x_1-x_0) (t_1-t_0)|\leq C r^3.
\end{equation*}
Since $t_1-t_0=O(r^{k+1})=O(r^3)$, we infer that
\begin{equation*}
    |f(x_1-x_0)|\leq Cr^3.
\end{equation*}
Let $l=l(x_0)\in \{1,\ldots,d\}$ be the number of non-zero eigenvalues of $A$, and let $\alpha=\alpha(x_0)$ be given by \eqref{min |eig|}. Then, by Lemma \ref{lem:loja 2}, we have
\begin{equation*}
    \text{dist}(x_1-x_0,\{f=0\})\leq \frac{1}{\sqrt{\alpha}}|f(x_1-x_0)|^{\frac{1}{2}}\leq \frac{1}{\sqrt{\alpha}}Cr^{\frac32}.
\end{equation*}
Since the eigenvalues of $A$ depend continuously on $x_0$, the positive constant $\alpha(x_0)$ is locally uniformly bounded away from zero within the set
\begin{equation*}
    S_{k,k/2,l}=\{x_0\in S_{k,k/2}: l(x_0)=l\},
\end{equation*}
and we may conclude as before.
\end{proof} 

\section{Infinite order second blow-up in the top stratum: the set $\Sigma_d^\infty$}\label{sec:top stratum infinite}
In this section we complete the proof of Theorem \ref{thm:sigma} by showing that $\Sigma_d^\infty = \emptyset$ for global solutions (unless that solution is a translation of $(-t)^+$), and by showing that $\Sigma_d^\infty(w)\cap\{t= t_0\}$ is empty for most $t_0$ for local solutions. 

First we turn to the statement for global solutions. Our basic tool will be the (non-cutoff) monotonicity formulas. We recall the definitions of the height, energy, and frequency functions $$\begin{aligned} H(r,f) :=& \int_{t=-r^2} f^2(x)G(x,t)\, dx,\\
D(r, f) :=& 2r^2\int_{t=-r^2}|\nabla f|^2(x) G(x,t)\, dx,\\
\phi(r,f):=& \frac{D(r,f)}{H(r,f)}.\end{aligned}$$ 

We show below that for global solutions these quantities are monotone when applied to the difference between $w$ and its first blowup. Recall from \eqref{eq:good sign monotonicity formulas}  that if $u = w(x,t) - (-t)^+$, then $u\mathbf Hu = (-t)^+\chi_{\{w = 0\}} > 0$. Given that this term has a sign, it is not surprising that $\phi$ is monotone (compare with \cite{figalli} for similar computations with a cutoff and \cite{AEK} where the sign of a related error term expression is leveraged to get monotonicity in the context of elliptic boundary unique continuation). 

\begin{prop}\label{p:almgrenmonotonicity}
    Let $w$ solve \eqref{eq:obstacle intro} in $\mathbb R^d\times (-1, 0]$ with $(0,0) \in \Sigma_d$. Let $u(x,t):= w(x,t) - (-t)^+$ and assume that $u(x,t_0) \not\equiv 0$ for any $t_0\in (-1,0]$. Then \begin{equation}\label{e:monotone}\begin{aligned}
       \frac{2\phi(r,u) + 4}{r} \geq  \frac{d}{dr}\log(H(r,u)) &\geq \frac{2\phi(r,u)}{r}, \qquad \forall r < 1/2\\
\phi(r,u) \geq \phi(s,u) &\geq 2, \qquad \forall 0 < s < r < 1/2.
        \end{aligned}
    \end{equation}
\end{prop}

\begin{proof}
   We compute, using \eqref{eq:H' comp} and \eqref{eq:good sign monotonicity formulas},
   \begin{equation}\label{e:derivh}
    \frac{d}{dr} H(r,u) = \frac{2}{r}(D(r,u)+2r^2\langle u,\cH u\rangle_r)\geq \frac{2}{r}D(r,u).
\end{equation}  
On the other hand, using \eqref{eq:good sign monotonicity formulas} we obtain
\begin{equation*} 4r\langle u,\cH u\rangle _r=4r\langle -u\chi_{\{w=0\}},1\rangle_r=4r^{-1}\langle u^2\chi_{\{w=0\}},1\rangle_r\leq 4r^{-1}H(r,u),    
\end{equation*}
which, dividing by $H(r,u)$ in \eqref{e:derivh}, readily implies the first line in \eqref{e:monotone}.

From \eqref{eq:H' comp} we have
    $$\frac{d}{dr}H(r,u) = 2r^{-1}\langle Zu,u\rangle_r.$$
We also have from \eqref{eq:D' comp} and \eqref{eq:good sign monotonicity formulas} that
\begin{equation}
    \frac{d}{dr}D(r,u)=\frac{1}{r}(2\langle Zu,Zu \rangle_r-4r^2\langle Zu,\cH u\rangle_{r})=\frac{2}{r}(\langle Zu,Zu \rangle_r-4r^2\langle u,\cH u\rangle_{r}).
\end{equation}
Combining the above equations,
\begin{multline} \label{e:derivativeofalmgren}
 H(r,u)^2\frac{d}{dr}   \phi(r,u)=2r^{-1}(\langle Zu, Zu\rangle_r \langle u,u\rangle_r -\langle Zu, u\rangle_r ^2)-8r\langle u, \cH u\rangle_r \langle u,u\rangle_r+4r\langle u,\cH u \rangle_r\langle Zu,u\rangle_r \\
 \geq 4r\langle u, \cH u \rangle_r (\langle Zu,u\rangle _r-2H(r,u))=4r\langle u,\cH u\rangle _r(D(r,u)+2r^2\langle u,\cH u\rangle _r-2H(r,u))\\\geq 4r\langle u,\cH u\rangle _r(D(r,u)-2H(r,u)).
\end{multline}

On the other hand, combining the above equations and using that $Zu = 2u$ in $\{w =0\}$, we get
\begin{equation}
    \frac{d}{dr}r^{-4}\left(D(r,u) - 2H(r,u)\right)=2r^{-5}\langle (Zu-2u,Zu-2u\rangle _r+2r^2 \langle 2 u-Zu,\cH u \rangle_r)=2r^{-5}\langle Zu,Zu\rangle_r\geq0.
\end{equation}

Finally we claim that $$\lim_{r\downarrow 0} r^{-4}(D(r,u) - 2H(r,u))= 0.$$ If this is true then we have that $D(r,u) \geq 2H(r,u)$ for all $0 < r < 1/2$, which, in light of \eqref{e:derivativeofalmgren} implies that $\frac{d}{dr}\phi(r,u) \geq 0$ and that $\phi(r,u) \geq 2$ for all $r < 1/2$. So we are done if we have proven the claim.

The claim follows by simple blow-up analysis. We see that $r^{-4}D(r,u) = r^{-4}D(r,w) = D(1, w_r)$ where $w_r = r^{-2}w(rx, r^2t)$. Note that $w_r(\cdot, -1)\in C^{1,1}_x$ (since $w_r$ also solves the obstacle problem \eqref{eq:obstacle intro}), so $w_r(\cdot, -1) \rightarrow 1$ in $C^1_{\operatorname{loc},x}$. It follows that $r^{-4}D(r,u) \rightarrow 0$ as $r\uparrow 0$. Similarly, by homogeneity we have that $r^{-4}H(r,u) = H(1, w_r - (-t)^+) \rightarrow H(1, (-t)^+-(-t)^+)$ by the same convergence as above. This finishes the proof of the claim and thus the lemma. 
\end{proof}

We now show the last ingredient needed for Theorem \ref{thm:sigma}.

\begin{prop}\label{prop:no jumps global}
    Let $w$ solve \eqref{eq:obstacle intro} in $\mathbb R^d\times (-1, 0]$ and assume that $(0,0) \in \Sigma_d^\infty(w)$. Then $w = (-t)^+$ in $\mathbb R^d \times (-1, 0]$. 
\end{prop}

\begin{proof}
    First we assume that $w\not\equiv (-t)^+$. By forward uniqueness of the obstacle problem and backwards uniqueness of bounded solutions to the heat equation in $\mathbb R^d$, this implies that $w(x,t_0) \neq (-t_0)^+$ for any $t_0 \in (-1, 0]$. In particular, if $u = w-(-t)^+$, then $H(r,u) \neq 0$ for any $r\in [0,1)$. 

    We apply Proposition \ref{p:almgrenmonotonicity} and get that $$\infty > \Lambda \equiv \phi(1/2, u) \geq \phi(0,u).$$  Using \eqref{e:monotone} we have, $H(r,u)/H(1/2,u) \geq r^{2\Lambda + C}.$ In other words, there exists some constant $\tilde{C} > 0$ (which depends on $w$ but is non-zero) such that $$H(r,u) \geq \tilde{C}r^{2\Lambda+C}, \qquad \forall r < 1/4.$$

    On the other hand, recalling the definition of $\Sigma_d^k$ (i.e. \eqref{e:defsigmak}) and the estimate \eqref{eq:enoijultlessim} 
    since $(0,0) \in \Sigma_d^k$ for any $k$ we know that for any $k\in \mathbb N$ there exists an $r_k$ such that if $r < r_k$, then $$\|u\|_{L^\infty(Q^-_r)} \leq r^k.$$ Let $k \geq \Lambda + C/2$ and we get our desired contradiction. 
\end{proof}

We are now ready to complete the proof of Theorem \ref{thm:sigma}.

\begin{proof}[Proof of Theorem \ref{thm:sigma}] By Propositions \ref{prop: -t dim bound all} and \ref{prop:no jumps global}, \eqref{eq:dimpar finite freq} and \eqref{dichotomy} hold, respectively. It remains to show the properties of $\Sigma_d^{\infty}$ for local solutions. By a rescaling and partition of unity argument, it is  enough to show that if $w:Q_2 \to \R$ is a bounded solution to \eqref{eq:obstacle intro}, then $\Sigma_d^{\infty} \cap Q_1 \subset \{x\in B_1: t=h(x)\}$ for some $h\in C^{\infty}(B_1)$, and that
\begin{equation} \label{eq:hauspflst}
    \dim_{\mathcal{H}}(\{t\in (-1,1): (x,t)\in \Sigma_d^{\infty}\cap Q_1 \text{ for some } x\in B_1\})=0.
\end{equation}
By Lemmas \ref{lem energy} and \ref{lem:comparability sigma d}, for every $k\geq3$, there exists $C_k>0$ such that for every $(x_0,t_0) \in \Sigma_{d}^{\infty}\cap Q_1$,
\begin{equation}
    \|w-(t_0-t)\|_{L^{\infty}(Q_r^-(x_0,t_0))}\leq C_{k}r^{k}, \quad r\in (0,1).
\end{equation}
Thus, by Lemma \ref{lem:two-sided cleaning}, up to increasing the value of $C_k$,
\begin{equation} \label{eq:prewhit}
 |t_1-t_0|\leq C_{k}|x_1-x_0|^k , \quad (x_0,t_0), (x_1,t_1) \in \Sigma_{d}^{\infty}\cap \overline{Q_{1}}.
\end{equation}
By \cite[Prop. 7.7]{FRS20}, this implies \eqref{eq:hauspflst}.

On the other hand, one readily sees from Lemma \ref{lem:cubic convergence} that $\Sigma_d^{\geq3}\cap \overline{Q_1}$ is a closed set. Thus, by Proposition \ref{prop:second blow-up char taylor} and Lemma \ref{lem: -t Taylor continuous dependence}, $\Sigma_d^{\infty}\cap \overline{Q_1}$ is closed. By \eqref{eq:prewhit} and Whitney's extension theorem, it follows that $\Sigma_d^{\infty}\cap Q_1 \subset \{t=h(x)\}$, where $h\in C^{\infty}(B_1)$, and now we can conclude.
\end{proof}

\section{Examples of non-trivial $\Sigma^\infty_d$ and $\Sigma^2_d$ }\label{sec:global}

In this section we collect some interesting  examples about the ``anomalous'' points in the top stratum, $\Sigma_d^2$ and $\Sigma_d^\infty$. \subsection{Example of a local solution with $\Sigma_d^{\infty} \neq \emptyset$. }Our first example shows that if $w$ is a solution in a parabolic cylinder (as opposed to globally defined) then we can have $\Sigma_d^{\infty}(w) \neq \emptyset$ even when $w\not\equiv (-t)^+$. This shows the necessity of our assumption that $w$ be a global solution in the last part of Theorem \ref{thm:sigma}.

\begin{exa}\label{ex:tychonoff}Let
\begin{equation}
    g(t)=\begin{cases}
        e^{-1/t^2} & t< 0\\
        0 & t\geq 0
    \end{cases}
\end{equation}
and consider the Tychonoff function
\begin{equation}
    \varphi(x,t)= \sum_{k=0}^{\infty}g^{(k)}(t)\frac{x_1^{2k}}{(2k)!}.
\end{equation}
Then $\varphi$ is a smooth caloric function in $\R^d \times \R$, which vanishes for $t\geq 0$. Let $\vep>0$, and let
\begin{equation}
    w(x,t)=(-t)^++\vep \varphi(x,t), \quad (x,t)\in Q_1.
\end{equation}
Then $w$ solves
\begin{equation}
    w_t-\Delta w=-\chi_{\{t<0\}}, \quad (x,t)\in Q_1.
\end{equation}
On the other hand, since $\varphi(\cdot,0)= \varphi_t(\cdot,0)\equiv0$, if we choose $\vep$ sufficiently small, then $w=-t+\vep \varphi>0$ and $w_t=-1+\vep \varphi_t<0$ in $Q_1\cap \{t<0\}$. Thus $\{w>0\}=\{w_t<0\}$, $w$ solves \eqref{eq:obstacle intro} in $Q_1$, $w \not \equiv (-t)^+$, and yet $w(x,0) \equiv 0$, so $B_1 \times \{0\} \subset \Sigma_d^{\infty}$.
    
\end{exa}

\subsection{Examples with $\Sigma_d^2\neq \emptyset$} In this subsection we construct global solutions $w$ such that $\Sigma_d^2(w) \neq \emptyset$. That is we construct points with second blow-up that converges slower than any power.

We begin by showing that global solutions must vanish in finite time (cf. Lemma \ref{lem:two-sided cleaning}).
\begin{prop} Let $w:\R^d \times [0,\infty)\to \R$ be a bounded solution to \eqref{eq:obstacle intro}. The solution $w$ vanishes in finite time, and the extinction time $T:=\sup\{t: \|w(\cdot,t)\|_{\infty} >0\}$. satisfies
\begin{equation}
    T\leq c_d^{-1}\|w(\cdot,0)\|_{L^{\infty}(\R^d)},
\end{equation}
where $c_d$ is the constant of \eqref{eq:w nondegeneracy}.    
\end{prop}
\begin{proof}
Let $x_0 \in \mathbb R^d$ and $t_0 >0$ be such that $(x_0,t_0) \in \overline{\{w>0\}}$. By \eqref{eq:w nondegeneracy} applied at scale $r=\sqrt{t_0}$ and centered at $(x_0,t_0)$, there exists a dimensional constant $c_d$ such that $$t_0c_d \leq \|w\|_{L^\infty}\leq \|w(\cdot,0)\|_{\infty}.$$ 
Thus, $w(\cdot,t_0)\equiv0$ for $t_0\geq \|w(\cdot,0)\|_{\infty}c_d^{-1}.$
\end{proof}

Our second lemma will help us to rule out non-quadratic second blow-ups in the radial setting.

\begin{lem} \label{lem: laguerre} Let $k\geq 2$, let $r=|x|$, and let $p$ be the unique homogeneous, radially symmetric, caloric polynomial of degree $2k$ such that $\frac{\partial^{2k} p}{\partial r^{2k}}=1$. Then there exist positive constants $c_1$ and $c_2$, depending only on $k$ and $d$, such that if $c_1t=-r^2$, then $\partial_r p(x,t)\leq -c_2r^{2k-1}$. 
\end{lem}
\begin{proof}
Setting $s=r^2/(-t)$, we have, by homogeneity and radial symmetry,
\begin{equation}
    p(x,t)=(-t)^kP(s),
\end{equation}
where $P$ is a polynomial of degree $k$ in one variable (which depends only on $k, d$). Computing the radial derivative, we have
\begin{equation}
    \partial_r p(x,t)=2r^{-1}(-t)^kP'(s)s.
\end{equation}
It is then enough to show that $-c_2 := P'(c_1)<0$ for some $c_1>0$. Since $p$ is caloric, it follows that $P$ satisfies
\begin{equation}
    sP''(s)+(d-s/2)P'(s)+kP(s)=0.
\end{equation}
Differentiating this equation, we have that $F=P'$ satisfies
\begin{equation}
    sF''(s)+(d+1-s/2)F'(s)+(k-1/2)F(s)=0,
\end{equation}
or, equivalently,
\begin{equation}
    (s^{d+1}e^{-s/2}F'(s))'=-(k-1/2)s^{d}e^{-s/2}F(s), \quad s\in (0,\infty).
\end{equation}
The map $W(s)=s^{d+1}e^{-s/2}F'(s)$ satisfies $W(0)=W(+\infty)=0$, with $W \not \equiv 0$ in $(0,\infty)$. Therefore, $W'$ must change sign in $(0,\infty)$. Since $W'(s)=-(k-1/2)s^de^{-s/2}P'(s)$, it follows that $P'(c_1)<0$ for some $c_1\in(0,\infty),$ as wanted.
\end{proof}

Finally, we show that, for radial initial data, the global solutions constructed in \cite[Thm. 9.3]{KK} exhibit, at their extinction time, a singular point with first blow up $(-t)^+$ and quadratic second blow-up. We remark that, by a simple maximum principle argument, these solutions can be shown to be radially decreasing, which yields an entire family of radial solutions for which the following result applies.

\begin{prop}\label{p:sigma2example} Let $w=w(r,t)$ be a radial solution to \eqref{eq:obstacle intro} with $\partial_r w \leq 0$, and assume that $t=0$ is the extinction time. Then $(0,0)\in \Sigma_d^2(w)$, i.e. the first blowup at $(0,0)$ is $-t$ and the second blow-up has degree $2$.
\end{prop}

\begin{proof}
On each time slice (in the positivity set) $\Delta w = 1+\partial_t w$. If $\partial_t w > - 1$ on the time slice (in the positivity set) then $\Delta w > 0$ which is a contradiction to the fact that $w$ achieves its maximum at $0$. Thus on each time slice (before the vanishing time) $\sup -w_t > 1$ which means the blow-up must be $-t$ at $(0,0)$. 

Assume, by contradiction, that $w$ has a second blow-up of degree $m\geq3$, and let $q$ be the homogeneous polynomial of degree $m$ given by Lemma \ref{lem poly d}. Since $w$ is radial, $q$ must be radial, and, in particular, $m=2k$ for some $k\geq 2.$ Since $w$ is radially decreasing, we must have, for each $t<0$, $\max w(\cdot,t)=w(0,t)$ and, in particular, $\Delta w(0,t)\leq 0$. By Lemma \ref{lem poly d} (and interior estimates for the heat equation), there exists an $\alpha > 0$ such that
\begin{equation}
    0\geq \Delta w(0,-r^2)=\Delta q(0,-r^2)+O(|r|^{2k-2+\alpha}), \quad r_0>r>0.
\end{equation}
Since $\Delta q (x,t)$ has parabolic degree $2k-2$, the above implies that  $\Delta q(0,t) \leq 0$. As such we infer that there exists $a>0$ such that $\frac{\partial^{2k}q}{\partial r^{2k}} = -a$. But then, letting $c_1$ and $c_2$ be the constants of Lemma \ref{lem: laguerre}, we have once more by Lemma \ref{lem poly d}, for small $r>0$,
\begin{equation}
 0 \geq \partial_r w(r,-c_1r^2) =q_r(r,-c_1r^2)+O(r^{2k-1+\alpha})\geq ac_2 r^{2k-1}+O(r^{2k-1+\alpha}).
\end{equation}
This yields a contradiction for $r$ sufficiently small.
    
\end{proof}

The existence of these radial solutions also allows us to build solutions for which $\pi_x(\Sigma_d^2)$ is locally infinite and for which $\pi_t(\Sigma_d^2)$ has an accumulation point at the initial time. 

\begin{exa}\label{ex:lotsofsing}
    We construct a global solution $w: \mathbb R^d\times \mathbb R_{>0}$ to \eqref{eq:obstacle intro} such that $\{w(x, 0) > 0\} \subset B_3$ and $\pi_x(\Sigma_d^2(w))\cap B_2(0)$ has infinite cardinality. Let $u(x,t)$ be a radial solution guaranteed by Proposition \ref{p:sigma2example}. There is a time $t_0 < 0$ (which depends on $u$) such that $\{x\mid u(x, t_0) > 0\} = B_1(0)$. If $e\in \mathbb S^{d-1}$ let $p_n = \sum_{k=0}^n 2^{-k}e$ so that $B_{2^{-n}}(p_n)\cap B_{2^{-m}}(p_m) = \emptyset$ when $m\neq n$ and $B_{2^{-n}}(p_n) \subset B_3$

    Define $$w(x,t) := \sum^\infty_{n=1}2^{-2n}u(2^{n}(x - p_n), 2^{2n}(t -t_0)).$$ We claim that this is the desired solution. Indeed each $2^{-2n}u(2^{n}(x - p_n), 2^{2n}t -t_0)$ solves the obstacle problem and is supported in $B_{2^{-n}}(p_n)\times (0, -2^{-2n}t_0)$. Since the supports are disjoint (by construction) the sum is a solution to \eqref{eq:obstacle intro}. Furthermore, since the supports are disjoint, Proposition \ref{p:sigma2example} shows that $\bigcup_{n=1}^\infty \{(p_n, -2^{-2n}t_0)\} = \Sigma_d^2(w)$. The claim follows. 
\end{exa}


\subsection*{Acknowledgments} ME was partially supported by NSF DMS CAREER 2143719 and by a grant from the Simons Foundation (Grant Award ID BD-Targeted-00017375-ME). IK was partially supported by NSF DMS 2452649. Part of this work was completed during IK's visit to KIAS, and she thanks their hospitality.

\end{document}